\documentclass{tran-l}
\usepackage[square,sort,comma,numbers]{natbib}
\usepackage[T1]{fontenc}
\usepackage[utf8]{inputenc}
\setcitestyle{square}
\usepackage{amsmath,amssymb,amsfonts,amsthm}
\usepackage{mathtools}
\usepackage{orcidlink}
\usepackage{mathrsfs}
\usepackage[normalem]{ulem}
\usepackage{chemfig}
\usepackage{chemformula}
\usepackage{centernot}
\usepackage{booktabs}
\usepackage{enumitem}  
\usepackage{longtable}
\usepackage{url}
\usepackage{xcolor}
\definecolor{navy}{RGB}{0,0,128}
\definecolor{royalblue}{RGB}{65,105,225}
\usepackage{hyperref}
\hypersetup{
    linktocpage,
    colorlinks,
    citecolor=royalblue,
    filecolor=black,
    linkcolor=royalblue,
    urlcolor=black
}
\usepackage[left=3cm, right=3cm]{geometry}
\usepackage{braket}
\usepackage{physics}
\usepackage{amscd}
\usepackage{tikz-cd} 
\usepackage{tikz-3dplot}

\usepackage{algorithm}
\usepackage{algpseudocode}
\usetikzlibrary{arrows.meta,positioning,shapes.geometric,fit,backgrounds,calc}
\usepackage{adjustbox}
\definecolor{accent}{RGB}{0,100,148} 
\usetikzlibrary{shapes.arrows}
\usetikzlibrary{positioning}
\usepackage{calc}
\usepackage{float}
\tikzset{%
  symbol/.style={
    draw=none,
    every to/.append style={
      edge node={node [sloped, allow upside down, auto=false]{$#1$}}
    },
  },
}

\usetikzlibrary{decorations.pathmorphing}

\let\LATEXth\th

\let\th\LATEXth

\usepackage{caption}
\newtheorem{theorem}{Theorem}[section]
\newtheorem{lemma}[theorem]{Lemma}

\newtheorem*{question}{Question}
\newtheorem{proposition}[theorem]{Proposition}
\newtheorem{corollary}[theorem]{Corollary}
\newtheorem*{summary*}{Summary}

\theoremstyle{definition}
\newtheorem{definition}[theorem]{Definition}
\newtheorem{example}[theorem]{Example}

\theoremstyle{remark}
\newtheorem{remark}[theorem]{Remark}

\numberwithin{equation}{section}

\DeclareMathOperator{\End}{End} 

\DeclareMathOperator{\Hom}{Hom}
\DeclareMathOperator{\Sing}{Sing}
\DeclareMathOperator{\Gr}{Gr}

\usepackage{amsmath,amssymb}

\providecommand{\cO}{\mathcal{O}}
\providecommand{\Gm}{\mathbb{G}_{\mathrm m}}
\providecommand{\PP}{\mathbb{P}}
\providecommand{\Spec}{\operatorname{Spec}}
\providecommand{\Proj}{\operatorname{Proj}}
\providecommand{\PP}{\mathbb{P}}
\providecommand{\Pic}{\operatorname{Pic}}
\providecommand{\GL}{\operatorname{GL}}
\providecommand{\PGL}{\operatorname{PGL}}
\providecommand{\SL}{\operatorname{SL}}

\providecommand{\Br}{\operatorname{Br}}

\providecommand{\per}{\operatorname{per}}

\providecommand{\End}{\operatorname{End}}
\providecommand{\Mat}{\operatorname{Mat}}

\providecommand{\codim}{\operatorname{codim}}
\providecommand{\Hilb}{\operatorname{Hilb}}

\newcommand{\Z}{\mathbb{Z}}

\newcommand{\id}{\mathrm{id}}

\newcommand{\ind}{\mathrm{ind}}

\DeclareMathOperator{\PU}{PU}

\usepackage[utf8]{inputenc}
\usepackage{bxcjkjatype}
\usepackage{multirow}

\definecolor{QGLBlue}{RGB}{65,105,225}   
\definecolor{QGLRed}{RGB}{203,65,84}     
\definecolor{QGLGreen}{RGB}{34,139,34}   

\definecolor{QGLItem}{RGB}{255,245,204} 

\usepackage[most]{tcolorbox}
\definecolor{QGLItem}{RGB}{65,105,225}
\newtcolorbox{HLblock}{enhanced,breakable,
  colback=QGLBlue!5, colframe=QGLBlue!5, 
  boxrule=0pt, arc=0pt, outer arc=0pt,
  left=2pt,right=2pt,top=1.5pt,bottom=1.4pt}

\begin{document}

\title[Quantum Entanglement Geometry on Severi--Brauer Schemes]%
      {Quantum Entanglement Geometry on Severi--Brauer Schemes:
       Subsystem Reductions of Azumaya Algebras}

\author{Kazuki Ikeda\,\orcidlink{0000-0003-3821-2669}}
\address{}
\curraddr{}

\email{kazuki.ikeda@umb.edu}
\address{Department of Physics, University of Massachusetts Boston, USA}
\address{Center For Nuclear Theory, Department of Physics and Astronomy, Stony Brook University, USA}

\thanks{The author expresses gratitude to Steven Rayan for his careful reading of the manuscript and for his invaluable comments. This work was partially supported by the NSF under Grant No. OSI-2328774.}

\date{}

\dedicatory{}
\begin{abstract}
Quantum entanglement is a basic resource of quantum theory, but its usual definition assumes a fixed decomposition into subsystems. We develop an algebro--geometric framework to address quantum entanglement. For the Severi--Brauer scheme \(SB(\mathcal A)\to X\) of an Azumaya algebra and a factorization type \(\mathbf d\), we prove that a global locus of product states exists exactly when the associated torsor under \(\PGL_n\) reduces to the stabilizer \(G_{\mathbf d}\) of the Segre variety. Thus entanglement in families is measured by a geometric obstruction to global subsystem structure.

Beyond this reduction criterion, we identify the cohomology that carries computable entangling obstructions. After determinant and incidence data \(L\subset R\subset K\) are chosen, the relevant classes lie in the mapping cone relative cohomology \(E^\bullet(BL\to BR)\). Their absolute images on \(X\), modulo classes induced from \(BK\), give determinant entangling obstructions beyond the Brauer class. We construct the moduli of subsystem structures as the quotient \(P/G_{\mathbf d}\) and realize it as a locally closed locus in the relative Hilbert scheme, with a compactification by degenerations of product state loci. In the bipartite case, a chosen subsystem structure globalizes the Schmidt rank stratification to a flat filtration with base change compatible resolutions and fiberwise constant numerical invariants. Examples show that reducibility can depend on the underlying torsor, not only on the Brauer class, with an interpretation by entangling monodromy.
\end{abstract}

\maketitle
\setcounter{tocdepth}{1} 
\tableofcontents

\section{Introduction}
\label{sec:introduction}

\subsection*{Entanglement as geometry}
For pure states, entanglement is defined after choosing a tensor product decomposition. Given a factorization
\[
\mathcal H \cong \mathcal H_1 \otimes \cdots \otimes \mathcal H_s, \qquad \dim(\mathcal H_i)=d_i,\ \ \prod_i d_i=n,
\]
the product states form the Segre variety
\[
\Sigma_{\mathbf d}:=\PP^{d_1-1}\times\cdots\times \PP^{d_s-1}\hookrightarrow \PP^{n-1}.
\]
This is the locus of product states of type $\mathbf d=(d_1,\dots,d_s)$. The entangled pure states are the points outside this Segre variety. 

The starting point of this paper is that the usual definition of entanglement presupposes the
\emph{existence of a global subsystem structure}.
In many geometric and physical situations one encounters families of pure-state spaces over a base $X$, where the fibers are projective spaces only locally and may be globally twisted. In such a family, one can often choose a tensor factorization on a simply connected chart, but it is not
automatic that these local subsystem choices transport and glue globally. When they do not, there is no globally defined product-state locus, hence no globally meaningful notion of ``entangled vs.\ product'' states for the family.

\subsection*{Known facts: Severi--Brauer families as twisted pure-state spaces}
Let $k$ be an algebraically closed field of characteristic $0$, and let $X$ be a smooth quasi-projective $k$--scheme. A degree-$n$ Azumaya algebra $\mathcal A$ determines a principal $\PGL_n$--torsor $P\to X$ and the associated Severi--Brauer scheme
\[
SB(\mathcal A)=P\times^{\PGL_n}\PP^{n-1}\to X,
\]
which is fppf-locally a trivial $\PP^{n-1}$--bundle. From the viewpoint of quantum information, $SB(\mathcal A)\to X$ is an algebro--geometric model for a \emph{family of pure-state spaces} whose projective trivializations may carry nontrivial twisting.

Grothendieck's theory of Azumaya algebras and the Brauer group provides the natural language for such families
\cite{Grothendieck1968BrauerI,Grothendieck1968BrauerII,Grothendieck1968BrauerIII} (see also \cite{Azumaya1951MaximallyCentral,AuslanderGoldman1960Brauer}).
These tools describe the twisting of the projective state bundle. They do not, however, choose a Segre locus inside $SB(\mathcal A)$.

\subsection*{The remaining problems}
Fix a factorization type $\mathbf d=(d_1,\dots,d_s)$ with $n=\prod_i d_i$, let $\Sigma_{\mathbf d}\subset \mathbb{P}^{n-1}$ be the Segre variety, and denote by $G_{\mathbf d}\subset \mathrm{PGL}_n$ its stabilizer (Definition \ref{def:segrestab}).

\begin{question}
We address the following questions.
\begin{enumerate}
\item How should quantum entanglement be defined for a Severi--Brauer fibration $\pi:SB(\mathcal A)\to X$?
\item When does $SB(\mathcal A)$ contain a relative Segre subscheme of type $\mathbf d$? Equivalently, when does the associated torsor under $\mathrm{PGL}_n$ admit a reduction to $G_{\mathbf d}$?
\item When such a reduction does not exist, what is the obstruction, and how should it be interpreted in quantum information?
\end{enumerate}
\end{question}

Unlike the standard untwisted setting, the subsystem choice is itself a local-to-global datum. After a local trivialization $SB(\mathcal A)|_U\simeq\mathbb P^{n-1}\times U$, the Segre variety $\Sigma_{\mathbf d}$ appears in each fiber. The twisting of the family may nevertheless prevent these local loci of product states from gluing. Without a global locus of product states, there is no uniform notion of product versus entangled states over $X$.

\subsection*{Our approach}
\paragraph{\textbf{(1) Stabilizer reductions and the obstruction to a global locus of product states.}}
Our first main result identifies this reduction problem with the geometric existence of a descended Segre locus.

The reduction criterion is the following. A relative Segre subscheme of type $\mathbf d$ exists in $SB(\mathcal A)$ \emph{if and only if} the torsor $P\to X$ under $\PGL_n$ admits a reduction of structure group to $G_{\mathbf d}$ (Corollary~\ref{cor:subsystem-structure-iff-segre}). Such a reduction is a subsystem structure of type $\mathbf d$ (Definition~\ref{def:subsystem-structure}). In this paper the failure of such a reduction is the obstruction to a global notion of product states.

In this sense, the obstruction measures the failure of local product state loci to globalize. When the reduction fails, transition functions necessarily exit the stabilizer $G_{\mathbf d}$ and the local product-state loci cannot be transported globally. In particular, a state that is locally a product state can become entangled after gluing along the base.

The Brauer class gives a necessary condition. We introduce $\Br_{\mathbf d}(X)\subset \Br(X)$, the subset of Brauer classes that admit at least one degree-$n$ Azumaya representative equipped with a subsystem structure of type $\mathbf d$ (Definition~\ref{def:br-d}). We prove a general torsion constraint
\[
\Br_{\mathbf d}(X)\subset \Br(X)\bigl[\mathrm{lcm}(\mathbf d)\bigr]
\qquad\text{(Proposition~\ref{prop:lcm-period-obstruction}),}
\]
and develop examples showing that the obstruction is not determined by the Brauer class and can depend on the isomorphism class of the torsor under $\PGL_n$.

\medskip
\paragraph{\textbf{(2) Moduli of subsystem structures as a Hilbert scheme locus}}
Subsystem structures, when they exist, are parametrized by a moduli space. We show that the functor of $G_{\mathbf d}$--reductions of $P$ is represented by the quotient $P/G_{\mathbf d}$, which embeds naturally as a locally closed subscheme of the relative Hilbert scheme:
\[
P/G_{\mathbf d}\ \cong\ \Hilb^{\Sigma_{\mathbf d}}\!\bigl(SB(\mathcal A)/X\bigr)
\ \subset\ \Hilb\bigl(SB(\mathcal A)/X\bigr)
\qquad\text{(Theorem~\ref{thm:subsystem-moduli-hilbert}).}
\]
In characteristic $0$ this locus is smooth over $X$ and étale-locally a trivial $\PGL_n/G_{\mathbf d}$--fibration (Proposition~\ref{prop:hilb-locus-geometry}). For $\mathbf d=(2,2)$ the relevant Hilbert component is $\PP^9$, and the boundary of the open Segre orbit is the (discriminant) divisor of singular quadrics (Proposition~\ref{prop:Hilb_quadric_P9} and Corollary~\ref{cor:relative_boundary_for_twisted}). This gives a compactification of $P/G_{\mathbf d}$ by degenerations inside the relative Hilbert scheme.

\medskip
\paragraph{\textbf{(3) From reductions to entanglement geometry}}
From Section~\ref{sec:schmidt} onward we specialize to the bipartite case $\mathbf d=(d_A,d_B)$. The determinantal part of entanglement geometry is \emph{not} available on a general Severi--Brauer family without extra structure. A twisted $\PP^{n-1}$--bundle does not come with a global tensor identification. Once a bipartite subsystem structure is fixed, the classical rank loci globalize to a flat filtration, which we call the \emph{entanglement filtration},
\[
\Sigma_{\le 1}(\mathcal A,\mathbf d)\subset \Sigma_{\le 2}(\mathcal A,\mathbf d)\subset \cdots \subset SB(\mathcal A),
\]
with base-change compatible incidence resolutions and a uniform singularity theory (Theorem~\ref{thm:relative-sing-and-resolution}). Moreover, the resulting intersection-theoretic and numerical data are fiberwise constant and computable (Theorem~\ref{thm:relative-numerics}, Corollary~\ref{cor:twisted_degree_pushforward}). Thus the chosen subsystem structure determines which relative rank loci and intersection classes are defined, and these constructions are compatible with base change.

\medskip
\paragraph{\textbf{(4) Explicit obstructions beyond the Brauer class}}
We give explicit obstruction computations. For instance, on $X=\Gm^2$ the generic Kummer $p^2$--symbol class does not admit a global $(p,p)$ subsystem structure (Theorem~\ref{thm:kummer-not-in-Brpp}), while we also construct nontrivial twisted degree-$p^2$ families that do admit a global $(p,p)$ reduction (Proposition~\ref{prop:pp_symbol_tensor_reducible} and Example~\ref{ex:pp_symbol_tensor_degree}). We further treat situations where different local tensor types occur on different charts (Example~\ref{ex:8-level}), so that the effective ``particle number'' depends on the locally chosen subsystem.

A recurrent theme is that subsystem reducibility is not determined by the Brauer class. Even for $\beta=0$ it can depend on the isomorphism class of the torsor under $\PGL_n$ (Proposition~\ref{prop:not-brauer-invariant}). On $\PP^1$ we obtain a complete splitting-type classification in the bipartite case (Theorem~\ref{thm:P1-bipartite-reducibility}).

Section~\ref{sec:algebraic-entangling-obstructions} gives the obstruction data not determined by the Brauer class. It starts with the stack of reductions to $G_{\mathbf d}$. The later constructions require determinant and incidence data and give relative motivic, Chow, and \'{e}tale classes on $X$.

\medskip
\paragraph{\textbf{(5) Quantum information interpretation through entangling monodromy}}
When $k=\mathbb C$, local trivializations on $X^{\mathrm{an}}$ can acquire monodromy in $\PGL_n(\mathbb C)$. Failure of reduction to $G_{\mathbf d}$ forces some transition elements outside $G_{\mathbf d}$, producing an \emph{entangling monodromy} acting on locally defined product states (Proposition~\ref{prop:entangling-monodromy-vs-reduction}). In a four-level model this monodromy is locally equivalent to a $\mathrm{CNOT}$-type gate
(Theorem~\ref{thm:p2-monodromy-CNOT}). Appendix~\ref{sec:spin-chain-toy} gives a spin chain example on a torus in which a locally product ground state becomes entangled after gluing.

We distinguish two local-to-global obstruction problems:
\begin{itemize}
  \item[\textbf{(I)}] \textbf{Obstruction of structure (this paper).}
  Whether a tensor-product type $\mathbf d$ can be chosen globally on the Severi--Brauer family, equivalent to the existence of a reduction of $P$ to $G_{\mathbf d}$ (equivalently, by the existence of a descended Segre subfibration).
  \item[\textbf{(II)}] \textbf{Obstruction of states (e.g.\ \cite{Ikeda:2025mgj,Abramsky:2011sbx}).} After fixing the background and a subsystem structure, whether local state data glue to a global state, and whether that global state is unique.
\end{itemize}
Both are local-to-global in nature, but they occur at different layers (structure versus states).

\subsection*{Organization}
Section~\ref{sec:azumaya-subsystems} introduces subsystem structures as $G_{\mathbf d}$--reductions and the subset $\Br_{\mathbf d}(X)\subset \Br(X)$. Section~\ref{subsec:segre-hilbert-locus} identifies the moduli of reductions with a canonically defined Hilbert scheme locus (Theorem~\ref{thm:subsystem-moduli-hilbert}). In the bipartite case, Section~\ref{sec:schmidt} develops the relative determinantal filtration, including uniform resolutions and numerical invariants. Examples and obstructions appear in Section~\ref{sec:examples}. Section~\ref{sec:algebraic-entangling-obstructions} gives the obstruction data not determined by the Brauer class. It separates the stack of reductions to $G_{\mathbf d}$, determinant lifting stacks, incidence lift groupoids, and relative cohomology classes on $X$. Some numerical invariants are collected in Section~\ref{sec:numerics}.

\section{Azumaya algebras and subsystem structures}
\label{sec:azumaya-subsystems}
\subsection{Azumaya algebras and subsystem structures}
\begin{definition}[Azumaya algebra of degree $n$]
An \emph{Azumaya algebra} of degree $n$ on $X$ is a sheaf of $\mathcal O_X$--algebras $\mathcal A$ which is fppf-locally isomorphic to $M_n(\mathcal O_X)$. It is associated to a principal $\PGL_n$--torsor $P\to X$ via
\[
  \mathcal A \;\cong\; P\times^{\PGL_n} \bigl(M_n(k)\otimes_k \mathcal O_X\bigr)
  \;\cong\; P\times^{\PGL_n} M_n(\mathcal O_X),
\]
where $\PGL_n$ acts on $M_n(k)$ by conjugation and acts trivially on $\mathcal O_X$ (so the action on $M_n(k)\otimes_k\mathcal O_X$ is $\mathcal O_X$--linear).
\end{definition}

\begin{definition}[Severi--Brauer scheme]
\label{def:SB}
Let $\mathcal A$ be an Azumaya algebra of degree $n$ on $X$. Its \emph{Severi--Brauer scheme} is the $X$--scheme
\[
  \pi:SB(\mathcal A)\longrightarrow X
\]
characterized fppf-locally by
$SB(\mathcal A)|_U \cong \mathbb P^{n-1}\times U$ whenever $\mathcal A|_U\cong M_n(\mathcal O_U)$. It represents locally free right ideals $I\subset \mathcal A$ of reduced rank $1$.
\end{definition}

We write $\Br(X):=H^2_{\mathrm{fppf}}(X,\mathbb G_m)_{\mathrm{tors}}$ for the Brauer group and denote by $\beta=[\mathcal A]\in \Br(X)$ the Brauer class of $\mathcal A$.

Let $\Sigma_{\mathbf d}\subset \mathbb P^{n-1}$ denote the Segre variety (\emph{locus of product states of type $\mathbf d$})
\[
  \Sigma_{\mathbf d}
  := \mathbb P^{d_1-1}\times\cdots\times \mathbb P^{d_s-1}
  \hookrightarrow \mathbb P(\,k^{d_1}\otimes\cdots\otimes k^{d_s}\,)
  \cong \mathbb P^{n-1}.
\]

\begin{definition}[Stabilizer]
\label{def:segrestab}
Define $G_{\mathbf d}\subset \PGL_n$ to be the (scheme-theoretic) stabilizer of the Segre subscheme:
\[
  G_{\mathbf d}
  := \{\, g\in \PGL_n \mid g(\Sigma_{\mathbf d})=\Sigma_{\mathbf d}\,\}.
\]
\end{definition}

\begin{remark}
The connected component of $G_{\mathbf d}$ contains the image of the tensor product representation $\prod_i \GL_{d_i}\to \GL_n\to \PGL_n$. If some factors have equal dimension, $G_{\mathbf d}$ may contain additional discrete symmetries permuting the equal Segre factors. None of our results depend on choosing an ordering convention.
\end{remark}

\begin{remark}[Structure of the stabilizer]
It is classical that the (projective) linear automorphisms preserving the product-state locus $\Sigma_{\mathbf d}\subset \PP^{n-1}$ form $\bigl(\prod_i \PGL_{d_i}\bigr)\rtimes\Pi$, where $\Pi$ permutes the equal-dimensional factors. In particular, $G_{\mathbf d}$ is smooth in $\mathrm{char}(k)=0$. See e.g.\ Westwick \cite{westwick1967transformations} (and a modern summary in \cite[Prop.~2.2]{2024arXiv240716767G}).
\end{remark}

\begin{definition}[Subsystem structure of type $\mathbf d$]
\label{def:subsystem-structure}
Let $\mathcal A$ be an Azumaya algebra of degree $n$ with associated $\PGL_n$--torsor $P\to X$ ($\PP^{n-1}$--bundle). A \emph{subsystem structure of type $\mathbf d$} on $\mathcal A$ is a reduction of structure group of $P$ to $G_{\mathbf d}$, consisting of a principal $G_{\mathbf d}$--torsor $P_{\mathbf d}\to X$ and an identification
\[
  P \;\cong\; P_{\mathbf d}\times^{G_{\mathbf d}} \PGL_n.
\]
\end{definition}

\begin{definition}[Brauer classes admitting subsystem structures of type $\mathbf d$]
\label{def:br-d}
Let $\mathbf d=(d_1,\dots,d_s)$ with $n=\prod_i d_i$ and let $G_{\mathbf d}\subset \PGL_n$ be the stabilizer. Define
\[
  \Br_{\mathbf d}(X)
  := \mathrm{Im}\Bigl(H^1(X,G_{\mathbf d})\to H^1(X,\PGL_n)\xrightarrow{\delta} \Br(X)\Bigr),
\]
where $\delta$ is the connecting morphism associated to $1\to \mathbb G_m\to \GL_n\to \PGL_n\to 1$ in fppf cohomology.
\end{definition}
Here $\beta\in \Br_{\mathbf d}(X)$ means that there exists at least one $\PGL_n$--torsor (Azumaya algebra of degree $n$) representing $\beta$ that admits a $G_{\mathbf d}$--reduction.

\begin{definition}[Brauer-theoretic obstruction to $\mathbf d$--subsystem structures]
\label{def:intrinsic-obstruction}
A Brauer class $\beta\in \Br(X)$ is \emph{$\mathbf d$--inadmissible} if $\beta\notin \Br_{\mathbf d}(X)$, meaning that no Azumaya algebra of degree $n$ representing $\beta$ admits a subsystem structure of type $\mathbf d$.
\end{definition}

\begin{remark}
In the quantum information interpretation, $\mathbf d$--inadmissibility means that the twisting encoded by $\beta$ forbids the existence of a global notion of ``product states of type $\mathbf d$'' on the family $SB(\mathcal A)\to X$. Thus the family has no global locus of product states of type $\mathbf d$. See Example \ref{ex:8-level}.
\end{remark}

\begin{theorem}
\label{thm:reduction-criterion}
Let $\mathcal A$ be an Azumaya algebra of degree $n$ with associated $\PGL_n$--torsor $P\to X$. Fix $\mathbf d$ and $G_{\mathbf d}\subset \PGL_n$.

\begin{enumerate}[label=\textup{(\roman*)}]
\item \textbf{Torsor level:} The algebra $\mathcal A$ admits a subsystem structure of type $\mathbf d$ if and only if
\[
  [P]\in \mathrm{Im}\bigl(H^1(X,G_{\mathbf d})\to H^1(X,\PGL_n)\bigr).
\]
\item \textbf{Brauer level:} If $\mathcal A$ admits a subsystem structure of type $\mathbf d$, then its Brauer class $[\mathcal A]=\delta([P])$ lies in $\Br_{\mathbf d}(X)$. Conversely, $\beta\in \Br_{\mathbf d}(X)$ guarantees the existence of some degree-$n$ Azumaya algebra $\mathcal A'$ representing $\beta$ that admits a subsystem structure of type $\mathbf d$, but this need not hold for every representative of $\beta$.
\end{enumerate}
\end{theorem}

\begin{proof}
This is the standard nonabelian cohomological criterion for reductions of structure group, applied to the inclusion $G_{\mathbf d}\hookrightarrow \PGL_n$. The Brauer-theoretic statements follow by applying the boundary map $\delta$.
\end{proof}


\subsection{From subsystem structures to a descended product-state locus}
\label{subsec:descended-product-locus}
\begin{definition}
\label{def:SSB}
Assume $\mathcal A$ is equipped with a subsystem structure $P_{\mathbf d}\to X$ (Definition~\ref{def:subsystem-structure}). Define the associated $X$--scheme
\[
  \Sigma_{\mathbf d}(\mathcal A)
  \;:=\; P_{\mathbf d}\times^{G_{\mathbf d}} \Sigma_{\mathbf d}
  \;\longrightarrow\; X.
\]
By construction, $\Sigma_{\mathbf d}(\mathcal A)$ is fppf-locally isomorphic to $\Sigma_{\mathbf d}\times U\to U$.
\end{definition}

\begin{proposition}[Descended product-state embedding]
\label{prop:segre-descends}
There is a closed immersion, functorial in $X$,
\[
  \iota_{\mathbf d}:\Sigma_{\mathbf d}(\mathcal A)\hookrightarrow SB(\mathcal A)
\]
which is fppf-locally identified with the classical Segre embedding $\Sigma_{\mathbf d}\hookrightarrow \mathbb P^{n-1}$.
\end{proposition}

\begin{proof}
Choose an fppf cover $U\to X$ trivializing the $G_{\mathbf d}$--torsor $P_{\mathbf d}$. After choosing a trivialization of $P_{\mathbf d}|_U$, we obtain isomorphisms
\[
  SB(\mathcal A)|_U \cong \mathbb P^{n-1}\times U,
  \qquad
  \Sigma_{\mathbf d}(\mathcal A)|_U\cong \Sigma_{\mathbf d}\times U,
\]
and hence the classical Segre embedding gives a closed immersion on $U$.

On overlaps $U\times_X U$, the two trivializations differ by a transition function with values in $G_{\mathbf d}$, and by definition $G_{\mathbf d}$ preserves the Segre subscheme $\Sigma_{\mathbf d}\subset \mathbb P^{n-1}$. Therefore the local Segre maps are compatible on overlaps and glue to a global closed immersion $\Sigma_{\mathbf d}(\mathcal A)\hookrightarrow SB(\mathcal A)$.
\end{proof}

\begin{corollary}[Subsystem structures $\Longleftrightarrow$ relative Segre subschemes]
\label{cor:subsystem-structure-iff-segre}
Let $\mathcal A$ be an Azumaya algebra of degree $n$ on $X$, with Severi--Brauer scheme $\pi:SB(\mathcal A)\to X$. Fix a factorization type $\mathbf d$ and the Segre subscheme $\Sigma_{\mathbf d}\subset \PP^{n-1}$.

A $\mathbf d$--subsystem structure on $\mathcal A$ exists if and only if there exists a closed subscheme
\[
  \Sigma\ \hookrightarrow\ SB(\mathcal A)
\]
flat over $X$ such that fppf-locally on $X$ the pair $(SB(\mathcal A),\Sigma)$ is isomorphic to $(\PP^{n-1}\times X,\ \Sigma_{\mathbf d}\times X)$, with $\Sigma$ identified with the standard Segre subscheme.

Moreover, when such $\Sigma$ exists it determines (up to unique isomorphism) a unique $G_{\mathbf d}$--reduction $P_{\mathbf d}\to X$ of the associated $\PGL_n$--torsor $P\to X$, and the induced subscheme $\Sigma_{\mathbf d}(\mathcal A)=P_{\mathbf d}\times^{G_{\mathbf d}}\Sigma_{\mathbf d}$ of Definition~\ref{def:SSB} identifies canonically with $\Sigma$.
\end{corollary}

\begin{proof}
The forward implication is Proposition~\ref{prop:segre-descends}.

Conversely, choose an fppf cover $U\to X$ trivializing $P$, so that $SB(\mathcal A)|_U\simeq \PP^{n-1}\times U$. By assumption, after refining $U$ we may identify $\Sigma|_U$ with $\Sigma_{\mathbf d}\times U$. On overlaps $U\times_X U$, the transition functions of $SB(\mathcal A)$ are $\PGL_n$--valued maps
$g_{ij}$, and the fact that the Segre subschemes glue forces $g_{ij}$ to preserve $\Sigma_{\mathbf d}$. Hence $g_{ij}$ takes values in $G_{\mathbf d}$, giving a $G_{\mathbf d}$--valued cocycle and therefore a $G_{\mathbf d}$--reduction $P_{\mathbf d}\to X$. By construction, the associated subscheme
$P_{\mathbf d}\times^{G_{\mathbf d}}\Sigma_{\mathbf d}$ recovers $\Sigma$.
\end{proof}

\subsection{Subsystem-structure loci}
\label{subsec:segre-hilbert-locus}
\subsubsection{Definitions}
Fix an Azumaya algebra $\mathcal A$ of degree $n$ on $X$, and let $P\to X$ be the associated $\PGL_n$--torsor so that $SB(\mathcal A)\cong P\times^{\PGL_n}\PP^{n-1}$. Fix a subsystem (factorization) type $\mathbf d=(d_1,\dots,d_s)$ with $n=\prod_i d_i$ and the stabilizer $G_{\mathbf d}\subset \PGL_n$ (Definition~\ref{def:segrestab}).

Since $X$ is noetherian and $\pi:SB(\mathcal A)\to X$ is projective, the full Hilbert functor
\begin{align}
\begin{aligned}
  \Hilb\bigl(SB(\mathcal A)/X\bigr)&:(\mathbf{Sch}/X)^{\mathrm{op}}\to \mathbf{Sets},\\\notag
  (T\xrightarrow{f}X)&\longmapsto
  \left\{
    \Sigma_T\subset SB(\mathcal A)\times_X T
    \;\middle|\;
    \Sigma_T\text{ is a closed subscheme flat over }T
  \right\}
\end{aligned}
\end{align}
is representable by a scheme, denoted by $\Hilb\bigl(SB(\mathcal A)/X\bigr)$.

Fix a relatively ample invertible sheaf $\mathcal L$ on $SB(\mathcal A)$. For each Hilbert polynomial $\mathsf P$ (with respect to $\mathcal L$), let $\Hilb^{\mathsf P}_{\mathcal L}\!\bigl(SB(\mathcal A)/X\bigr)$ denote the corresponding (projective) Hilbert scheme. Then the full Hilbert scheme decomposes as a disjoint union of open-and-closed subschemes
\[
  \Hilb\bigl(SB(\mathcal A)/X\bigr)
  \;=\;
  \coprod_{\mathsf P}\Hilb^{\mathsf P}_{\mathcal L}\bigl(SB(\mathcal A)/X\bigr).
\]
We will only use the existence of the full Hilbert scheme and the fact that each
$\Hilb^{\mathsf P}_{\mathcal L}$ is projective over $X$.

\begin{definition}[Reduction functor (subsystem structures in families)]
\label{def:reduction-functor}
Let $\mathbf{Sch}/X$ be the category of schemes over $X$. Define a functor
\[
  \mathfrak{Red}_{\mathbf d}(\mathcal A):(\mathbf{Sch}/X)^{\mathrm{op}}\to \mathbf{Sets}
\]
by sending $f:T\to X$ to the set of \emph{$G_{\mathbf d}$--reductions of $P_T:=P\times_X T$}, namely pairs
\[
  (P_{\mathbf d,T},\ \alpha_T)
\]
where $P_{\mathbf d,T}\to T$ is a principal $G_{\mathbf d}$--torsor and $\alpha_T:P_{\mathbf d,T}\times^{G_{\mathbf d}} \PGL_n \xrightarrow{\sim} P_T$ is an isomorphism of $\PGL_n$--torsors, modulo the usual notion of isomorphism of reductions.
\end{definition}

\begin{definition}
\label{def:segre-hilbert-functor}
Define a subfunctor
\[
  \mathfrak{Hilb}^{\Sigma_{\mathbf d}}_{\mathcal A}\ \subset\
  \Hilb\bigl(SB(\mathcal A)/X\bigr)
\]
as follows. For $f:T\to X$, $\mathfrak{Hilb}^{\Sigma_{\mathbf d}}_{\mathcal A}(T)$ consists of closed subschemes
\[
  \Sigma_T\ \subset\ SB(\mathcal A)\times_X T
\]
flat over $T$, such that there exists an fppf cover $\{T_i\to T\}$ with:
\begin{enumerate}[label=(\roman*)]
\item $SB(\mathcal A)\times_X T_i\ \cong\ \PP^{n-1}\times T_i$ over $T_i$
      ($\mathcal A_{T_i}\cong M_n(\mathcal O_{T_i})$, so the induced $\PGL_n$--torsor on $T_i$ is trivial), and
\item under such an identification, $\Sigma_T|_{T_i}$ is identified with $\Sigma_{\mathbf d}\times T_i\subset \PP^{n-1}\times T_i$.
\end{enumerate}
\end{definition}

\subsubsection{The Segre orbit inside the full Hilbert scheme of $\PP^{n-1}$}
Let $\Hilb(\PP^{n-1})$ denote the full Hilbert scheme of $\PP^{n-1}$ (the disjoint union of $\Hilb^{\mathsf P}(\PP^{n-1})$ over all Hilbert polynomials ${\mathsf P}$ with respect to $\mathcal O_{\PP^{n-1}}(1)$). The natural action of $\PGL_n$ on $\PP^{n-1}$ induces an action on $\Hilb(\PP^{n-1})$.

Let $[\Sigma_{\mathbf d}]\in \Hilb(\PP^{n-1})(k)$ be the point corresponding to the Segre subscheme, and let $\mathscr O_{\Sigma_{\mathbf d}}\subset \Hilb(\PP^{n-1})$ denote its $\PGL_n$--orbit. Note that $\mathscr O_{\Sigma_{\mathbf d}}$ is automatically contained in the open-and-closed component of $\Hilb(\PP^{n-1})$ containing $[\Sigma_{\mathbf d}]$.

\begin{lemma}[Segre orbit as a homogeneous space]
\label{lem:segre-orbit-homogeneous}
Let $\mathscr O_{\Sigma_{\mathbf d}} \subset \Hilb(\PP^{n-1})$ be the $\PGL_n$--orbit of $[\Sigma_{\mathbf d}]$.
Then:
\begin{enumerate}[label=(\roman*)]
\item The stabilizer of $[\Sigma_{\mathbf d}]$ in $\PGL_n$ is precisely $G_{\mathbf d}$.
\item The orbit morphism
\[
  \phi:\ \PGL_n \longrightarrow \Hilb(\PP^{n-1}),\qquad g\longmapsto g\cdot[\Sigma_{\mathbf d}]
\]
factors through the quotient to give a morphism
\[
  \bar\phi:\ \PGL_n/G_{\mathbf d}\ \longrightarrow\ \Hilb(\PP^{n-1})
\]
which is a locally closed immersion with image $\mathscr O_{\Sigma_{\mathbf d}}$.
In particular, $\mathscr O_{\Sigma_{\mathbf d}}$ is a smooth locally closed subscheme and
\[
  \mathscr O_{\Sigma_{\mathbf d}}\ \cong\ \PGL_n/G_{\mathbf d}.
\]
\end{enumerate}
\end{lemma}

\begin{proof}
(i) This is immediate from the definition of $G_{\mathbf d}$ as the scheme-theoretic stabilizer of the Segre subscheme $\Sigma_{\mathbf d}\subset \PP^{n-1}$ (Definition~\ref{def:segrestab}).

(ii) Since $G_{\mathbf d}$ stabilizes $[\Sigma_{\mathbf d}]$, the orbit map $\phi$ is constant on left cosets of $G_{\mathbf d}$, hence factors through a unique morphism $\bar\phi:\PGL_n/G_{\mathbf d}\to \Hilb(\PP^{n-1})$.

We claim that $\bar\phi$ is a monomorphism. Let $T$ be a scheme and let $x_1,x_2\in (\PGL_n/G_{\mathbf d})(T)$ with $\bar\phi(x_1)=\bar\phi(x_2)$. After an fppf cover $T'\to T$ we may choose lifts $g_1,g_2\in \PGL_n(T')$ mapping to $x_1|_{T'}$ and $x_2|_{T'}$. The equality $\bar\phi(x_1)=\bar\phi(x_2)$ implies
\[
  g_1(\Sigma_{\mathbf d}\times T')\;=\;g_2(\Sigma_{\mathbf d}\times T') \quad\text{as closed subschemes of }\PP^{n-1}\times T',
\]
hence $g_2^{-1}g_1$ stabilizes $\Sigma_{\mathbf d}\times T'$, hence $g_2^{-1}g_1\in G_{\mathbf d}(T')$ by (i). Therefore $g_1$ and $g_2$ define the same $T'$--point of $\PGL_n/G_{\mathbf d}$, so $x_1|_{T'}=x_2|_{T'}$, and hence $x_1=x_2$. Thus $\bar\phi$ is a monomorphism.

A standard result implies that a monomorphism locally of finite type is an immersion. Hence $\bar\phi$ is an immersion. Its image is exactly the orbit $\mathscr O_{\Sigma_{\mathbf d}}$, so $\mathscr O_{\Sigma_{\mathbf d}}$ is locally closed. Smoothness follows because in characteristic $0$ the stabilizer $G_{\mathbf d}$ is smooth (its identity component is smooth and the component group is finite étale), hence $\PGL_n/G_{\mathbf d}$ is smooth.
\end{proof}

\subsubsection{Twisting to the Severi--Brauer scheme}

\begin{lemma}[Full relative Hilbert scheme as an associated bundle]
\label{lem:hilb-associated-bundle}
There is a canonical isomorphism of $X$--schemes
\[
  \Hilb\bigl(SB(\mathcal A)/X\bigr)
  \ \cong\
  P\times^{\PGL_n} \Hilb(\PP^{n-1}),
\]
where $\PGL_n$ acts on $\Hilb(\PP^{n-1})$ via its action on $\PP^{n-1}$. Under this identification, universal families correspond by descent.
\end{lemma}

\begin{proof}
Choose an fppf cover $U\to X$ trivializing the $\PGL_n$--torsor $P$, so that
\[
  SB(\mathcal A)|_U \cong \PP^{n-1}\times U.
\]
By base change for Hilbert schemes, for each Hilbert polynomial $\mathsf P$ one has
\[
  \Hilb^{\mathsf P}\!\bigl(SB(\mathcal A)/X\bigr)\times_X U
  \ \cong\
  \Hilb^{\mathsf P}\!\bigl((\PP^{n-1}\times U)/U\bigr)
  \ \cong\
  \Hilb^{\mathsf P}(\PP^{n-1})\times U.
\]
Taking the disjoint union over all $\mathsf P$ yields
\[
  \Hilb\bigl(SB(\mathcal A)/X\bigr)\times_X U
  \ \cong\
  \Hilb(\PP^{n-1})\times U.
\]
On overlaps $U\times_X U$, two trivializations differ by a transition function with values in $\PGL_n$. The induced gluing on the Hilbert factor is precisely the $\PGL_n$--action on $\Hilb(\PP^{n-1})$. Thus the gluing rule on overlaps is governed precisely by the $\PGL_n$--action on $\Hilb(\PP^{n-1})$, which is exactly the associated-bundle construction $P\times^{\PGL_n}\Hilb(\PP^{n-1})$.
\end{proof}

\begin{definition}[The subsystem-structure locus]
\label{def:segre-hilbert-locus}
Define the \emph{subsystem-structure locus} to be the locally closed $X$--subscheme
\[
  \Hilb^{\Sigma_{\mathbf d}}\!\bigl(SB(\mathcal A)/X\bigr)
  \;:=\; P\times^{\PGL_n}\mathscr O_{\Sigma_{\mathbf d}}
  \ \subset\
  P\times^{\PGL_n}\Hilb(\PP^{n-1})
  \ \cong\
  \Hilb\bigl(SB(\mathcal A)/X\bigr),
\]
where $\mathscr O_{\Sigma_{\mathbf d}}\subset \Hilb(\PP^{n-1})$ is the orbit given in Lemma~\ref{lem:segre-orbit-homogeneous}.
\end{definition}

\begin{lemma}[Identification with the quotient $P/G_{\mathbf d}$]
\label{lem:hilb-locus-is-quotient}
There is a canonical isomorphism of $X$--schemes
\[
  \Hilb^{\Sigma_{\mathbf d}}\!\bigl(SB(\mathcal A)/X\bigr)\ \cong\ P/G_{\mathbf d},
\]
and under this identification the inclusion $\Hilb^{\Sigma_{\mathbf d}}(SB(\mathcal A)/X)\hookrightarrow \Hilb(SB(\mathcal A)/X)$
is a locally closed immersion.
\end{lemma}

\begin{proof}
\noindent
Here $P/G_{\mathbf d}$ denotes the fppf quotient sheaf, which is represented by an $X$--scheme since $G_{\mathbf d}\subset \PGL_n$ is a smooth affine subgroup. So, $P/G_{\mathbf d}\cong P\times^{\PGL_n}(\PGL_n/G_{\mathbf d})$.

By Lemma~\ref{lem:segre-orbit-homogeneous}, $\mathscr O_{\Sigma_{\mathbf d}}\cong \PGL_n/G_{\mathbf d}$ as $\PGL_n$--schemes. Hence
\[
  P\times^{\PGL_n}\mathscr O_{\Sigma_{\mathbf d}}
  \ \cong\
  P\times^{\PGL_n}(\PGL_n/G_{\mathbf d}).
\]
There is a standard isomorphism
\[
  P\times^{\PGL_n}(\PGL_n/G_{\mathbf d})\ \xrightarrow{\sim}\ P/G_{\mathbf d},
  \qquad [p,\ gG_{\mathbf d}] \longmapsto (p\cdot g)\,G_{\mathbf d},
\]
which is well-defined and functorial. Local closedness follows fppf-locally on $X$: over a trivializing cover $U\to X$, the inclusion becomes
\[
  (\PGL_n/G_{\mathbf d})\times U\ \hookrightarrow\ \Hilb(\PP^{n-1})\times U,
\]
which is locally closed by Lemma~\ref{lem:segre-orbit-homogeneous}, and these local immersions glue.
\end{proof}

\begin{proposition}[Geometric structure of the Hilbert scheme locus]
\label{prop:hilb-locus-geometry}
Assume $\mathrm{char}(k)=0$, so that $G_{\mathbf d}$ is smooth. Then the subsystem-structure locus
\[
  \Hilb^{\Sigma_{\mathbf d}}\!\bigl(SB(\mathcal A)/X\bigr)\ \cong\ P/G_{\mathbf d}
\]
is a smooth $\PGL_n/G_{\mathbf d}$-fibration. In particular, the structure morphism
\[
  P/G_{\mathbf d}\ \longrightarrow\ X
\]
is smooth of relative dimension $\dim(\PGL_n/G_{\mathbf d})$.

Moreover, since $G_{\mathbf d}^\circ$ is (up to finite kernel) the image of $\prod_i \PGL_{d_i}$, one has
\[
  \dim(\PGL_n/G_{\mathbf d})
  \;=\; n^2-\sum_{i=1}^s d_i^2+(s-1).
\]
\end{proposition}

\begin{proof}
Since $\mathcal A$ is an Azumaya algebra, there exists an \'etale cover $U\to X$ such that $\mathcal A|_U\simeq M_n(\mathcal O_U)$. After such a base change, the associated $\PGL_n$--torsor $P\to X$ becomes trivial and one has $P/G_{\mathbf d}|_U\simeq (\PGL_n/G_{\mathbf d})\times U$.

In characteristic $0$ the stabilizer $G_{\mathbf d}$ is smooth, hence the homogeneous space $\PGL_n/G_{\mathbf d}$ is smooth (cf.\ Lemma~\ref{lem:segre-orbit-homogeneous}). Therefore the projection
\[
  (\PGL_n/G_{\mathbf d})\times U \longrightarrow U
\]
is smooth of relative dimension $\dim(\PGL_n/G_{\mathbf d})$, and smoothness (and relative dimension) descends along \'etale covers. This proves that $P/G_{\mathbf d}\to X$ is smooth with the stated relative dimension.

Finally, since $G_{\mathbf d}^\circ$ is (up to finite kernel) the image of $\prod_i \PGL_{d_i}$, we have
\[
  \dim(\PGL_n/G_{\mathbf d})
  = \dim(\PGL_n)-\dim(G_{\mathbf d})
  = (n^2-1)-\sum_{i=1}^s(d_i^2-1)
  = n^2-\sum_{i=1}^s d_i^2+(s-1),
\]
and the finite component group does not affect the dimension.
\end{proof}

\subsubsection{Equivalence of moduli functors and representability}

\begin{theorem}[Subsystem structures as a Hilbert-scheme locus]
\label{thm:subsystem-moduli-hilbert}
Let $\mathcal A$ be an Azumaya algebra of degree $n$ on $X$ with associated $\PGL_n$--torsor $P\to X$, and fix $\mathbf d$ with stabilizer $G_{\mathbf d}\subset \PGL_n$.

\begin{enumerate}[label=\textup{(\roman*)}]
\item The functor $\mathfrak{Red}_{\mathbf d}(\mathcal A)$ of $G_{\mathbf d}$--reductions of $P$ is represented by the quotient $X$--scheme $P/G_{\mathbf d}$.

\item The subfunctor $\mathfrak{Hilb}^{\Sigma_{\mathbf d}}_{\mathcal A}\subset \Hilb(SB(\mathcal A)/X)$ of relative closed subschemes that are fppf-locally isomorphic to $\Sigma_{\mathbf d}\times X$ is represented by the same scheme.

\item Under these identifications, a global $\mathbf d$--subsystem structure on $\mathcal A$ is equivalently a section
\[
X\ \longrightarrow\ P/G_{\mathbf d}\ \cong\ \Hilb^{\Sigma_{\mathbf d}}\!\bigl(SB(\mathcal A)/X\bigr).
\]
\end{enumerate}
\end{theorem}

\begin{proof}
We construct mutually inverse natural transformations directly, by gluing in families.

\medskip
\noindent\textbf{(Step 1) From reductions to Segre families.}
Let $f:T\to X$ and let $(P_{\mathbf d,T},\alpha_T)\in \mathfrak{Red}_{\mathbf d}(\mathcal A)(T)$. Form the associated $T$--scheme
\[
  \Sigma_{\mathbf d}(\mathcal A_T)\ :=\ P_{\mathbf d,T}\times^{G_{\mathbf d}} \Sigma_{\mathbf d}.
\]
Using $\alpha_T$, the induced isomorphism of associated $\PP^{n-1}$--bundles
\[
  P_{\mathbf d,T}\times^{G_{\mathbf d}}\PP^{n-1}\ \xrightarrow{\sim}\ P_T\times^{\PGL_n}\PP^{n-1}
  \;=\; SB(\mathcal A_T)
\]
identifies the descended closed immersion $\Sigma_{\mathbf d}(\mathcal A_T)\hookrightarrow P_{\mathbf d,T}\times^{G_{\mathbf d}}\PP^{n-1}$
(with fiberwise the classical Segre embedding) with a closed immersion
\[
  \Sigma_{\mathbf d}(\mathcal A_T)\hookrightarrow SB(\mathcal A_T)=SB(\mathcal A)\times_X T.
\]
Let $\Sigma_T$ denote its image. Then $\Sigma_T$ is flat over $T$ and fppf-locally on $T$ it identifies with $\Sigma_{\mathbf d}\times T$, hence $\Sigma_T\in \mathfrak{Hilb}^{\Sigma_{\mathbf d}}_{\mathcal A}(T)$. This yields a natural transformation
\[
  \Theta:\ \mathfrak{Red}_{\mathbf d}(\mathcal A)\to \mathfrak{Hilb}^{\Sigma_{\mathbf d}}_{\mathcal A}.
\]

\medskip
\noindent\textbf{(Step 2) From Segre families to reductions.}
Conversely, let $\Sigma_T\in \mathfrak{Hilb}^{\Sigma_{\mathbf d}}_{\mathcal A}(T)$. Choose an fppf cover $\{T_i\to T\}$ satisfying Definition~\ref{def:segre-hilbert-functor}. Thus $SB(\mathcal A)\times_X T_i\simeq \PP^{n-1}\times T_i$ and $\Sigma_T|_{T_i}\simeq \Sigma_{\mathbf d}\times T_i$. On overlaps $T_{ij}$, the transition functions for the $\PP^{n-1}$--bundle $SB(\mathcal A)\times_X T\to T$ are maps $g_{ij}:T_{ij}\to \PGL_n$. Since the Segre subschemes match in both trivializations, each $g_{ij}$ preserves $\Sigma_{\mathbf d}$, hence takes values in $G_{\mathbf d}$ by definition of the stabilizer. Therefore $\{g_{ij}\}$ is a $G_{\mathbf d}$--valued cocycle, defining a principal $G_{\mathbf d}$--torsor $P_{\mathbf d,T}\to T$ and an induced identification $P_{\mathbf d,T}\times^{G_{\mathbf d}}\PGL_n \simeq P_T$. This yields a natural transformation
\[
  \Psi:\ \mathfrak{Hilb}^{\Sigma_{\mathbf d}}_{\mathcal A}\to \mathfrak{Red}_{\mathbf d}(\mathcal A).
\]

\medskip
\noindent\textbf{(Step 3) Inverses.}
The constructions are inverse because they are governed by the same $G_{\mathbf d}$--valued descent datum. If one starts with a reduction, the descended Segre family has transition functions in $G_{\mathbf d}$. Conversely, if one starts with a Segre family, the transition functions of $SB(\mathcal A_T)$ are forced to lie in $G_{\mathbf d}$, and this recovers the reduction.

\medskip
\noindent\textbf{(Step 4) Representability.}
Let $f:T\to X$ be an $X$--scheme and let $\Sigma_T\subset SB(\mathcal A)\times_X T$ be the family corresponding to a morphism
$\xi:T\to \Hilb\bigl(SB(\mathcal A)/X\bigr)$.

By Lemma~\ref{lem:hilb-associated-bundle} one has a natural identification
\[
  \Hilb\bigl(SB(\mathcal A)/X\bigr)
  \;\simeq\;
  P\times^{\PGL_n}\Hilb(\PP^{n-1}),
\]
and by Definition~\ref{def:segre-hilbert-locus} the subsystem-structure locus (Hilbert scheme locus) is
\[
  \Hilb^{\Sigma_{\mathbf d}}\bigl(SB(\mathcal A)/X\bigr)
  \;=\;
  P\times^{\PGL_n}\mathscr O_{\Sigma_{\mathbf d}}
  \;\subset\;
  P\times^{\PGL_n}\Hilb(\PP^{n-1}).
\]
Since $\mathscr O_{\Sigma_{\mathbf d}}\hookrightarrow \Hilb(\PP^{n-1})$ is a locally closed immersion, it is a monomorphism. Hence the condition that a morphism factors through $P\times^{\PGL_n}\mathscr O_{\Sigma_{\mathbf d}}$ may be checked fppf--locally on $T$.

Choose an fppf covering $\{T_i\to T\}$ trivializing the pullback torsor $P_T\to T$, so that $SB(\mathcal A)\times_X T_i\simeq \PP^{n-1}\times T_i$. Over each $T_i$ the family $\Sigma_T|_{T_i}$ corresponds to a classifying morphism
$\xi_i:T_i\to \Hilb(\PP^{n-1})$. By construction, $\xi$ factors through $P\times^{\PGL_n}\mathscr O_{\Sigma_{\mathbf d}}$
if and only if each $\xi_i$ factors through $\mathscr O_{\Sigma_{\mathbf d}}\subset \Hilb(\PP^{n-1})$.

On the other hand, by Definition~\ref{def:segre-hilbert-functor}, $\Sigma_T\in \mathfrak{Hilb}^{\Sigma_{\mathbf d}}_{\mathcal A}(T)$ means precisely that after an fppf refinement of $\{T_i\to T\}$ one may choose trivializations $SB(\mathcal A)\times_X T_{ij}\simeq \PP^{n-1}\times T_{ij}$ under which $\Sigma_T|_{T_{ij}}$ identifies with $\Sigma_{\mathbf d}\times T_{ij}$. After such a refinement the corresponding classifying maps $T_{ij}\to \Hilb(\PP^{n-1})$ land in the $\PGL_n$--orbit $\mathscr O_{\Sigma_{\mathbf d}}$.
Thus
\[
  \Sigma_T\in \mathfrak{Hilb}^{\Sigma_{\mathbf d}}_{\mathcal A}(T)
  \quad\Longleftrightarrow\quad
  \xi \text{ factors through }
  \Hilb^{\Sigma_{\mathbf d}}\!\bigl(SB(\mathcal A)/X\bigr),
\]
so $\Hilb^{\Sigma_{\mathbf d}}\!\bigl(SB(\mathcal A)/X\bigr)$ represents the functor $\mathfrak{Hilb}^{\Sigma_{\mathbf d}}_{\mathcal A}$.

Finally, Lemma~\ref{lem:hilb-locus-is-quotient} identifies $\Hilb^{\Sigma_{\mathbf d}}\bigl(SB(\mathcal A)/X\bigr)\cong P/G_{\mathbf d}$. Combining this with Steps~2--3, we conclude that $P/G_{\mathbf d}$ represents both $\mathfrak{Hilb}^{\Sigma_{\mathbf d}}_{\mathcal A}$ and $\mathfrak{Red}_{\mathbf d}(\mathcal A)$. Statement~(iii) is obtained by evaluating the representing property at $T=X$.

\end{proof}

\begin{remark}
\label{rem:compactification}
The identification in Theorem \ref{thm:subsystem-moduli-hilbert}~{(iii)}
\[
  \Hilb^{\Sigma_{\mathbf d}}\bigl(SB(\mathcal A)/X\bigr)\ \cong\ P/G_{\mathbf d}
\]
realizes the moduli of subsystem structures as a canonically defined locally closed subscheme
of the full relative Hilbert scheme. In particular, by
Proposition~\ref{prop:hilb-locus-geometry} it is a smooth $X$--scheme (an \'{e}tale-locally trivial
$\PGL_n/G_{\mathbf d}$--bundle).

Let $\Hilb_{[\Sigma_{\mathbf d}]}(SB(\mathcal A)/X)\subset \Hilb(SB(\mathcal A)/X)$ denote the open-and-closed component
containing the Segre point (the component corresponding to the Hilbert polynomial of $\Sigma_{\mathbf d}$).
This component is projective over $X$. Hence the closure
\[
  \overline{\Hilb^{\Sigma_{\mathbf d}}}\ \subset\ \Hilb_{[\Sigma_{\mathbf d}]}\bigl(SB(\mathcal A)/X\bigr)
\]
provides a canonical projective compactification of $P/G_{\mathbf d}$ over $X$.
Its boundary
$\partial:=\overline{\Hilb^{\Sigma_{\mathbf d}}}\setminus \Hilb^{\Sigma_{\mathbf d}}$
parametrizes flat degenerations of product-state loci inside the twisted family $SB(\mathcal A)\to X$:
fiberwise, points of $\partial$ correspond to limits of Segre fibers in $\PP^{n-1}$.

In the case $\mathbf d=(2,2)$, this boundary becomes explicit. The relevant Hilbert component is $\PP^9$ (Proposition~\ref{prop:Hilb_quadric_P9}),
and the boundary is the discriminant divisor of singular quadrics. Twisting yields the corresponding
relative discriminant divisor (Corollary~\ref{cor:relative_boundary_for_twisted}).
\end{remark}

\subsection{The case $\mathbf d=(2,2)$}\label{subsec:22_boundary}

Let $n=4$ and $\mathbf d=(2,2)$. Then $\Sigma_{2,2}\subset \PP^3$ is a smooth quadric surface. Let $\mathscr H:=\Hilb_{[\Sigma_{2,2}]}(\PP^3)$ be the Hilbert scheme component containing $\Sigma_{2,2}$.

\begin{proposition}\label{prop:Hilb_quadric_P9}
There is a natural isomorphism
\[
\mathscr H \;\simeq\; \PP\big(\mathrm{Sym}^2(k^4)^\vee\big)\;\simeq\; \PP^9,
\]
parameterizing quadric hypersurfaces in $\PP^3$. Under this identification, the $\PGL_4$-orbit of $[\Sigma_{2,2}]$ equals the open subset of smooth quadrics, and its complement is the discriminant hypersurface of singular quadrics. Moreover, the stabilizer of $[\Sigma_{2,2}]$ in $\PGL_4$ is $G_{(2,2)}$ (so the open orbit is $\PGL_4/G_{(2,2)}$).
\end{proposition}

\begin{proof}
A quadric in $\PP^3$ is defined by a homogeneous polynomial of degree $2$, hence is a point of $\PP(\mathrm{Sym}^2(k^4)^\vee)=\PP^9$. The universal family of quadrics is flat, giving a morphism $\PP^9\to \Hilb(\PP^3)$ whose image is the component $\mathscr H$, and this morphism is an isomorphism. Smoothness is detected by the nonvanishing of the determinant of the associated symmetric matrix,
so the smooth locus is the complement of the discriminant hypersurface. Transitivity of $\PGL_4$ on smooth quadrics identifies this open locus with $\PGL_4/G_{(2,2)}$.
\end{proof}

\begin{corollary}\label{cor:relative_boundary_for_twisted}
Let $P\to X$ be a $\PGL_4$-torsor, and let $SB(\mathcal{A})\to X$ be the associated Severi--Brauer fibration. Then the Hilbert component $\Hilb_{[\Sigma_{2,2}]}(SB(\mathcal{A})/X)$ is the $\PP^9$-bundle
\[
P\times^{\PGL_4}\PP^9 \;\to\; X,
\]
and the subsystem-structure space $P/G_{(2,2)}$ identifies with the open complement of the relative discriminant divisor. In particular, the closure of $P/G_{(2,2)}$ inside this Hilbert component is projective over $X$ and its boundary is a divisor (the relative discriminant).
\end{corollary}

\begin{proof}
Apply the functorial identification of the relative Hilbert scheme with the associated bundle $P\times^{\PGL_4}\Hilb(\PP^3)$, and use Proposition~\ref{prop:Hilb_quadric_P9} fiberwise.
\end{proof}

\section{Entanglement filtration}
\label{sec:schmidt}

From here on we assume $\mathbf d=(d_A,d_B)$ with $n=d_A d_B$. Write $\Sigma_{A,B}:=\mathbb P^{d_A-1}\times \mathbb P^{d_B-1}$ for the product-state locus. In the bipartite case we also write $\Sigma_{A,B}=\Sigma_{(d_A,d_B)}$.

\begin{remark}[Global nature of determinantal geometry in the twisted setting]
Even though the determinantal loci $R_{\le r}\subset \PP(V_A\otimes V_B)$ are classical, they do \emph{not} make sense on a general Severi--Brauer scheme without extra data. A twisted $\PP^{n-1}$--bundle carries no canonical global tensor-product identification.
A $\mathbf d$--subsystem structure is precisely the additional datum required to globalize the entire entanglement filtration, its incidence resolutions, and the resulting intersection-theoretic invariants
uniformly over $X$.
\end{remark}

\subsection{Entanglement filtration and relative geometry}
\label{subsec:relative-sing-resolution}
\subsubsection{Schmidt-rank loci as $G_{\mathbf d}$--invariant subschemes}
Let $V_A:=k^{d_A}$ and $V_B:=k^{d_B}$. For $1\le r\le \min(d_A,d_B)$, let $R_{\le r}\subset \mathbb P(V_A\otimes V_B)$ be the classical determinantal variety of tensors of (matrix) rank $\le r$, defined by the $(r+1)\times(r+1)$ minors of a flattening matrix.
It is a closed, irreducible, $\GL(V_A)\times \GL(V_B)$--invariant subscheme.

\begin{definition}[Loci and strata]
\label{def:relative-schmidt}
Let $(\mathcal A,P_{\mathbf d})$ be an Azumaya algebra with a chosen subsystem structure.
Define the closed subscheme
\[
  \Sigma_{\le r}(\mathcal A,\mathbf d)
  \;:=\; P_{\mathbf d}\times^{G_{\mathbf d}} R_{\le r}
  \;\subset\; P_{\mathbf d}\times^{G_{\mathbf d}}\mathbb P(V_A\otimes V_B)
  \;\cong\; SB(\mathcal A).
\]
We call $\Sigma_{\le r}(\mathcal A,\mathbf d)$ the \emph{Schmidt-rank $\le r$ locus}, and $\Sigma_{= r}(\mathcal A,\mathbf d):=\Sigma_{\le r}(\mathcal A,\mathbf d)\setminus \Sigma_{\le r-1}(\mathcal A,\mathbf d)$ the \emph{Schmidt-rank $r$ stratum} when $r>1$. In particular, when $r=1$, they are called \emph{product-state locus} and \emph{product-state stratum}, respectively. We then have the natural \emph{entanglement filtration}:
\[
\Sigma_{\le 1}(\mathcal A,\mathbf d)\subset \Sigma_{\le 2}(\mathcal A,\mathbf d)\subset\cdots\subset SB(\mathcal A).
\]
\end{definition}
By definition, for all $r\ge 2$, every element of $\Sigma_{=r}$ is entangled, while every element of $\Sigma_{=1}$ is a product state. We define $\Sigma_{=0}=\emptyset$. 

\begin{proposition}[Relative Schmidt rank loci]
\label{prop:schmidt-descent}
Each $\Sigma_{\le r}(\mathcal A,\mathbf d)\subset SB(\mathcal A)$ is a well-defined closed subscheme, flat over $X$.
Moreover, for any morphism $f:Y\to X$,
\[
  \Sigma_{\le r}(f^*\mathcal A,\mathbf d) \;\cong\; f^*\Sigma_{\le r}(\mathcal A,\mathbf d)
\]
as closed subschemes of $SB(f^*\mathcal A)\cong f^*SB(\mathcal A)$.
\end{proposition}

\begin{proof}
The determinantal variety $R_{\le r}\subset \PP(V_A\otimes V_B)$ is invariant under the stabilizer $G_{\mathbf d}\subset \PGL_n$. Hence, in a $G_{\mathbf d}$--adapted local trivialization of $SB(\mathcal A)$, the closed subscheme $R_{\le r}$ is carried consistently across overlaps by the transition functions. This produces a globally defined closed subscheme
\[
  \Sigma_{\le r}(\mathcal A,\mathbf d)=P_{\mathbf d}\times^{G_{\mathbf d}} R_{\le r}\ \subset\ SB(\mathcal A).
\]

Flatness follows because fppf-locally on $X$ the pair $\bigl(SB(\mathcal A),\Sigma_{\le r}(\mathcal A,\mathbf d)\bigr)$ is isomorphic to
$\bigl(\PP(V_A\otimes V_B),R_{\le r}\bigr)\times U$, and $R_{\le r}\times U\to U$ is flat. Compatibility with base change is formal for associated bundles.
\end{proof}

\begin{remark}
The filtration
\[
  \Sigma_{A,B}(\mathcal A)=\Sigma_{\le 1}(\mathcal A,\mathbf d)
  \subset \Sigma_{\le 2}(\mathcal A,\mathbf d)\subset \cdots \subset \Sigma_{\le \min(d_A,d_B)}(\mathcal A,\mathbf d)=SB(\mathcal A)
\]
is the algebro-geometric counterpart of the entanglement hierarchy by Schmidt rank. It exists only when a subsystem structure exists.
\end{remark}

\subsubsection{Relative singularities, normality, and canonical resolutions}

We continue in the bipartite case $\mathbf d=(d_A,d_B)$ with $n=d_A d_B$. Let $V_A:=k^{d_A}$ and $V_B:=k^{d_B}$, so that $SB(\mathcal A)_x\cong \PP(V_A\otimes V_B)$ in a $G_{\mathbf d}$--adapted local trivialization.

Recall that $R_{\le r}\subset \PP(V_A\otimes V_B)$ denotes the classical determinantal variety of tensors (of matrix rank) $\le r$, and that
\[
  \Sigma_{\le r}(\mathcal A,\mathbf d)\;=\;P_{\mathbf d}\times^{G_{\mathbf d}} R_{\le r}\ \subset\ SB(\mathcal A)\qquad \text{(Definition~\ref{def:relative-schmidt}).}
\]

Let
\[
  \Mat_{d_A\times d_B}:=\Hom(V_B^\vee,V_A)
\]
be the affine space of $d_A\times d_B$ matrices (rows indexed by $V_A$, columns indexed by $V_B$). Via either natural identification $V_A\otimes V_B \cong \Hom(V_A^\vee,V_B)\cong \Hom(V_B^\vee,V_A)$, matrix rank coincides with tensor (Schmidt) rank in the bipartite model.

Let $Z_{\le r}\subset \Mat_{d_A\times d_B}$ be the affine determinantal locus defined by the $(r+1)\times(r+1)$ minors.

\begin{proposition}[Singular locus of determinantal varieties]
\label{prop:det-singular-locus}
Let $1\le r\le \min(d_A,d_B)$.
\begin{enumerate}[label=(\roman*)]
\item If $1\le r<\min(d_A,d_B)$, then the singular locus is
\[
  \Sing(Z_{\le r})=Z_{\le r-1}
\qquad\text{and}\qquad
  \Sing(R_{\le r})=R_{\le r-1}.
\]
\item If $r=\min(d_A,d_B)$, then $R_{\le r}=\PP(V_A\otimes V_B)$ is smooth, hence
\[
  \Sing(R_{\le r})=\varnothing.
\]
\end{enumerate}
\end{proposition}
\begin{proof}
This is classical determinantal geometry: for the affine determinantal variety $Z_{\le r}$ of $d_A\times d_B$ matrices, the singular locus is $Z_{\le r-1}$ for $r<\min(d_A,d_B)$, and $Z_{\le r}$ is smooth for $r=\min(d_A,d_B)$. Projectivizing gives $\Sing(R_{\le r})=R_{\le r-1}$ in the projective space. See e.g.\ standard references on determinantal varieties (e.g.\ Bruns--Vetter \cite{BrunsVetter1988} or Weyman \cite{Weyman_2003}) for detailed proofs.
\end{proof}

There are several standard desingularizations of $R_{\le r}$. For our relative purposes (and to allow possible discrete symmetries when $d_A=d_B$), it is convenient to use a symmetric incidence resolution involving both Grassmannians.

Let $\Gr_A(r):=\Gr(r,V_A)$ and $\Gr_B(r):=\Gr(r,V_B)$, with tautological subbundles $\mathcal U_A\subset V_A\otimes\mathcal O_{\Gr_A(r)}$ and $\mathcal U_B\subset V_B\otimes\mathcal O_{\Gr_B(r)}$.

\begin{definition}[Symmetric incidence resolution of $R_{\le r}$]
\label{def:incidence-resolution}
Define the smooth projective variety
\[
  \widetilde{R}_{\le r}
  \;:=\;
  \PP(\mathcal U_A\boxtimes \mathcal U_B)
  \ \longrightarrow\
  \Gr_A(r)\times \Gr_B(r),
\]
and let
\[
  \rho_r:\widetilde{R}_{\le r}\longrightarrow \PP(V_A\otimes V_B)
\]
be the morphism induced by the natural inclusion $\mathcal U_A\boxtimes\mathcal U_B \hookrightarrow (V_A\otimes V_B)\otimes\mathcal O$.
Its image lies in $R_{\le r}$, since any tensor in $U_A\otimes U_B$ has rank $\le r$.
\end{definition}

\begin{proposition}[Properties of the incidence resolution]
\label{prop:incidence-resolution-properties}
For $1\le r\le \min(d_A,d_B)$, the morphism
\[
  \rho_r:\widetilde{R}_{\le r}\longrightarrow R_{\le r}
\]
is projective and birational. Moreover, 
\begin{enumerate}[label=(\roman*)]
\item $\widetilde{R}_{\le r}$ is smooth.
\item $\rho_r$ restricts to an isomorphism over the open locus
\[
  R_{=r}\;:=\;R_{\le r}\setminus R_{\le r-1}
\]
of tensors of rank exactly $r$. (In particular, if $r<\min(d_A,d_B)$ then $R_{=r}$ is the smooth locus of $R_{\le r}$
by Proposition~\ref{prop:det-singular-locus}.) 
\item Consequently, $\rho_r$ is a resolution of singularities of $R_{\le r}$ (and for $r=\min(d_A,d_B)$, a projective birational morphism from a smooth variety onto the already smooth $R_{\le r}=\PP(V_A\otimes V_B)$).
\end{enumerate}
The construction is functorial for the natural $\GL(V_A)\times \GL(V_B)$--action. Moreover, when $d_A=d_B$ it is compatible with the factor-swap involution (it exchanges the two Grassmannian factors $\Gr_A(r)$ and $\Gr_B(r)$). Hence the construction is equivariant for the full stabilizer $G_{\mathbf d}\subset \PGL(V_A\otimes V_B)$.
\end{proposition}

\begin{proof}
Smoothness of $\widetilde{R}_{\le r}=\PP(\mathcal U_A\boxtimes \mathcal U_B)$ follows because it is a projective bundle over the smooth base $\Gr_A(r)\times \Gr_B(r)$. The map $\rho_r$ is projective by construction and is birational since it restricts to an isomorphism over the open rank-$r$ locus $R_{=r}$, where the image subspaces $\mathrm{im}(\psi_A)$ and $\mathrm{im}(\psi_B)$ vary regularly. Equivariance under $\GL(V_A)\times\GL(V_B)$ (and the factor-swap involution when $d_A=d_B$) is functorial. For details, see standard treatments of incidence resolutions of determinantal varieties.
\end{proof}

\begin{remark}[One-sided resolutions and small resolutions]
There are also standard one-sided incidence resolutions using only $\Gr_B(r)$ (image) or only $\Gr_A(r)$ (kernel/transpose). The symmetric resolution above dominates both by forgetting one Grassmannian factor. In certain special cases one-sided resolutions can be small. We do not use smallness here.
\end{remark}

\subsubsection{Relative entangling: singular locus and resolution for $\Sigma_{\le r}(\mathcal A,\mathbf d)$}

\begin{definition}[Relative incidence resolution of Schmidt-rank loci]
\label{def:relative-incidence-resolution}
Let $(\mathcal A,P_{\mathbf d})$ be an Azumaya algebra with a chosen subsystem structure. Define the $X$--scheme
\[
  \widetilde{\Sigma}_{\le r}(\mathcal A,\mathbf d)
  \;:=\;
  P_{\mathbf d}\times^{G_{\mathbf d}} \widetilde{R}_{\le r},
\]
and let
\[
  \rho_r:\widetilde{\Sigma}_{\le r}(\mathcal A,\mathbf d)\ \longrightarrow\ \Sigma_{\le r}(\mathcal A,\mathbf d)
\]
be the induced morphism from the $G_{\mathbf d}$--equivariant map $\widetilde{R}_{\le r}\to R_{\le r}$.
\end{definition}

\begin{theorem}[Relative singular locus and resolution]
\label{thm:relative-sing-and-resolution}
Let $X$ be smooth over $k$ (char $0$) and let $(\mathcal A,P_{\mathbf d})$ be an Azumaya algebra with a chosen subsystem structure.
Fix $1\le r\le \min(d_A,d_B)$. Then:
\begin{enumerate}[label=(\roman*)]
\item \textbf{Relative singular locus:}
If $1\le r<\min(d_A,d_B)$, then scheme-theoretically
\[
  \Sing\bigl(\Sigma_{\le r}(\mathcal A,\mathbf d)\bigr)
  \;=\;
  \Sigma_{\le r-1}(\mathcal A,\mathbf d).
\]
In particular, let $\pi_r:=\pi|_{{\Sigma_{\le r}(\mathcal A,\mathbf d)}}$. The morphism $\pi_r:\Sigma_{\le r}(\mathcal A,\mathbf d)\to X$ is smooth exactly over the open stratum $\Sigma_{=r}(\mathcal A,\mathbf d):=\Sigma_{\le r}(\mathcal A,\mathbf d)\setminus \Sigma_{\le r-1}(\mathcal A,\mathbf d)$.

If $r=\min(d_A,d_B)$, then $\Sigma_{\le r}(\mathcal A,\mathbf d)=SB(\mathcal A)$ is smooth over $X$, hence
\[
  \Sing\bigl(\Sigma_{\le r}(\mathcal A,\mathbf d)\bigr)=\varnothing.
\]
(Here, one may adopt the convention $\Sigma_{\le 0}(\mathcal A,\mathbf d)=\varnothing$ so that for $r=1$ the formula reads $\Sing(\Sigma_{\le 1})=\Sigma_{\le 0}=\varnothing$.)
\item \textbf{Resolution:}
The total space $\widetilde{\Sigma}_{\le r}(\mathcal A,\mathbf d)$ is smooth over $X$, and
\[
  \rho_r:\widetilde{\Sigma}_{\le r}(\mathcal A,\mathbf d)\longrightarrow \Sigma_{\le r}(\mathcal A,\mathbf d)
\]
is a projective birational morphism which is an isomorphism over $\Sigma_{=r}(\mathcal A,\mathbf d)$. In particular, $\rho_r$ gives a resolution of singularities of $\Sigma_{\le r}(\mathcal A,\mathbf d)$.
\item \textbf{Base change compatibility:}
For any morphism $f:Y\to X$,
\[
  f^*\widetilde{\Sigma}_{\le r}(\mathcal A,\mathbf d)\ \cong\ \widetilde{\Sigma}_{\le r}(f^*\mathcal A,\mathbf d),
\qquad
  f^*\Sigma_{\le r}(\mathcal A,\mathbf d)\ \cong\ \Sigma_{\le r}(f^*\mathcal A,\mathbf d),
\]
and these identifications intertwine the resolution maps.
\end{enumerate}
\end{theorem}

\begin{proof}
Choose an fppf cover (in fact \'etale since $G_{\mathbf d}$ is smooth) $U\to X$ trivializing $P_{\mathbf d}$. After choosing a trivialization of $P_{\mathbf d}|_U$, we obtain isomorphisms
\[
  SB(\mathcal A)|_U\ \cong\ \PP(V_A\otimes V_B)\times U,
\qquad
  \Sigma_{\le r}(\mathcal A,\mathbf d)|_U\ \cong\ R_{\le r}\times U,
\]
and similarly
\[
  \widetilde{\Sigma}_{\le r}(\mathcal A,\mathbf d)|_U
  \ \cong\ \widetilde{R}_{\le r}\times U,
\qquad
  \rho_r|_U\ =\ (\rho_r:\widetilde{R}_{\le r}\to R_{\le r})\times \id_U.
\]

\medskip
\noindent\textbf{(i) Singular locus.}
Assume first that $1\le r<\min(d_A,d_B)$. By Proposition~\ref{prop:det-singular-locus}, $\Sing(R_{\le r})=R_{\le r-1}$. Since $U$ is smooth, taking the product with $U$ preserves singular loci:
\[
  \Sing(R_{\le r}\times U)\;=\;\Sing(R_{\le r})\times U\;=\;R_{\le r-1}\times U.
\]
Indeed, the singular locus is defined by scheme-theoretically by the appropriate Fitting ideal of $\Omega^1$, and for $U$ smooth one has $\Omega^1_{R_{\le r}\times U/k}\cong \mathrm{pr}_1^*\Omega^1_{R_{\le r}/k}\oplus \mathrm{pr}_2^*\Omega^1_{U/k}$, so the Fitting ideal pulls back from $R_{\le r}$.

Transporting back via the identifications gives
\[
  \Sing\bigl(\Sigma_{\le r}(\mathcal A,\mathbf d)|_U\bigr)
  \;=\;
  \Sigma_{\le r-1}(\mathcal A,\mathbf d)|_U.
\]
Because $\Sing(-)$ is defined scheme-theoretically by the relevant Fitting ideal of $\Omega^1$ and Fitting ideals commute with étale base change, the equality descends, hence $\Sing(\Sigma_{\le r}(\mathcal A,\mathbf d))=\Sigma_{\le r-1}(\mathcal A,\mathbf d)$ scheme-theoretically on $X$. The equivalent smoothness statement for $\pi_r$ follows immediately.

If $r=\min(d_A,d_B)$, then $R_{\le r}=\PP(V_A\otimes V_B)$ is smooth, so $\Sigma_{\le r}(\mathcal A,\mathbf d)|_U\cong \PP(V_A\otimes V_B)\times U$ is smooth. Hence $\Sigma_{\le r}(\mathcal A,\mathbf d)=SB(\mathcal A)$ is smooth over $X$ and its singular locus is empty.

\medskip
\noindent\textbf{(ii) Resolution and smoothness over $X$.}
By Proposition~\ref{prop:incidence-resolution-properties}, $\widetilde{R}_{\le r}$ is smooth and $\rho_r:\widetilde{R}_{\le r}\to R_{\le r}$ is a projective birational morphism, an isomorphism over $R_{=r}$. Therefore $\widetilde{R}_{\le r}\times U$ is smooth over $U$, and the product map $(\rho_r\times \id_U)$ is a resolution of $R_{\le r}\times U$. By descent (again using that smoothness and properness are fpqc-local on the source/target), $\widetilde{\Sigma}_{\le r}(\mathcal A,\mathbf d)$ is smooth over $X$, and $\rho_r$ is projective birational and an isomorphism over $\Sigma_{=r}$.

\medskip
\noindent\textbf{(iii) Base change.}
Both $\Sigma_{\le r}$ and $\widetilde{\Sigma}_{\le r}$ are defined as associated bundles $P_{\mathbf d}\times^{G_{\mathbf d}}(\,\cdot\,)$, hence commute with pullback in the base. The compatibility of the maps follows formally from functoriality of associated bundles.
\end{proof}

\begin{corollary}
\label{cor:relative-normal-CM}
For $1\le r<\min(d_A,d_B)$, the locus $\Sigma_{\le r}(\mathcal A,\mathbf d)$ is normal and Cohen--Macaulay. Moreover, in characteristic $0$ it has rational singularities (hence in particular is Cohen--Macaulay).
\end{corollary}

\begin{proof}
These properties are \'{e}tale local on the source. Over an \'{e}tale cover $U\to X$ trivializing $P_{\mathbf d}$ one has
$\Sigma_{\le r}|_U\cong R_{\le r}\times U$. The classical determinantal variety $R_{\le r}$ is normal and Cohen--Macaulay (and has rational singularities in char $0$), and taking the product with a smooth scheme preserves these properties.
Therefore $\Sigma_{\le r}(\mathcal A,\mathbf d)$ has the same local singularity profile.
\end{proof}

\begin{remark}
The resolution $\rho_r$ is canonical once the subsystem structure is fixed, and it exists uniformly in families. Thus one can apply the usual birational and intersection theoretic constructions to the relative Schmidt rank loci. 
\end{remark}

\subsection{Cycle classes and computability in the split model}
The construction gives the cycle class
\[
  [\Sigma_{\le r}(\mathcal A,\mathbf d)] \in CH^*(SB(\mathcal A)).
\]
This cycle class is not intrinsic to $\mathcal A$, but becomes canonical once a subsystem structure $\mathbf d$ is chosen and fixed.

\begin{corollary}[A computable numerical invariant in twisted families]
\label{cor:twisted_degree_pushforward}
Assume $X$ is connected and smooth, and let $(\mathcal A,P_{\mathbf d})$ be equipped with a bipartite subsystem structure $\mathbf d=(d_A,d_B)$.
Let $\pi:SB(\mathcal A)\to X$ and write $H:=c_1(\cO_{SB(\mathcal A)}(1))$. For $1\le r\le \min(d_A,d_B)$ set
\[
N_r:=\dim(R_{\le r})=r(d_A+d_B-r)-1.
\]
Then the class
\[
\pi_*\bigl(H^{N_r}\cap [\Sigma_{\le r}(\mathcal A,\mathbf d)]\bigr)\ \in\ CH^0(X)
\]
is the constant integer $\deg(R_{\le r})$ on each connected component of $X$. So every geometric fiber $(\Sigma_{\le r})_x\subset SB(\mathcal A)_x\simeq \PP^{n-1}$ has the same projective degree $\deg(R_{\le r})$ as the classical determinantal variety.
\end{corollary}

\begin{proof}
The statement is fiberwise and follows from flatness. Indeed, $\Sigma_{\le r}(\mathcal A,\mathbf d)\to X$ is flat with fibers isomorphic (fppf-locally on $X$) to the fixed projective subscheme $R_{\le r}\subset \PP^{n-1}$ equipped with the hyperplane polarization. Hence the Hilbert polynomial of the fibers of $\Sigma_{\le r}(\mathcal A,\mathbf d)\to X$ with respect to $\cO_{SB(\mathcal A)}(1)|_{\Sigma_{\le r}}$ is locally constant, and therefore constant on each connected component of $X$.

By definition, the projective degree of a fiber is the leading coefficient of this Hilbert polynomial (up to the standard factorial normalization), so the fiberwise degree is constant on $X$. On a geometric point $x\to X$ where the subsystem structure trivializes, the fiber $(\Sigma_{\le r})_x\subset SB(\mathcal A)_x\simeq \PP^{n-1}$ identifies with $R_{\le r}$, hence has degree $\deg(R_{\le r})$. Therefore every fiber has degree $\deg(R_{\le r})$.

Finally, the class $\pi_*\!\bigl(H^{N_r}\cap[\Sigma_{\le r}(\mathcal A,\mathbf d)]\bigr)\in CH^0(X)$ is determined by these fiberwise degrees, and equals $\deg(R_{\le r})\cdot [X]$ on each connected component.
\end{proof}

\begin{example}[A twisted $(p,p)$ model with an explicit relative degree]
\label{ex:pp_symbol_tensor_degree}
Work with the setup of Proposition~\ref{prop:pp_symbol_tensor_reducible} on $X=\Gm^2$, so that the degree-$p^2$ Azumaya algebra $A$ admits a global $(p,p)$--subsystem structure. Then the product-state locus $\Sigma_{\le 1}(A,(p,p))\subset SB(A)$ is a flat family whose geometric fibers are the Segre varieties $\PP^{p-1}\times \PP^{p-1}\subset \PP^{p^2-1}$. Hence every fiber has degree
\[
\deg\bigl((\Sigma_{\le 1})_x\bigr)=\deg\bigl(\PP^{p-1}\times \PP^{p-1}\bigr)=\binom{2p-2}{p-1}.
\]
Consequently, with $H=c_1(\cO_{SB(A)}(1))$ one has in $CH^0(X)$
\[
  \pi_*\bigl(H^{2p-2}\cap [\Sigma_{\le 1}(A,(p,p))]\bigr)
  \;=\;\binom{2p-2}{p-1}\cdot [X].
\]
For $p=2$ this recovers the familiar degree-$2$ quadric surface in $\PP^3$.
\end{example}

In the split case one can compute these classes by standard determinantal formulas.

\begin{proposition}[Split computation via Thom--Porteous]
\label{prop:TP-split}
Assume $\mathcal A\cong \underline{\End}(E)$ is split and the subsystem structure comes from a tensor factorization $E\cong E_A\otimes E_B$ with $\mathrm{rk}(E_A)=d_A$, $\mathrm{rk}(E_B)=d_B$. Then $SB(\mathcal A)\cong \mathbb P(E)$ and $\Sigma_{\le r}(\mathcal A,\mathbf d)$ is the degeneracy locus of the universal map
\[
  \Phi:\ \pi^*(E_A^\vee)\otimes \mathcal O_{\mathbb P(E)}(-1)\longrightarrow \pi^*(E_B),
\]
and its class in $CH^*(\mathbb P(E))$ is given by the Thom--Porteous formula in terms of the Chern classes of the virtual bundle $\pi^*(E_B)-\pi^*(E_A^\vee)\otimes \mathcal O_{\mathbb P(E)}(-1)$.
\end{proposition}

\begin{proof}
A point $[\psi]\in \mathbb P(E_x)=\mathbb P(E_{A,x}\otimes E_{B,x})$ corresponds to a tensor $\psi$, hence a linear map $E_{A,x}^\vee\to E_{B,x}$ by contraction. The rank of this map is exactly the Schmidt rank of $\psi$. Globalizing over $\mathbb P(E)$ uses the tautological line subbundle $\mathcal O(-1)\subset \pi^*E$ and yields $\Phi$. The degeneracy locus $\{\mathrm{rank}(\Phi)\le r\}$ is precisely $\Sigma_{\le r}$, and Thom--Porteous gives the class.
\end{proof}

\begin{remark}
Even in the split case, the determinantal strata depend on the \emph{chosen} tensor factorization $E\cong E_A\otimes E_B$.
There is no canonical way to extract such a factorization from $E$ or from $\underline{\End}(E)$. Thus the strata and their classes are attached to the chosen subsystem structure, not to $E$ or $\underline{\End}(E)$ alone.
\end{remark}

\section{Examples and obstructions beyond the Brauer class}
\label{sec:examples}
We collect examples illustrating quantum entanglement geometry. Further physical examples appear in\cite{Ikeda:2026fen}. 

\begin{summary*}[Outline of the examples]\
\begin{enumerate}[label=\textup{(\roman*)}]
\item \textbf{Brauer-theoretic obstruction (torsion).}
Proposition~\ref{prop:lcm-period-obstruction} gives the general $\mathrm{lcm}(\mathbf d)$--torsion constraint, and Theorem~\ref{thm:kummer-not-in-Brpp} exhibits a Kummer $p^2$--symbol class on $X=\Gm^2$ that cannot admit a global $(p,p)$ subsystem structure.

\item \textbf{Twisted but reducible case.}
Proposition~\ref{prop:pp_symbol_tensor_reducible} constructs a nontrivial twisted degree-$p^2$ class on the same base
that does globalize a $(p,p)$ subsystem structure (hence supports the full entanglement filtration). Example~\ref{ex:pp_symbol_tensor_degree} gives an explicit numerical invariant of its product-state locus.

\item \textbf{Torsion is not sufficient.}
Theorem~\ref{thm:indec_exp_p_ind_p2_not_Brpp} shows that even exponent-$p$ classes can fail to lie in $\Br_{(p,p)}$:
the obstruction depends on the isomorphism class of the torsor (index/indecomposability phenomena).

\item \textbf{Not a Brauer invariant (even split).}
Proposition~\ref{prop:not-brauer-invariant} and Theorem~\ref{thm:P1-bipartite-reducibility}
show that $\mathbf d$--reducibility depends on the underlying $\PGL_n$--torsor, not merely on the Brauer class. Example~\ref{ex:P1-split-different} gives an explicit split illustration on $\PP^1$.

\item \textbf{Monodromy and its quantum information interpretation}
The analytic monodromy picture in \S\ref{sec:entangling_monodromy} translates the failure of reduction into monodromy of local trivializations. In particular, Theorem~\ref{thm:p2-monodromy-CNOT} identifies the $p=2$ monodromy with a $\mathrm{CNOT}$-type operation. Example~\ref{ex:8-level} gives an eight-level example combining these phenomena.

\item \textbf{Quantum spin chain.} Appendix~\ref{sec:spin-chain-toy} constructs a spin chain example on a torus in which locally product ground states become entangled after gluing.
\end{enumerate}
\end{summary*}

\subsection{Obstructions: general constraints and explicit symbol examples}
\subsubsection{A general period obstruction: an $\mathrm{lcm}(\mathbf d)$--torsion constraint}
\label{subsec:lcm-period-obstruction}

\begin{proposition}[A general period obstruction for $\mathbf d$--decomposable Brauer classes]
\label{prop:lcm-period-obstruction}
Fix a subsystem type $\mathbf d=(d_1,\dots,d_s)$ and set $n=\prod_i d_i$ and
\[
  \ell:=\operatorname{lcm}(d_1,\dots,d_s).
\]
Let $G_{\mathbf d}\subset \PGL_n$ be the \emph{stabilizer} (Definition~\ref{def:segrestab}), possibly disconnected when some $d_i$ coincide, and let $\Br_{\mathbf d}(X)$ be as in Definition~\ref{def:br-d}. Then for every scheme $X$ one has
\[
  \Br_{\mathbf d}(X)\ \subset\ \Br(X)[\ell].
\]
In particular, if $\beta\in \Br_{\mathbf d}(X)$, then $\per(\beta)\mid \ell$.
\end{proposition}

\begin{proof}
Let $G_{\mathbf d}\subset \PGL_n$ be the stabilizer. Define $\widetilde{G}_{\mathbf d}$ to be the subgroup fppf sheaf of $\GL_n$ generated by
\begin{enumerate}[label=(\alph*)]
\item the image of the tensor-product homomorphism $\prod_{i=1}^s \SL_{d_i}\to \SL_n\subset \GL_n$, and
\item the permutation-matrix lifts of all permutations of equal Segre factors acting on
$k^{d_1}\otimes\cdots\otimes k^{d_s}$.
\end{enumerate}
By construction, the composite morphism of fppf sheaves
\[
  \widetilde{G}_{\mathbf d}\hookrightarrow \GL_n \twoheadrightarrow \PGL_n
\]
has image $G_{\mathbf d}$. In particular, $\widetilde{G}_{\mathbf d}\to G_{\mathbf d}$ is surjective as an fppf sheaf.

\medskip
\noindent\textbf{Claim.} The kernel of $\widetilde{G}_{\mathbf d}\to G_{\mathbf d}$ is the scalar subgroup $\mu_\ell\subset \GL_n$, where $\ell=\mathrm{lcm}(d_1,\dots,d_s)$.

Indeed, the kernel consists of those elements mapping to the identity in $\PGL_n$, hence of scalar matrices in $\GL_n$. Moreover, if an element of $\widetilde{G}_{\mathbf d}$ maps to a scalar matrix in $\GL_n$, then its image in the (permutation) component group is necessarily trivial (a nontrivial permutation matrix cannot be scalar), so such a scalar element must lie in the tensor-product part. The only scalars in the image of $\prod_i \SL_{d_i}\to \GL_n$ are tensor products of central elements $\zeta_i\cdot I_{d_i}$ with $\zeta_i\in \mu_{d_i}$. Such an element maps to $(\prod_i \zeta_i)\cdot I_n$. Since $d_i\mid \ell$ for all $i$, these scalars lie in $\mu_\ell$, and the subgroup they generate is exactly $\mu_\ell$. This proves the claim.

\medskip
Thus we obtain a central extension of fppf sheaves of groups
\[
  1\longrightarrow \mu_\ell \longrightarrow \widetilde{G}_{\mathbf d}
  \longrightarrow G_{\mathbf d}\longrightarrow 1,
\]
and a commutative diagram of central extensions
\[
\begin{tikzcd}
1 \arrow[r] & \mu_\ell \arrow[r] \arrow[d,hook] &
\widetilde{G}_{\mathbf d} \arrow[r] \arrow[d,hook] &
G_{\mathbf d} \arrow[r] \arrow[d,hook,"i"] & 1\\
1 \arrow[r] & \Gm \arrow[r] &
\GL_n \arrow[r] &
\PGL_n \arrow[r] & 1.
\end{tikzcd}
\]

Let $[P_{\mathbf d}]\in H^1(X,G_{\mathbf d})$ and set $[P]:=i_*[P_{\mathbf d}]\in H^1(X,\PGL_n)$. Since $\mu_\ell$ is central, the boundary map for the top extension lands in the abelian group
\[
  \partial_\ell: H^1(X,G_{\mathbf d})\longrightarrow H^2(X,\mu_\ell).
\]
By functoriality of boundary maps for morphisms of extensions, we have
\[
  \delta_n([P]) \;=\; (\mu_\ell\hookrightarrow \Gm)_*\bigl(\partial_\ell([P_{\mathbf d}])\bigr)
  \ \in\ H^2(X,\Gm),
\]
where $\delta_n:H^1(X,\PGL_n)\to H^2(X,\Gm)$ is the usual connecting morphism. 

Finally, the image of $H^2(X,\mu_\ell)\to H^2(X,\Gm)$ is $\ell$--torsion (because $\mu_\ell\subset \Gm$ is killed by $\ell$), so $\ell\cdot \delta_n([P])=0$. Hence $\delta_n([P])\in \Br(X)[\ell]$, proving $\Br_{\mathbf d}(X)\subset \Br(X)[\ell]$.
\end{proof}

\subsubsection{Period/index obstruction via a Kummer symbol (generic period $m$)}
\label{subsec:kummer-obstruction}

Fix an integer $m\ge 2$ and assume $\mu_m\subset k$ (so\ $k$ contains a primitive $m$-th root of unity $\zeta_m$). (Throughout the paper $\mathrm{char}(k)=0$, so $m$ is automatically invertible in $k$.) Let
\[
  X := \mathbb G_{m,k}^2 = \Spec k[u^{\pm 1},v^{\pm 1}]
\]
which is smooth, integral, and regular.

\begin{definition}[A symbol Azumaya algebra of degree $m$]
\label{def:symbol-azumaya}
Define the sheaf of $\mathcal O_X$--algebras $\mathcal A_{(u,v),m}$ on $X$ by generators and relations: it is locally free of rank $m^2$ as an $\mathcal O_X$--module, generated by $x,y$ with relations
\[
  x^m=u,\qquad y^m=v,\qquad yx=\zeta_m\, xy,
\]
and with the obvious $\mathcal O_X$--linearity.
\end{definition}
For the relationship between symbol constructions and Brauer classes via the norm-residue homomorphism (and its consequences for central simple/Azumaya algebras), see
\cite{1983IzMat..21..307M}.

\begin{proposition}[Azumaya and Brauer class]
\label{prop:symbol-azumaya}
$\mathcal A_{(u,v),m}$ is an Azumaya algebra of degree $m$ on $X$. Its Brauer class $\beta_{(u,v),m}:=[\mathcal A_{(u,v),m}]\in \Br(X)$ is $m$--torsion.
\end{proposition}

\begin{proof}
This is the standard Kummer/cyclic (symbol) construction. Over the Kummer cover adjoining an $m$--th root of $u$ (and similarly for $v$), the algebra becomes a matrix algebra. Therefore, it is a crossed product algebra associated to the $\mu_m$--torsors defined by $u$ and $v$. Hence it is an Azumaya algebra over $X$, and its Brauer class is annihilated by $m$.
\end{proof}

\begin{definition}[Period and index]
\label{def:period-index}
Let $X$ be integral with function field $K=k(X)$ and let $\beta\in \Br(X)$.
\begin{itemize}
\item The period $\mathrm{per}(\beta)$ is the order of $\beta$ in $\Br(X)$.
\item The index $\mathrm{ind}(\beta)$ is the index of the class $\beta_K\in Br(K)$ in the Brauer group of the field $K$ (the minimal degree of a central simple algebra
representing $\beta_K$).
\end{itemize}
One always has $\mathrm{per}(\beta)\mid \mathrm{ind}(\beta)$.
\end{definition}

\begin{lemma}[Generic symbol has period $m$]
\label{lem:symbol-period-m}
Let $K=k(u,v)$ be the rational function field. The symbol class $\beta_{(u,v),m}|_{\Spec K}\in \Br(K)$ has period $m$. Consequently, $\per(\beta_{(u,v),m})=m$ in $\Br(X)$, hence also $\ind(\beta_{(u,v),m})=m$.
\end{lemma}

\begin{proof}
Over $K$ the class is the usual symbol algebra $(u,v)_m$. Consider the discrete valuation $v_u$ corresponding to the prime divisor $u=0$ on $\Spec k[u,v]$. For $m$ invertible in $k$ and $\mu_m\subset k$, the residue map for $m$--torsion Brauer classes sends the symbol $(u,v)_m$ to the Kummer class of $v$ in the residue field $k(v)$. Since $v$ is not an $m$-th power in $k(v)$, this residue has exact order $m$. Hence $(u,v)_m$ has exact order $m$ in $\Br(K)$.

Since $X$ is regular, the natural map $\Br(X)\hookrightarrow \Br(K)$ is injective, so $\per(\beta_{(u,v),m})=m$ in $\Br(X)$ as well. Finally, for a degree-$m$ central simple algebra one has $\ind\mid m$, while $\per=m$ forces $\ind=m$.
\end{proof}

\subsubsection{Obstruction for the bipartite type $(p,p)$}
We now compare this class to the bipartite subsystem type $\mathbf d=(p,p)$ (so $n=p^2=m$). From this point on we specialize to $m=p^2$ (with $p$ prime) in order to discuss the bipartite subsystem type $(p,p)$.

\begin{corollary}[A necessary torsion constraint for $\mathbf d=(p,p)$]
\label{cor:pp-torsion-constraint}
Let $p$ be a prime and let $X$ be any scheme. If $\beta\in \Br_{(p,p)}(X)$, then $\per(\beta)\mid p$.
\end{corollary}

\begin{proof}
This is Proposition~\ref{prop:lcm-period-obstruction} with $\ell=\mathrm{lcm}(p,p)=p$.
\end{proof}

\begin{theorem}[A Kummer symbol class not in $\Br_{(p,p)}(X)$]
\label{thm:kummer-not-in-Brpp}
With $X=\mathbb G_m^2$ and $\beta=\beta_{(u,v),p^2}$ as above, one has
\[
  \beta \notin \Br_{(p,p)}(X).
\]
\end{theorem}

\begin{proof}
By Lemma~\ref{lem:symbol-period-m}, $\mathrm{per}(\beta)=p^2$. But Corollary~\ref{cor:pp-torsion-constraint} shows that any class in $\Br_{(p,p)}(X)$ must have period dividing $p$. Hence $\beta\notin \Br_{(p,p)}(X)$.
\end{proof}

\begin{remark}[Interpretation of the example]
This example gives a geometric obstruction to a subsystem structure. The fibration $SB(\mathcal A_{(u,v),p^2})\to \mathbb G_m^2$ exists, but no bipartite $(p,p)$ Segre subfibration exists for this Brauer class.
\end{remark}

\subsection{\label{sec:entangling_monodromy}An analytic model with entangling monodromy}
The results of \S\ref{sec:azumaya-subsystems}--\S\ref{sec:numerics} are algebro--geometric and do not require analytic topology. In this subsection we specialize to $k=\mathbb C$ and interpret local trivializations of the $\PGL_n$--torsor on $X^{\mathrm{an}}$ in terms of analytic continuation along loops. The resulting monodromy elements in $\PGL_n(\mathbb C)$ translate the failure of reduction to $G_{\mathbf d}$ into monodromy of local trivializations.

\medskip

Fix a prime $p$ and set $m=p^{2}$. Put
\[
\mathcal{X} = (\mathbb{C}^{\times})^{2}, \qquad (u,v)\in \mathcal{X},
\]
and let
\[
\zeta := e^{2\pi i/m}.
\]
We describe an explicit local model of the symbol Azumaya algebra and show that a gauge monodromy becomes \emph{entangling} with respect to the bipartite type $(p,p)$.

\vskip0.2cm
We first define a local (analytic) trivialization and a parameter-dependent representation. Choose a simply connected open set $U\subset \mathcal{X}$ on which we can pick analytic branches of the $m$-th roots $u^{1/m}$ and $v^{1/m}$. Fix the Hilbert space
\[
\mathcal H := \mathbb{C}^{m},
\]
with computational basis $\{|r\rangle\}_{r\in \mathbb{Z}/m\mathbb{Z}}$. Define the Weyl (shift/clock) operators on $\mathcal H$ by
\begin{equation}
\label{eq:gate}
\mathbf{X}|r\rangle = |r+1\rangle,\qquad 
\mathbf{Z}|r\rangle = \zeta^{\,r}|r\rangle.
\end{equation}
Then $\mathbf{Z}\mathbf{X}=\zeta\,\mathbf{X}\mathbf{Z}$. For $(u,v)\in U$, define
\[
x(u,v):=u^{1/m}\mathbf{X},\qquad y(u,v):=v^{1/m}\mathbf{Z}.
\]
A direct computation gives
\[
x(u,v)^{m}=u\,\mathrm{id}_{\mathcal H},\qquad y(u,v)^{m}=v\,\mathrm{id}_{\mathcal H},\qquad 
y(u,v)\,x(u,v)=\zeta\,x(u,v)\,y(u,v),
\]
corresponding to the setup defined by Definition~\ref{def:symbol-azumaya}. Thus, on $U$, the generators $x,y$ of the symbol algebra are realized as the $(u,v)$-dependent operators $x(u,v),y(u,v)\in \mathrm{End}(\mathcal H)$.

\vskip0.2cm
Let $(u_{0},v_{0})\in U$ be a base point and consider the loops
\[
\gamma_{u}(t)=(u_{0}e^{2\pi it},\,v_{0}),\qquad 
\gamma_{v}(t)=(u_{0},\,v_{0}e^{2\pi it})\qquad (0\le t\le 1).
\]
Analytic continuation of the chosen branches yields
\[
u^{1/m}\longmapsto \zeta\,u^{1/m}\quad \text{along }\gamma_{u},\qquad
v^{1/m}\longmapsto \zeta\,v^{1/m}\quad \text{along }\gamma_{v}.
\]
Hence $x(u,v)$ picks up a factor $\zeta$ along $\gamma_{u}$, and $y(u,v)$ picks up a factor $\zeta$ along $\gamma_{v}$. These phase changes are absorbed by conjugation:
\[
y\,x\,y^{-1}=\zeta\,x,\qquad x^{-1}y\,x=\zeta\,y,
\]
so the corresponding projective monodromies can be taken as
\[
g_{u}=[\,y(u_{0},v_{0})\,]\in \PGL(\mathcal H),\qquad 
g_{v}=[\,x(u_{0},v_{0})^{-1}\,]\in \PGL(\mathcal H).
\]
Here $[\cdot]$ denotes the class in $\PGL(\mathcal H)=\GL(\mathcal H)/\mathbb{C}^{\times}$.

Moreover, the commutator of lifts in $\GL(\mathcal H)$ is a scalar:
\[
g_{u}g_{v}g_{u}^{-1}g_{v}^{-1}
\ \sim\ 
(\mathbf{Z})(\mathbf{X}^{-1})(\mathbf{Z}^{-1})(\mathbf{X})
=\zeta^{-1}\mathrm{id}_{\mathcal H}.
\]
Thus the monodromies commute in $\PGL(\mathcal H)$ but have a nontrivial scalar discrepancy of order $m$ in $\GL(\mathcal H)$, which is the local manifestation of the Brauer twisting.

\vskip0.2cm
We explicitly construct a local bipartite identification $H\simeq \mathbb{C}^{p}\otimes \mathbb{C}^{p}$. Write each $r\in \mathbb{Z}/m\mathbb{Z}$ uniquely as
\[
r=a+pb,\qquad a,b\in\{0,1,\dots,p-1\}.
\]
This gives a local identification
\[
\mathcal H \ \cong\ \mathcal H_{A}\otimes \mathcal H_{B},\qquad \mathcal H_{A}\cong \mathbb{C}^{p},\ \mathcal H_{B}\cong \mathbb{C}^{p},
\]
by sending $|a\rangle_{A}\otimes |b\rangle_{B}\leftrightarrow |a+pb\rangle$. In this tensor-product picture, the projective monodromy around $\gamma_{v}$ acts as $\mathbf{X}^{-1}$ (up to an overall phase), so
\[
g_{v}\sim \mathbf{X}^{-1}.
\]
Its action on the tensor basis is
\[
\mathbf{X}^{-1}|a,b\rangle =
\begin{cases}
|a-1,\ b\rangle & (a\neq 0),\\[4pt]
|p-1,\ b-1\rangle & (a=0),
\end{cases}
\]
hence the action on the tensor basis depends on $a$, so it cannot be of the form $U_{A}\otimes U_{B}$.

\vskip0.2cm
Now consider the product state
\[
|\psi\rangle := (|0\rangle_{A}+|1\rangle_{A})\otimes |0\rangle_{B}.
\]
Applying $\mathbf{X}^{-1}$ gives
\[
\mathbf{X}^{-1}|\psi\rangle
=
|0\rangle_{A}\otimes |0\rangle_{B}+|p-1\rangle_{A}\otimes |p-1\rangle_{B}.
\]
This state has Schmidt rank $2$ (its coefficient matrix has two independent diagonal entries), hence it is entangled. Therefore the monodromy $g_{v}$ is \emph{entangling} with respect to the local bipartite type $(p,p)$. Thus $g_{v}$ does not lie in the stabilizer (the local-operation group) $\PGL_{p}\times \PGL_{p}$.

\begin{remark}[The case $p=2$]
When $p=2$, the above computation gives
\[
(|0\rangle+|1\rangle)\otimes |0\rangle \ \longmapsto\ |00\rangle+|11\rangle,
\]
which is (up to normalization) the Bell state. Thus the gauge monodromy generated by the $(u,v)$-twist contains an entangling component.
\end{remark}

\begin{theorem}[Universality of $p=2$ monodromy gate]
\label{thm:p2-monodromy-CNOT}
Let $\mathcal{X}:=(\mathbb{C}^\times)^2$ with coordinates $(u,v)$, set $m=4$ and $\zeta=e^{2\pi i/m}=i$, and let $\mathbf{X},\mathbf{Z}\in \GL(\mathcal H)$ ($\mathcal H=\mathbb{C}^4$) be the Weyl operators as before. Then:
\begin{enumerate}
\item[(i)] $g_u$ is local: in $\PGL(\mathcal H_A\otimes \mathcal H_B)$ one has $\mathbf{Z}\equiv S_A\otimes Z_B$, where $S_A=\mathrm{diag}(1,i)$ and $Z_B=\mathrm{diag}(1,-1)$.
\item[(ii)] $g_v$ is entangling and CNOT-class: in $\PGL(\mathcal H_A\otimes \mathcal H_B)$,
\[
\mathbf{X}^{-1}\equiv \mathrm{CNOT}_{A, B}\,(X_A\otimes I_B)
\]
(projective equality), hence $g_v$ is locally equivalent to $\mathrm{CNOT}$ (Controlled-NOT).
\item[(iii)] Consequently, if one can implement a universal single-qubit gate set on each factor and the monodromy gate $g_v$, then these gates generate a dense subgroup of $\PU(\mathcal H_A\otimes \mathcal H_B)\cong \PU(4)$.
\end{enumerate}
\end{theorem}

\begin{proof}
Along $\gamma_u$ the branch change $u^{1/4}\mapsto \zeta\,u^{1/4}$ is absorbed projectively by conjugation using $\mathbf{ZXZ}^{-1}=\zeta \mathbf{X}$, giving $g_u=[\mathbf{Z}]$. Similarly, along $\gamma_v$, one gets $g_v=[\mathbf{X}^{-1}]$. Under $|a,b\rangle\leftrightarrow |a+2b\rangle$,
\[
\mathbf{Z}|a,b\rangle=i^{a+2b}|a,b\rangle=i^a(-1)^b|a,b\rangle,
\]
so $\mathbf{Z}\equiv S_A\otimes Z_B$ in $\PGL$. A direct basis check shows $\mathbf{X}^{-1}$ agrees with $\mathrm{CNOT}_{A,B}(X_A\otimes I_B)$ in $\PGL$, hence $g_v$ is locally equivalent to $\mathrm{CNOT}$. Finally, it is standard that a gate set consisting of $\mathrm{CNOT}$ and universal single-qubit gates is universal on two qubits (up to phase), and local equivalence is implemented by single-qubit gates, proving (iii).
\end{proof}

\begin{proposition}[Monodromy and failure of reduction to $G_{\mathbf d}$]
\label{prop:entangling-monodromy-vs-reduction}
Let $P\to X$ be a $\PGL_n$--torsor and fix $\mathbf d$ with stabilizer $G_{\mathbf d}\subset \PGL_n$. Suppose $U\to X$ is an fppf cover on which $P$ is trivial, so that the associated bundle $SB(\mathcal A)|_U\simeq \PP^{n-1}\times U$ carries the standard Segre subscheme $\Sigma_{\mathbf d}\times U$.

Then $P$ admits a reduction to $G_{\mathbf d}$ if and only if, after refining the cover and changing the local trivializations if necessary, the transition functions can be chosen with values in $G_{\mathbf d}$.
Thus, for a fixed trivialization, a transition element $g\in \PGL_n$ is ``non-entangling'' (for type $\mathbf d$) precisely when $g(\Sigma_{\mathbf d})=\Sigma_{\mathbf d}$.
In particular, if some transition element lies outside $G_{\mathbf d}$, then it sends some point of $\Sigma_{\mathbf d}$ outside $\Sigma_{\mathbf d}$.
\end{proposition}

\begin{proof}
This is the usual local descent criterion for reduction of structure group.  Refining the cover and changing local trivializations corresponds to replacing the transition cocycle by a cohomologous one.  Such a cocycle can be chosen with values in $G_{\mathbf d}$ exactly when the torsor has a reduction to $G_{\mathbf d}$.  The final assertion follows from the definition of the stabilizer $G_{\mathbf d}$.
\end{proof}

\begin{remark}[Entangling holonomic gate]
From the perspective of quantum computing, Proposition~\ref{prop:entangling-monodromy-vs-reduction} can be interpreted as follows. Choose an open set $U\subset X$ over which a $\mathbf d$-subsystem structure is trivialized, so that product states form the standard product-state locus $\Sigma_{\mathbf d}\subset \mathbb{P}^{n-1}$ on each fiber. For a loop $\gamma$ based in $U$, let $\rho(\gamma)\in \PGL_n$ denote the associated projective monodromy.
Then:
\[
\rho(\gamma)\in G_{\mathbf d}\quad \Longleftrightarrow\quad \rho(\gamma)(\Sigma_{\mathbf d})=\Sigma_{\mathbf d}.
\]
In particular, if $\rho(\gamma)\notin G_{\mathbf d}$, then $\gamma$ acts as an entangling holonomic gate: there exists a product state $[\psi]\in \Sigma_{\mathbf d}$ such that $\rho(\gamma)[\psi]\notin \Sigma_{\mathbf d}$.
\end{remark}

\begin{remark}[Brauer obstruction and monodromy outside $G_{(2,2)}$]
In the Kummer-symbol setting on $\mathcal X=(\mathbb{C}^\times)^2$, the Brauer class of the degree-$4$ symbol algebra has period $4$, while a global $(2,2)$ subsystem structure would force the class to be $2$-torsion. Consequently no global reduction to $G_{(2,2)}$ exists, and any locally chosen tensor decomposition must acquire nonlocal monodromy around some loop. Theorem~\ref{thm:p2-monodromy-CNOT} makes this mechanism explicit: the $v$-loop produces a monodromy in the $\mathrm{CNOT}$-type, hence an entangling operation relative to the local $(2,2)$ splitting. From the Hilbert-scheme viewpoint, this monodromy is exactly the monodromy of the point in the subsystem-structure locus (Definition \ref{def:segre-hilbert-locus}) rather than a symmetry inside the stabilizer.
\end{remark}

\begin{example}[Twisted $8$-level system on $X=\Gm^2$ with locally multiple factorizations]\label{ex:8-level} 
Let \[ X:=\Gm^2=\Spec \mathbb C[u^{\pm1},v^{\pm1}],\qquad n=8, \] so each fiber of the pure state space is (locally) $\PP^{7}$.

Fix a primitive $8$th root $\zeta:=e^{2\pi i/8}$ and consider the $\cO_X$-algebra \[ A:=A(u,v)_8 :=\cO_X\langle x,y\rangle\Big/\big(x^8=u,\;y^8=v,\;yx=\zeta\,xy\big). \] Let $\pi:SB(A)\to X$ be the associated Severi--Brauer scheme (family of pure state spaces). 

On a simply connected analytic open $U\subset X$ choose branches $u^{1/8},v^{1/8}$. Let $\mathcal H:=\mathbb C^8$ with basis $\{|r\rangle\}_{r\in\Z/8\Z}$ and define $\mathbf{X}$ and $\mathbf{Z}$ as eq.~\eqref{eq:gate}.
Then 
\[ 
x(u,v):=u^{1/8}\mathbf{X},\qquad y(u,v):=v^{1/8}\mathbf{Z} 
\] 
gives a local representation, hence 
\[ 
A|_U\simeq \End(\mathcal H),\qquad SB(A)|_U\simeq \PP(\mathcal H)\times U. 
\] 

On $U$ one may identify $\mathcal H\simeq\mathbb C^8$ with either \[ \text{(A) }\mathcal H\simeq \mathbb C^2\otimes\mathbb C^2\otimes\mathbb C^2 \quad(d_3=(2,2,2)), \] via binary expansion $r=a+2b+4c$ $(a,b,c\in\{0,1\})$, i.e. $|a\rangle_A\otimes|b\rangle_B\otimes|c\rangle_C \leftrightarrow |r\rangle$, or \[ \text{(B) }\mathcal H\simeq \mathbb C^4\otimes\mathbb C^2\quad(d_2=(4,2)), \] via $r=a+4b$ $(a\in\{0,1,2,3\},\,b\in\{0,1\})$, i.e. $|a\rangle_A\otimes|b\rangle_B\leftrightarrow |r\rangle$. 

Let the fundamental loops be \[ \gamma_u(t)=(u_0e^{2\pi it},v_0),\qquad \gamma_v(t)=(u_0,v_0e^{2\pi it}). \] Analytic continuation multiplies $u^{1/8}$ or $v^{1/8}$ by $\zeta$. Compensating by conjugation yields projective monodromies \[ g_u=[\mathbf{Z}]\in \PGL(\mathcal H),\qquad g_v=[\mathbf{X}^{-1}]\in \PGL(\mathcal H). \] 

The element $g_v$ is not local. It sends product states to entangled states.  To see this, we consider the following two cases (a) and (b).

\emph{(a) Under the $3$-qubit factorization $d_3=(2,2,2)$.}\; Take a local product state \[ |\psi\rangle=(|0\rangle_A+|1\rangle_A)\otimes|0\rangle_B\otimes|0\rangle_C \;\;\leftrightarrow\;\; |0\rangle+|1\rangle. \] Since $\mathbf{X}^{-1}|0\rangle=|7\rangle$ and $\mathbf{X}^{-1}|1\rangle=|0\rangle$, \[ g_v|\psi\rangle \sim |7\rangle+|0\rangle \;\;\leftrightarrow\;\;|111\rangle+|000\rangle, \] which is the GHZ state (entangled). 

\emph{(b) Under the bipartite factorization $d_2=(4,2)$.}\; Take \[ |\phi\rangle=(|0\rangle_A+|1\rangle_A)\otimes|0\rangle_B \;\;\leftrightarrow\;\; |0\rangle+|1\rangle. \] Then again $g_v|\phi\rangle\sim |7\rangle+|0\rangle$ corresponds to \[ |0,0\rangle+|3,1\rangle \in \mathbb C^4\otimes\mathbb C^2, \] which has Schmidt rank $\ge 2$, hence is entangled. 

We now turn to the
\textbf{Brauer-period obstruction to global subsystem structures.}\;
Let $\beta:=[A]\in \Br(X)$. By Lemma~\ref{lem:symbol-period-m} we have $\per(\beta)=8$. For a subsystem type $d=(d_1,\dots,d_s)$ set $\ell=\mathrm{lcm}(d_1,\dots,d_s)$. Then \[ d_3=(2,2,2)\Rightarrow \ell=2,\qquad d_2=(4,2)\Rightarrow \ell=4, \] so $\per(\beta)=8\nmid \ell$ in both cases. Hence, by Proposition~\ref{prop:lcm-period-obstruction}, $\beta\notin \Br_{d_3}(X)$ and $\beta\notin \Br_{d_2}(X)$. By Theorem~\ref{thm:reduction-criterion}, the associated $\PGL_{8}$-torsor admits no reduction to $G_{d_3}$ nor $G_{d_2}$. Therefore the subsystem-structure loci $\Hilb^{\Sigma_{d_i}}(SB(A)/X)\simeq P/G_{d_i}$ admit no section $X\to P/G_{d_i}$ for $i=2,3$. 
\end{example}

\subsection{A twisted but $(p,p)$-reducible symbol tensor product on $\Gm^2$}

Throughout this subsection let $p$ be a prime, assume $\mathrm{char}(k)\nmid p$ and $k$ contains a primitive $p$-th root of unity $\zeta_p$.
Set
\[
X=\Gm^2=\Spec k[u^{\pm1},v^{\pm1}].
\]

\begin{definition}\label{def:pp_symbol_tensor}
Fix $a,b\in k^\times$. Let $A_u$ (resp.\ $A_v$) be the sheaf of $\cO_X$-algebras generated by $x_u,y_u$ (resp.\ $x_v,y_v$) with relations
\[
x_u^p=u,\quad y_u^p=a,\quad y_u x_u=\zeta_p x_u y_u,
\qquad
x_v^p=v,\quad y_v^p=b,\quad y_v x_v=\zeta_p x_v y_v.
\]
Define the degree-$p^2$ Azumaya algebra
\[
A := A_u\otimes_{\cO_X} A_v.
\]
\end{definition}

\begin{proposition}\label{prop:pp_symbol_tensor_reducible}
The algebra $A$ is Azumaya of degree $p^2$ on $X$. Its Brauer class $\beta:=[A]\in\Br(X)$ satisfies $\beta\in\Br(X)[p]$. Moreover, \ $A$ admits a global $(p,p)$-subsystem structure. Hence its Brauer class $\beta$ lies in $\Br_{(p,p)}(X)$. 
\end{proposition}

\begin{proof}
Azumaya-ness is checked after the Kummer cover adjoining $p$-th roots of $u$ and $v$: over $X'=\Spec k[u^{\pm 1/p},v^{\pm 1/p}]$ one has $A_u|_{X'}\simeq \Mat_p(\cO_{X'})$ and $A_v|_{X'}\simeq \Mat_p(\cO_{X'})$, hence $A|_{X'}\simeq \Mat_{p^2}(\cO_{X'})$. Thus $A$ is Azumaya of degree $p^2$.

Since $[A]=[A_u]+[A_v]$ in $\Br(X)$ and each of $[A_u],[A_v]$ is $p$-torsion, we have $\beta\in\Br(X)[p]$.

For reducibility, let $P_u\to X$ and $P_v\to X$ be the $\PGL_p$-torsors associated to $A_u$ and $A_v$. Then $P_u\times_X P_v$ is a $\PGL_p\times\PGL_p$-torsor. Pushing it forward along the tensor representation $\PGL_p\times\PGL_p\to \PGL_{p^2}$ yields the $\PGL_{p^2}$-torsor associated to $A$.
Since the image lies in the stabilizer $G_{(p,p)}\subset \PGL_{p^2}$, this exhibits a $G_{(p,p)}$-reduction of $P(A)$, hence a $(p,p)$-subsystem structure.
\end{proof}

\begin{remark}\label{rem:contrast_kummer}
Compare with the Kummer $p^2$-symbol class on the same base $X$: there one has $\mathrm{per}=p^2$, hence the torsion obstruction rules out
membership in $\Br_{(p,p)}(X)$, although $SB(A)\to X$ exists. Proposition~\ref{prop:pp_symbol_tensor_reducible} shows that on $X=\Gm^2$ there are also
nontrivial twisted classes of degree $p^2$ that do globalize a $(p,p)$-subsystem.
\end{remark}

\subsection{A quantum information system and a \emph{non-entangling} monodromy}

Fix a prime $p$ and put
\[
\mathcal{X} = (\mathbb{C}^{\times})^{2}, \qquad (u,v)\in \mathcal{X},
\]
and let
\[
\zeta := e^{2\pi i/p}.
\]
Fix nonzero constants $a,b\in \mathbb{C}^{\times}$ and choose once and for all $p$-th roots $\tilde{\alpha}:=a^{1/p}$ and $\tilde{\beta}:=b^{1/p}$.
We describe an explicit local model realizing the symbol tensor product algebra of Definition~\ref{def:pp_symbol_tensor} and show that the resulting gauge monodromies are \emph{non-entangling} with respect to the bipartite type $(p,p)$.

\vskip0.2cm
\noindent
\textbf{Step 1: Local trivialization and a parameter-dependent representation.}
Choose a simply connected open set $U\subset \mathcal{X}$ on which we can pick analytic branches of the $p$-th roots $u^{1/p}$ and $v^{1/p}$. Fix Hilbert spaces
\[
\mathcal H_A := \mathbb{C}^{p}, \qquad \mathcal H_B := \mathbb{C}^{p},\qquad \mathcal H := \mathcal H_A\otimes \mathcal H_B,
\]
with computational bases $\{|r\rangle_A\}_{r\in\mathbb{Z}/p\mathbb{Z}}$ and $\{|s\rangle_B\}_{s\in\mathbb{Z}/p\mathbb{Z}}$.

Define Weyl (shift/clock) operators on each factor by
\[
\mathbf{X}_A|r\rangle_A = |r+1\rangle_A,\qquad 
\mathbf{Z}_A|r\rangle_A = \zeta^{\,r}|r\rangle_A,
\]
\[
\mathbf{X}_B|s\rangle_B = |s+1\rangle_B,\qquad 
\mathbf{Z}_B|s\rangle_B = \zeta^{\,s}|s\rangle_B.
\]
Then $\mathbf{Z}_A\mathbf{X}_A=\zeta\,\mathbf{X}_A\mathbf{Z}_A$ and
$\mathbf{Z}_B\mathbf{X}_B=\zeta\,\mathbf{X}_B\mathbf{Z}_B$.

For $(u,v)\in U$, define operators on $\mathcal H$ by
\[
x_u(u,v):=u^{1/p}\,(\mathbf{X}_A\otimes \mathrm{id}_{\mathcal H_B}),\qquad
y_u(u,v):=\tilde{\alpha}\,(\mathbf{Z}_A\otimes \mathrm{id}_{\mathcal H_B}),
\]
\[
x_v(u,v):=v^{1/p}\,(\mathrm{id}_{\mathcal H_A}\otimes \mathbf{X}_B),\qquad
y_v(u,v):=\tilde{\beta}\,(\mathrm{id}_{\mathcal H_A}\otimes \mathbf{Z}_B).
\]
A direct computation gives
\[
x_u(u,v)^{p}=u\,\mathrm{id}_{\mathcal H},\qquad y_u(u,v)^{p}=a\,\mathrm{id}_{\mathcal H},\qquad
y_u(u,v)\,x_u(u,v)=\zeta\,x_u(u,v)\,y_u(u,v),
\]
and similarly
\[
x_v(u,v)^{p}=v\,\mathrm{id}_{\mathcal H},\qquad y_v(u,v)^{p}=b\,\mathrm{id}_{\mathcal H},\qquad
y_v(u,v)\,x_v(u,v)=\zeta\,x_v(u,v)\,y_v(u,v).
\]
Moreover, the $u$-generators commute with the $v$-generators since they act on different tensor factors. Thus, on $U$, the generators of
\[
A = A_u\otimes_{\cO_{\mathcal{X}}} A_v
\]
from Definition~\ref{def:pp_symbol_tensor} are realized as the $(u,v)$-dependent operators $\{x_u,y_u,x_v,y_v\}\subset \End(\mathcal H)$.

\vskip0.2cm
\noindent
\textbf{Step 2: Monodromy along the basic loops.}
Let $(u_0,v_0)\in U$ be a base point and consider the loops
\[
\gamma_{u}(t)=(u_{0}e^{2\pi it},\,v_{0}),\qquad 
\gamma_{v}(t)=(u_{0},\,v_{0}e^{2\pi it})\qquad (0\le t\le 1).
\]
Analytic continuation of the chosen branches yields
\[
u^{1/p}\longmapsto \zeta\,u^{1/p}\quad \text{along }\gamma_{u},\qquad
v^{1/p}\longmapsto \zeta\,v^{1/p}\quad \text{along }\gamma_{v}.
\]
Hence $x_u(u,v)$ picks up a factor $\zeta$ along $\gamma_u$, and $x_v(u,v)$ picks up a factor $\zeta$ along $\gamma_v$. These phase changes are absorbed by conjugation:
\[
y_u\,x_u\,y_u^{-1}=\zeta\,x_u,\qquad y_v\,x_v\,y_v^{-1}=\zeta\,x_v.
\]
Therefore the corresponding projective monodromies can be taken as
\[
g_{u}=[\,y_u(u_0,v_0)\,]\in \PGL(\mathcal H),\qquad
g_{v}=[\,y_v(u_0,v_0)\,]\in \PGL(\mathcal H).
\]
(Overall scalars are physically invisible on pure states, hence we work in $\PGL(\mathcal H)$.) Since $y_u$ and $y_v$ act on different tensor factors, their lifts already commute in $\GL(\mathcal H)$:
\[
g_u g_v g_u^{-1} g_v^{-1}
\ \sim\
(\mathbf{Z}_A\otimes \mathrm{id})(\mathrm{id}\otimes \mathbf{Z}_B)
(\mathbf{Z}_A^{-1}\otimes \mathrm{id})(\mathrm{id}\otimes \mathbf{Z}_B^{-1})
= \mathrm{id}_{\mathcal H}.
\]
Thus these basic monodromies remain local.

\vskip0.2cm
\noindent
\textbf{Step 3: Non-entangling nature with respect to the bipartite type $(p,p)$.}
In the tensor-product picture $\mathcal H=\mathcal H_A\otimes \mathcal H_B$, the monodromies are explicitly local:
\[
g_u \sim \mathbf{Z}_A\otimes \mathrm{id}_{\mathcal H_B},\qquad
g_v \sim \mathrm{id}_{\mathcal H_A}\otimes \mathbf{Z}_B.
\]
Hence $g_u$ and $g_v$ lie in the local-operation subgroup $\PGL(\mathcal H_A)\times \PGL(\mathcal H_B)\subset \PGL(\mathcal H)$, so they preserve the product-state locus and cannot generate entanglement from a product state.

For instance, consider the product state
\[
|\psi\rangle := (|0\rangle_A+|1\rangle_A)\otimes |0\rangle_B.
\]
Applying $g_u$ gives
\[
g_u|\psi\rangle
\ \sim\
(\mathbf{Z}_A\otimes \mathrm{id})\bigl((|0\rangle_A+|1\rangle_A)\otimes |0\rangle_B\bigr)
=
(|0\rangle_A+\zeta\,|1\rangle_A)\otimes |0\rangle_B,
\]
which is still a product state (Schmidt rank $1$). Similarly,
\[
g_v|\psi\rangle \ \sim\ (|0\rangle_A+|1\rangle_A)\otimes (\mathbf{Z}_B|0\rangle_B)
=
(|0\rangle_A+|1\rangle_A)\otimes |0\rangle_B
\]
(up to a phase), again a product state. Thus the gauge monodromy in this model is \emph{non-entangling} with respect to $(p,p)$.

\begin{remark}[The case $p=2$]
When $p=2$, one has $\mathcal H=\mathbb{C}^2\otimes \mathbb{C}^2$ and the monodromies are
\[
g_u \sim \sigma_z\otimes \mathrm{id},\qquad g_v \sim \mathrm{id}\otimes \sigma_z
\]
with Pauli $\sigma_z,\sigma_x$, so they are manifestly local phase flips and (in contrast to the example in \S\ref{sec:entangling_monodromy}) never produce Bell-type entanglement.
\end{remark}

\subsection{Period is not sufficient: indecomposable classes of exponent $p$ and index $p^2$}

\begin{lemma}\label{lem:Brpp_quadratic_decomposable}
Let $F$ be a field and $d=(p,p)$. If $\beta\in \Br_{(p,p)}(F)$, then there exists a separable extension $L/F$ of degree dividing $2$ such that $\beta_L$ is represented by a tensor product of two degree-$p$ central simple $L$-algebras.
\end{lemma}

\begin{proof}
By definition, $\beta$ comes from a $G_{(p,p)}$-reduction of a $\PGL_{p^2}$-torsor. When $d=(p,p)$, the stabilizer has a component group $\pi_0(G_{(p,p)})\simeq \mathfrak{S}_2$, so the reduction determines (and is determined by) a possibly nontrivial $\mathfrak{S}_2$-torsor. After base change to its splitting field $L/F$ (degree $\le 2$), the reduction lands in the identity component, which is the image of $\PGL_p\times\PGL_p$ under the tensor representation. Over $L$, this yields an ordered bipartite reduction, hence a decomposition of the associated degree-$p^2$ central simple algebra as a tensor product of two degree-$p$ algebras.
\end{proof}

\begin{theorem}\label{thm:indec_exp_p_ind_p2_not_Brpp}
Let $p$ be an odd prime. There exists a field $F$ of characteristic $0$ and a Brauer class $\beta\in \Br(F)$ such that
\[
\exp(\beta)=p,\qquad \ind(\beta)=p^2,
\]
and $\beta$ remains indecomposable after any prime-to-$p$ field extension. In particular, $\beta\notin \Br_{(p,p)}(F)$ although $\beta\in \Br(F)[p]$.
\end{theorem}

\begin{proof}
The existence of such $\beta$ (indecomposable of exponent $p$ and index $p^2$, and stable under prime-to-$p$ extensions) is provided by the theory of generic abelian crossed products (see \cite[Theorem~3.16]{McKinniePrimeToP2007}). Assuming $\beta\in \Br_{(p,p)}(F)$,  Lemma~\ref{lem:Brpp_quadratic_decomposable} would imply that $\beta$ becomes decomposable after a quadratic extension of $F$, which is a prime-to-$p$ extension. This contradicts the prime-to-$p$ indecomposability.
\end{proof}

\begin{remark}\label{rem:torsion_not_sufficient}
The torsion obstruction $\Br_{(p,p)}(F)\subset \Br(F)[p]$ is therefore sharp but not sufficient: there exist classes of exponent $p$ which cannot arise from any global $(p,p)$-subsystem structure.
\end{remark}

\subsection{A numerical split obstruction on curves (including elliptic curves)}

Let $C$ be a smooth projective curve over an algebraically closed field $k$. Since $\Br(C)=0$, every $\PP^{n-1}$-bundle over $C$ is of the form $\PP(V)$ for some rank-$n$ vector bundle $V$, unique up to twisting by a line bundle.

\begin{definition}\label{def:deg_mod_n_projective_bundle}
For a $\PP^{n-1}$-bundle $P\simeq \PP(V)\to C$, define
\[
\deg(P)\in \mathbb Z/n\mathbb Z,\qquad \deg(P):=\deg(V)\ \mathrm{mod}\ n.
\]
This is well-defined because $\deg(V\otimes L)=\deg(V)+n\deg(L)$.
\end{definition}

\begin{proposition}\label{prop:curve_gcd_degree_obstruction}
Let $d=(d_A,d_B)$ with $n=d_A d_B$ and set $g:=\gcd(d_A,d_B)$. If a $\PP^{n-1}$-bundle $P\to C$ admits a $d$-subsystem structure (equivalently, is $d$-reducible), then $\deg(P)\equiv 0 \ (\mathrm{mod}\ g)$ in $\mathbb Z/g\mathbb Z$.
\end{proposition}

\begin{proof}
Choose $V$ with $P\simeq \PP(V)$. If $P$ is $d$-reducible, then (for $d_A\neq d_B$, or after an étale cover removing the permutation ambiguity when $d_A=d_B$) one has an isomorphism
\[
V \;\simeq\; (V_A\otimes V_B)\otimes L
\]
for some rank-$d_A$ bundle $V_A$, rank-$d_B$ bundle $V_B$, and a line bundle $L$. Taking degrees gives
\[
\deg(V)\equiv \deg(V_A\otimes V_B) \equiv d_B\deg(V_A)+d_A\deg(V_B)\equiv 0\pmod g,
\]
and twisting by $L$ does not change $\deg(V)\pmod g$ since $n\equiv 0\pmod g$. Thus $\deg(P)\equiv 0\pmod g$.
\end{proof}

\begin{example}\label{ex:elliptic_degree_one_not_reducible}
Let $E$ be an elliptic curve, let $n=p^2$ and $d=(p,p)$ with $p$ prime. Pick a line bundle $M$ of degree $1$ and set $V:=\cO_E^{\oplus (n-1)}\oplus M$. Then $\PP(V)\to E$ has $\deg(\PP(V))\equiv 1\ (\mathrm{mod}\ p)$, hence it is not $(p,p)$-reducible by Proposition~\ref{prop:curve_gcd_degree_obstruction}, although it is split (Brauer class $0$).
\end{example}

\subsection{Azumaya algebras from moduli of vector bundles}

Let $C$ be a smooth projective curve over an algebraically closed field $k$. Fix integers $r\ge 2$ and a line bundle $\mathcal L$ on $C$ of degree $d$.
Let $M:=M_C(r,\mathcal L)$ denote the moduli space of stable vector bundles on $C$ of rank $r$ with fixed determinant $\mathcal L$.

\begin{theorem}[Biswas--Sengupta \cite{2018arXiv180402494B}]\label{thm:BS_Brauer_moduli}
The Brauer group $\Br(M)$ is cyclic of order $\gcd(r,d)$. Moreover, the Brauer class of the universal projective bundle restricted to $\{x\}\times M$ (for any point $x\in C$) generates $\Br(M)$.
\end{theorem}

\begin{corollary}\label{cor:moduli_obstructs_pp}
Let $r=p^2$ with $p$ prime, and assume $\gcd(r,d)=p^2$ (e.g.\ $d\equiv 0\ (\mathrm{mod}\ p^2)$). Let $\alpha\in \Br(M)$ be the generator from Theorem~\ref{thm:BS_Brauer_moduli}. Then $\per(\alpha)=p^2$, hence for the bipartite type $(p,p)$ one has
\[
\alpha\notin \Br_{(p,p)}(M).
\]
Equivalently, although the Severi--Brauer fibration $SB(A_\alpha)\to M$ exists, the $(p,p)$-subsystem structure cannot be globalized on $M$.
\end{corollary}

\begin{proof}
By Theorem~\ref{thm:BS_Brauer_moduli}, $\per(\alpha)=\gcd(r,d)=p^2$. For $d=(p,p)$ one has $\ell=\mathrm{lcm}(p,p)=p$, hence $\Br_{(p,p)}(M)\subset \Br(M)[p]$ by the torsion obstruction. Therefore $\alpha\notin \Br_{(p,p)}(M)$.
\end{proof}

\subsection{Dependence on the torsor and split classification on $\PP^1$}
\subsubsection{Same Brauer class, different subsystem behavior}
\label{subsec:same-brauer-different-representatives}

Subsystem reducibility is a property of the torsor under $\PGL_n$, not only of the Brauer class. We state this explicitly.

\begin{proposition}[Subsystem reducibility is not a Brauer invariant]
\label{prop:not-brauer-invariant}
Fix a subsystem type $\mathbf d=(d_1,\dots,d_s)$ with $n=\prod_i d_i$ and assume that $\Sigma_{\mathbf d}\subset \PP^{n-1}$ is a proper subscheme (thus, at least two of the $d_i$ are $>1$). Then there exist a scheme $X$ and Azumaya algebras $\mathcal A,\mathcal A'$ of degree $n$ on $X$ such that
\[
  [\mathcal A]=[\mathcal A']\in \Br(X),
\]
but $\mathcal A$ admits a subsystem structure of type $\mathbf d$ whereas $\mathcal A'$ does not. In fact, one may take $X=\PP^1_k$ and $[\mathcal A]=[\mathcal A']=0$.
\end{proposition}

\begin{proof}
Take $X=\PP^1_k$, so $\Br(X)=0$.

\medskip
\noindent\textbf{Step 1:  A reducible representative.}
Let
\[
  E_{\mathrm{red}}:=\mathcal O_{\PP^1}^{\oplus n}.
\]
Then $\mathcal A:=\underline{\End}(E_{\mathrm{red}})$ is split, hence $[\mathcal A]=0$. Moreover $E_{\mathrm{red}}$ admits the evident tensor factorization
\[
  E_{\mathrm{red}}\ \cong\ \bigl(\mathcal O^{\oplus d_1}\bigr)\otimes\cdots\otimes\bigl(\mathcal O^{\oplus d_s}\bigr),
\]
so the associated $\PGL_n$--torsor $\PP(E_{\mathrm{red}})\to \PP^1$ reduces to the stabilizer $G_{\mathbf d}$. Thus $\mathcal A$ admits a subsystem structure of type $\mathbf d$.

\medskip
\noindent\textbf{Step 2: A non-reducible representative with the same Brauer class.}
Let
\[
  E_{\mathrm{irr}}:=\mathcal O_{\PP^1}^{\oplus (n-1)}\oplus \mathcal O_{\PP^1}(1),
  \qquad
  \mathcal A':=\underline{\End}(E_{\mathrm{irr}}).
\]
Again $\mathcal A'$ is split, so $[\mathcal A']=0=[\mathcal A]$.

We claim that $\mathcal A'$ admits \emph{no} globally compatible subsystem structure of type $\mathbf d$. Suppose to the contrary that the associated $\PGL_n$--torsor $\PP(E_{\mathrm{irr}})\to \PP^1$ admits a reduction to $G_{\mathbf d}$.

Since $k$ is algebraically closed, $\pi_1^{\mathrm{\acute et}}(\PP^1_k)=1$, hence every finite \'etale torsor on $\PP^1_k$ is trivial. In particular,
any component-group (permutation) part of $G_{\mathbf d}$ produces no monodromy. Thus (after a global choice of ordering of equal factors) the reduction may be viewed as a reduction to the identity component, which is the image of $\prod_i \PGL_{d_i}$.

Since $\Br(\PP^1)=0$, each induced $\PGL_{d_i}$--torsor lifts to a $\GL_{d_i}$--torsor, hence to a rank-$d_i$ vector bundle $F_i$. The tensor product representation then gives an identification of $\PGL_n$--torsors
\[
  \PP(E_{\mathrm{irr}})\ \simeq\ \PP(F_1\otimes\cdots\otimes F_s).
\]
By Lemma~\ref{lem:P1-projective-classification}, this implies
\[
  E_{\mathrm{irr}}\ \simeq\ (F_1\otimes\cdots\otimes F_s)\otimes L
\]
for some line bundle $L$ on $\PP^1$.

Now write the Grothendieck splitting of each $F_i$ as $F_i\simeq \bigoplus_{j=1}^{d_i}\mathcal O(a_{ij})$ with $a_{i1}\le \cdots\le a_{i d_i}$. Let $m_i$ be the multiplicity of the minimal degree $a_{i1}$ inside $\{a_{ij}\}_j$. Then the tensor product splits as a direct sum of line bundles $\mathcal O\!\left(\sum_i a_{i j_i}\right)$, and the minimal degree occurs with multiplicity
\[
  m\ :=\ \prod_{i=1}^s m_i.
\]
Twisting by $L$ shifts all degrees uniformly and does not change this multiplicity.

On the other hand, $E_{\mathrm{irr}}=\mathcal O^{\oplus(n-1)}\oplus \mathcal O(1)$ has normalized splitting type $\{0,\dots,0,1\}$, so the minimal degree occurs with multiplicity $n-1$. Hence we must have $\prod_i m_i=n-1$.

But each $m_i$ satisfies $1\le m_i\le d_i$, and since $\Sigma_{\mathbf d}$ is proper we have at least two factors $d_i>1$, so for every $i$ we have
\[
  \frac{n}{d_i}=\prod_{j\neq i} d_j\ \ge\ 2.
\]
If for some $i$ we had $m_i\le d_i-1$, then
\[
  \prod_{j=1}^s m_j\ \le\ (d_i-1)\prod_{j\neq i} d_j
  \;=\; n-\frac{n}{d_i}\ \le\ n-2,
\]
so it could not equal $n-1$. Therefore we must have $m_i=d_i$ for all $i$, and then $\prod_i m_i=n$, again impossible. This contradiction shows that no such decomposition of $E_{\mathrm{irr}}$ exists, hence $\PP(E_{\mathrm{irr}})$ does not reduce to $G_{\mathbf d}$ and $\mathcal A'$ admits no subsystem structure of type $\mathbf d$.

Thus $[\mathcal A]=[\mathcal A']=0$ but only $\mathcal A$ is subsystem-reducible of type $\mathbf d$, as claimed.
\end{proof}

\begin{remark}
Even for $\beta=0$, subsystem structures are not automatic. Thus the subsystem structure is additional data not determined by the Brauer class. Section \ref{sec:algebraic-entangling-obstructions} presents obstruction data beyond the Brauer class.
\end{remark}

\subsubsection{Split Brauer class on $\PP^1$: classification of subsystem reducibility}
\label{subsec:P1-classification}

Throughout this subsection let $k$ be algebraically closed of characteristic $0$ and set $X=\PP^1_k$. In particular $\Br(X)=0$.

\begin{lemma}[Projective bundles on $\PP^1$ and splitting types]
\label{lem:P1-projective-classification}
Let $P\to \PP^1$ be a principal $\PGL_n$--torsor. Then there exists a rank-$n$ vector bundle $E$ on $\PP^1$ such that $P\simeq \PP(E)$. Moreover, if $E$ and $E'$ are rank-$n$ bundles with $\PP(E)\simeq \PP(E')$, then $E'\simeq E\otimes L$ for some line bundle $L\in \Pic(\PP^1)$.

Consequently, after choosing any such lift $E$ and writing its Grothendieck splitting
\[
  E \;\cong\; \bigoplus_{m=1}^n \cO_{\PP^1}(a_m)
  \qquad (a_1\le \cdots \le a_n),
\]
the multiset $\{a_m-a_1\}_{m=1}^n$ depends only on the isomorphism class of $P$.
\end{lemma}

\begin{proof}
Since $\Br(\PP^1)=0$, the connecting map $H^1(\PP^1,\PGL_n)\to H^2(\PP^1,\Gm)=\Br(\PP^1)$ is trivial, so $H^1(\PP^1,\GL_n)\to H^1(\PP^1,\PGL_n)$ is surjective. Thus $P$ admits a lift to a $\GL_n$--torsor, so $P\simeq \PP(E)$ for some vector bundle $E$.

For the uniqueness statement, apply nonabelian cohomology to $1\to \Gm \to \GL_n \to \PGL_n \to 1$. The kernel of $H^1(\PP^1,\GL_n)\to H^1(\PP^1,\PGL_n)$ is the image of $H^1(\PP^1,\Gm)=\Pic(\PP^1)$. Concretely this says that two lifts differ by tensoring with a line bundle.

Finally, Grothendieck splitting holds for any vector bundle on $\PP^1$, and twisting by a line bundle shifts all $a_m$ by the same integer, so the normalized multiset $\{a_m-a_1\}$ is intrinsic to $P$.
\end{proof}

\begin{theorem}[A splitting-type criterion for bipartite tensor reducibility on $\PP^1$]
\label{thm:P1-bipartite-reducibility}
Fix a bipartite type $\mathbf d=(d_A,d_B)$ and set $n=d_A d_B$. Let $P\to \PP^1$ be a principal $\PGL_n$--torsor, and choose a lift $P\simeq \PP(E)$ as in Lemma~\ref{lem:P1-projective-classification}. Write
\[
  E \;\cong\; \bigoplus_{m=1}^n \cO_{\PP^1}(a_m)
  \qquad (a_1\le \cdots \le a_n).
\]

Then the following are equivalent:
\begin{enumerate}
\item $P$ is $G_{\mathbf d}$--reducible (admits a reduction of structure group to the stabilizer $G_{\mathbf d}\subset \PGL_n$), so the split Azumaya algebra $\underline{\End}(E)$ admits a subsystem structure of type $\mathbf d$. 
\item There exist vector bundles $F$ and $G$ on $\PP^1$ of ranks $d_A$ and $d_B$ and a line bundle $L$ such that
\[
  E \;\cong\; (F\otimes G)\otimes L.
\]
\item (\emph{Additive decomposition of splitting type}) There exist integers $b_1\le\cdots\le b_{d_A}$, $c_1\le\cdots\le c_{d_B}$ and an integer $t$ such that the multiset $\{a_m\}_{m=1}^n$ equals the multiset
\[
  \{\, b_i+c_j+t \mid 1\le i\le d_A,\ 1\le j\le d_B \,\}.
\]
\end{enumerate}

In particular, the $G_{\mathbf d}$--reducible locus inside the split Brauer fiber $\delta^{-1}(0)=H^1(\PP^1,\PGL_n)$ is a proper subset (unless $(d_A,d_B)=(1,n)$ or $(n,1)$), described by the purely numerical condition (3).
\end{theorem}

\begin{proof}
(2)$\Leftrightarrow$(3) is immediate from Grothendieck splitting on $\PP^1$: writing $F\simeq \oplus_i \cO(b_i)$, $G\simeq \oplus_j \cO(c_j)$ and $L\simeq \cO(t)$ gives
\[
  (F\otimes G)\otimes L \;\simeq\; \bigoplus_{i,j} \cO(b_i+c_j+t),
\]
and conversely any such multiset determines $F,G,L$.

(2)$\Rightarrow$(1): choose an fppf (in fact Zariski) cover $\{U_\alpha\}$ trivializing $F$, $G$, and $L$. On overlaps $U_{\alpha\beta}$ the transition functions of $E\simeq (F\otimes G)\otimes L$ are of the form
\[
  g^E_{\alpha\beta} \;=\; (g^F_{\alpha\beta}\otimes g^G_{\alpha\beta})\cdot \lambda_{\alpha\beta}
  \in \GL_n,
\]
where $\lambda_{\alpha\beta}\in \Gm$ is the scalar transition of $L$. Passing to $\PGL_n$ kills scalars, so the induced cocycle lies in the image of the tensor product subgroup, hence in $G_{\mathbf d}\subset \PGL_n$. This is exactly a reduction of $P$ to $G_{\mathbf d}$.

(1)$\Rightarrow$(2): since $\PP^1$ has trivial étale fundamental group, any finite étale torsor is trivial. In particular, if $d_A=d_B$ and $G_{\mathbf d}$ has a component group permuting equal Segre factors, there is no nontrivial monodromy, so we may (after a global choice of ordering) regard a $G_{\mathbf d}$--reduction as a reduction to the identity component, which is the image of $\PGL_{d_A}\times\PGL_{d_B}$.

Thus (1) yields principal $\PGL_{d_A}$-- and $\PGL_{d_B}$--torsors $Q_A$ and $Q_B$ on $\PP^1$ whose induced $\PGL_n$--torsor (via the tensor representation) is isomorphic to $P$. Because $\Br(\PP^1)=0$, both $Q_A$ and $Q_B$ lift to vector bundles $F$ and $G$ of ranks $d_A$ and $d_B$.
The induced $\PGL_n$--torsor is then $\PP(F\otimes G)$ by inspection of transition functions. Hence $\PP(E)\simeq \PP(F\otimes G)$, and Lemma~\ref{lem:P1-projective-classification} implies $E\simeq (F\otimes G)\otimes L$ for some line bundle $L$.
\end{proof}

\begin{corollary}[Complete criterion for $(2,2)$ on $\PP^1$]
\label{cor:P1-22-criterion}
Let $n=4$ and $\mathbf d=(2,2)$. Let $P\simeq \PP(E)\to \PP^1$ with $E\simeq \cO(a_1)\oplus\cO(a_2)\oplus\cO(a_3)\oplus\cO(a_4)$ and $a_1\le a_2\le a_3\le a_4$.

Then $P$ is $G_{(2,2)}$--reducible if and only if
\[
  a_1+a_4 \;=\; a_2+a_3.
\]
So, after twisting $E$ so that $a_1=0$, the normalized splitting type is
\[
  (0,\ b,\ c,\ b+c)\qquad \text{for some integers }0\le b\le c.
\]
When the condition holds, one can take explicitly
\[
  F=\cO(a_1)\oplus \cO(a_3),\qquad
  G=\cO(0)\oplus \cO(a_2-a_1),
\]
so that $E\simeq F\otimes G$ and the $(2,2)$ subsystem structure is induced from this factorization.
\end{corollary}

\begin{proof}
Apply Theorem~\ref{thm:P1-bipartite-reducibility} with $d_A=d_B=2$. In the $2\times 2$ case, an additive decomposition of the multiset $\{a_m\}$ is equivalent to the ``parallelogram'' condition $a_1+a_4=a_2+a_3$ (after ordering). The explicit $F,G$ are obtained by the constructive choice explained in the statement.
\end{proof}

\begin{remark}
Even though $\beta=0\in \Br_{(2,2)}(\PP^1)$ (so some representative admits a subsystem structure), Corollary~\ref{cor:P1-22-criterion} shows that many split representatives do not. For example, the splitting types $(0,0,1,1)$ satisfy $0+1=0+1$ and are $(2,2)$--reducible, whereas $(0,0,0,1)$ fail the condition and are not $(2,2)$--reducible. Thus, already over $\PP^1$ the existence of entanglement geometry is controlled by the isomorphism class of the torsor under $\PGL_n$, not by the Brauer class alone.
\end{remark}

\begin{example}[An explicit split example on $\mathbb P^1$ for $\mathbf d=(2,2)$]
\label{ex:P1-split-different}
Let $X=\PP^1_k$, $n=4$, and $\mathbf d=(2,2)$.

\smallskip\noindent
\textbf{(i) $(2,2)$--reducible.}
Take $E_{\mathrm{red}}:=\cO^{\oplus 2}\oplus \cO(1)^{\oplus 2}$, so that $a_1+a_4=a_2+a_3$ holds and $\PP(E_{\mathrm{red}})$ is $(2,2)$--reducible by
Corollary~\ref{cor:P1-22-criterion}.

\smallskip\noindent
\textbf{(ii) not $(2,2)$--reducible.}
Take $E_{\mathrm{irr}}:=\cO^{\oplus 3}\oplus \cO(1)$, for which $a_1+a_4\neq a_2+a_3$. Hence $\PP(E_{\mathrm{irr}})$ is not $(2,2)$--reducible by Corollary~\ref{cor:P1-22-criterion}, even though both bundles define split Azumaya algebras (Brauer class $0$).
\end{example}

\section{Algebraic entangling obstructions beyond the Brauer class}
\label{sec:algebraic-entangling-obstructions}

We now separate the Brauer class from the reduction data that it does not see.
The obstruction to a subsystem structure is not, in general, an abelian cohomology class.
It is first the problem of lifting the classifying map of the torsor under $\PGL_n$
through $BG_{\mathbf d}\to B\PGL_n$. Abelian classes enter only after linear
lift data or incidence data have been chosen. The central idea of the following derivation of entangling obstructions is applied to a study of quantum many-body systems \cite{Ikeda:2026yld}.

Throughout this section $BG=[\Spec k/G]$ denotes the classifying stack of a linear
algebraic group. Maps into $BG$ are taken in the fppf topology. If the base is stacky,
the same notation denotes the corresponding $2$-mapping stack.  When we write Chow groups
of $BG$, we use the Totaro--Edidin--Graham approximation, which agrees with the usual equivariant Chow
groups in this quotient stack setting \cite{Totaro1999Chow,EdidinGraham1998EquivariantChow}.
For nonabelian cohomology and gerbes we follow the standard language of
\cite{Giraud1971CohomologieNonAbelienne}. 

\subsection{The stack of reductions to $G_{\mathbf d}$}

Let $P\to X$ be the torsor under $\PGL_n$ attached to the Azumaya algebra $\mathcal A$, and let
$f_P:X\to B\PGL_n$ be its classifying map.  The full lifting problem controlling a locus of product states
of type $\mathbf d$ is
\[
  X \xrightarrow{f_P} B\PGL_n
  \qquad\text{through}\qquad
  BG_{\mathbf d}\longrightarrow B\PGL_n .
\]
\begin{definition}[Stack of reductions to $G_{\mathbf d}$]
\label{def:stack-reductions-Gd}
The stack of reductions of $P$ to $G_{\mathbf d}$ is
\[
  \operatorname{Red}_{\mathbf d}(P)
  := \operatorname{Map}_{/B\PGL_n}(X,BG_{\mathbf d}).
\]
Equivalently,
\[
  \operatorname{Red}_{\mathbf d}(P)\simeq \Gamma_X(P/G_{\mathbf d}),
\]
where $P/G_{\mathbf d}:=P\times^{\PGL_n}(\PGL_n/G_{\mathbf d})\to X$ is the associated homogeneous
fibration.  This is the stack version of the reduction functor
$\mathfrak{Red}_{\mathbf d}(\mathcal A)$ used above.
\end{definition}

We say that the projective obstruction vanishes if
\[
  \operatorname{Red}_{\mathbf d}(P)\neq\varnothing.
\]
This is a statement about the nonemptiness of a groupoid or stack, not the vanishing of a
preferred element of an abelian group.  The construction recovers exactly the subsystem
structures of Definition~\ref{def:subsystem-structure} and the Hilbert scheme locus description of
Theorem~\ref{thm:subsystem-moduli-hilbert}.

\begin{proposition}[Local cocycle criterion for reduction]
\label{prop:cech-reduction-criterion}
Let $\{U_i\to X\}$ be an fppf cover over which $P$ is represented by a cocycle with values in $\PGL_n$,
$g_{ij}:U_{ij}\to \PGL_n$.  Then $\operatorname{Red}_{\mathbf d}(P)$ is nonempty if and only if, after refining the
cover if necessary, there exist maps $a_i:U_i\to \PGL_n$ such that
\[
  a_i g_{ij} a_j^{-1}\in G_{\mathbf d}(U_{ij})
  \qquad\text{for all }i,j.
\]
\end{proposition}

\begin{proof}
This is the usual nonabelian descent criterion for reducing a torsor along a closed subgroup
$G_{\mathbf d}\subset \PGL_n$.  A choice of the $a_i$ changes the local trivializations of $P$; the
condition displayed says that the transformed transition functions glue to a torsor under $G_{\mathbf d}$.
Conversely, any reduction to $G_{\mathbf d}$ gives such adapted local trivializations.
\end{proof}

\begin{remark}[Labeled and unlabeled subsystems]
\label{rem:labeled-unlabeled-segre}
If the tensor factors are labeled, one may replace $G_{\mathbf d}$ by the identity component of $G_{\mathbf d}$, namely the image of $\prod_i \PGL_{d_i}$.  If equal-dimensional factors may be
permuted, the correct group is the full scheme-theoretic stabilizer, or equivalently the relevant
normalizer extension.  In the latter case the reduction stack gives finite monodromy of
the labels. Passing only to invariant cohomology may lose finite descent information; the
quotient stack or normalizer group should be used.
\end{remark}

\subsection{The Brauer class and residual obstructions with trivial Brauer class}

The Brauer class
\[
  \beta(P)=\delta[P]\in H^2_{\rm fppf}(X,\Gm)_{\rm tors}
\]
gives a necessary abelian condition for the reduction problem above.  Proposition~\ref{prop:lcm-period-obstruction}
shows that a subsystem structure of type $\mathbf d$ forces
\[
  \operatorname{per}(\beta(P))\mid \operatorname{lcm}(d_1,\ldots,d_s),
\]
but the converse fails in general, and even $\beta(P)=0$ does not force
$\operatorname{Red}_{\mathbf d}(P)$ to be nonempty.

The examples over $\PP^1$ make this residual obstruction explicit.  If
$P\simeq \PP(E)$ and
\[
  E\simeq \bigoplus_{m=1}^4\cO_{\PP^1}(a_m),
  \qquad a_1\le a_2\le a_3\le a_4,
\]
then Corollary~\ref{cor:P1-22-criterion} gives
\[
  P \text{ reduces to }G_{(2,2)}
  \quad\Longleftrightarrow\quad
  a_1+a_4=a_2+a_3.
\]
Hence the integer
\[
  \Delta_{22}(P):=a_1+a_4-a_2-a_3
\]
is a splitting type obstruction with trivial Brauer class.  It is not a cohomology class; it is a splitting-type
defect measuring the failure of a split projective bundle to be projectively tensor-decomposable
of type $(2,2)$. In \cite{Ikeda:2026fen}, this gives an explicit reason for an entangling phase beyond the conventional Berry phase.

More generally, in the bipartite case on $\PP^1$, Theorem~\ref{thm:P1-bipartite-reducibility}
says that a splitting type $\{a_m\}$ is reducible of type $(d_A,d_B)$ if and only if it admits an
additive decomposition
\[
  \{a_m\}=
  \{b_i+c_j+t\mid 1\le i\le d_A,
                     1\le j\le d_B\}
\]
for integers $b_i,c_j,t$.  This criterion should be viewed as a computable invariant associated
with the nonabelian stack of reductions to $G_{\mathbf d}$ rather than as a replacement for it.

\subsection{Linear and determinant lifting stacks}

The reduction stack above is defined at the projective level.  Determinant invariants require a linear lift, or a stack
on which the corresponding determinant objects are defined.  Let
\[
  \varphi:K\longrightarrow \PGL_n
\]
be an algebraic group homomorphism.

\begin{definition}[$K$ lifting stack]
\label{def:determinant-lift-stack}
The stack of $K$ lifts of $P$ is
\[
  \mathfrak L_K(P):=X\times_{B\PGL_n}BK.
\]
The groupoid of global $K$ lifts is
\[
  \operatorname{Lift}_K(P):=\Gamma_X(\mathfrak L_K(P))
  \simeq \operatorname{Map}_{/B\PGL_n}(X,BK).
\]
\end{definition}

For $K=\GL_n$ and the central extension
\[
  1\longrightarrow \Gm\longrightarrow \GL_n\longrightarrow \PGL_n\longrightarrow 1,
\]
the lifting stack $\mathfrak L_{\GL_n}(P)$ is a $\Gm$ gerbe whose class is the Brauer class
$\beta(P)$.  It is neutral exactly when $P$ comes from a vector bundle.  If it is not neutral,
determinant objects may still be formulated on the lifting stack or in a twisted sheaf category,
but they are not ordinary line bundles on $X$.

A further distinction appears after a reduction to $G_{\mathbf d}$ has been chosen.  Suppose, for simplicity,
that we are in the labeled case and that
\[
  G_{\mathbf d}^{\circ}\simeq \prod_i \PGL_{d_i}.
\]
A reduction $\rho\in\operatorname{Red}_{\mathbf d}(P)$ gives factor projective torsors
$P_i\to X$.  Their Brauer classes satisfy
\[
  \beta(P)=\sum_i \beta(P_i).
\]
Here the sum is the Brauer class of the tensor product of the factor projective torsors.  Even if
$\beta(P)=0$, the individual $\beta(P_i)$ need not vanish.  Therefore a projective tensor
factorization is weaker than a linear tensor decomposition.

\begin{definition}[Factor lifting stack]
\label{def:factor-lift-stack}
For a labeled reduction $\rho$ with factor projective torsors $P_i$, define the factor
lifting stack
\[
  \mathfrak L_{\rm fac}(\rho)
  :=\mathfrak L_{\GL_{d_1}}(P_1)\times_X\cdots\times_X\mathfrak L_{\GL_{d_s}}(P_s),
  \qquad
  \mathfrak L_{\GL_{d_i}}(P_i):=X\times_{B\PGL_{d_i}}B\GL_{d_i}.
\]
The groupoid of factor lifts is
\[
  \operatorname{Lift}_{\rm fac}(\rho):=\Gamma_X(\mathfrak L_{\rm fac}(\rho)).
\]
Its nonemptiness is the condition that the projective factorization determined by $\rho$ comes
from vector bundles $E_i$ and hence from a linear tensor expression, up to a line twist.
\end{definition}

Thus, even when the Brauer class is trivial, two questions remain.
\[
  \operatorname{Red}_{\mathbf d}(P)\neq\varnothing
  \qquad\text{and}\qquad
  \operatorname{Lift}_{\rm fac}(\rho)\neq\varnothing.
\]
Over $\PP^1$ the second obstruction vanishes because $\Br(\PP^1)=0$, which is why the
classification of Theorem~\ref{thm:P1-bipartite-reducibility} reduces to splitting types.

\subsection{Incidence lifts and relative cohomology}

The following construction uses chosen determinant data.  This part is not
an invariant of the Azumaya algebra alone.  Fix
algebraic groups
\[
  L\subset R\subset K,
  \qquad K\longrightarrow \PGL_n,
\]
a $K$ lift $P_K\to X$ of $P$, and a reduction to $R$
\[
  \tau:X\longrightarrow P_K/R.
\]
Let $Q_R\to X$ be the principal $R$ bundle obtained by pulling back the tautological bundle $P_K\to P_K/R$ along $\tau$.
Let
\[
  f_R:X\longrightarrow BR,
  \qquad f_K:X\longrightarrow BK
\]
be the classifying maps of $Q_R$ and $P_K$.  The subgroup $L$ is an incidence subgroup: it is
the local subgroup against which a refined sector reduction is tested.  It is not assumed to be
a global reduction.

\begin{definition}[Incidence lift groupoid]
\label{def:incidence-lift-stack}
The groupoid of $L$ incidence lifts of $\tau$, equivalently of $f_R$, is
\[
  \operatorname{Lift}_L(\tau)
  :=\operatorname{hofib}_{f_R}
  \Bigl(\operatorname{Map}(X,BL)\longrightarrow \operatorname{Map}(X,BR)\Bigr).
\]
Equivalently,
\[
  \operatorname{Lift}_L(\tau)
  \simeq
  \Gamma_X\bigl(X\times_{P_K/R}P_K/L\bigr).
\]
\end{definition}

The incidence lifting problem is again nonabelian: it is the nonemptiness of
$\operatorname{Lift}_L(\tau)$.  Characteristic classes obtained from characters or representations
of $R$ are abelian classes associated with this lifting problem.  A nonzero such class obstructs an
$L$ lift, but the vanishing of all such classes need not construct a lift.

Let $E^\bullet$ be a cohomology theory for quotient stacks which is represented by a complex
of sheaves, or by a functor $R\Gamma(-,E)$ for which mapping cones are defined.  For the map
$Bi:BL\to BR$, define relative cohomology by the mapping cone
\[
  R\Gamma(BL\to BR,E)
  :=\operatorname{Cone}\!\bigl(R\Gamma(BR,E)\to R\Gamma(BL,E)\bigr)[-1],
\]
and set
\[
  \operatorname{Rel}^m_{R,L}(E):=E^m(BL\to BR).
\]
The long exact sequence is
\[
\cdots\to E^{m-1}(BR)\to E^{m-1}(BL)\to E^m(BL\to BR)
\to E^m(BR)\to E^m(BL)\to\cdots.
\]
Thus, in general,
\[
  E^m(BL\to BR)
  \not\cong
  \ker\bigl(E^m(BR)\to E^m(BL)\bigr).
\]
The kernel is the absolute image of the relative group:
\[
  E^m_{\rm abs}(R,L)
  :=\operatorname{im}\bigl(E^m(BL\to BR)\to E^m(BR)\bigr)
  =\ker\bigl(E^m(BR)\to E^m(BL)\bigr).
\]
The equality with the kernel follows from exactness, but it forgets the boundary trivialization
encoded by the full mapping-cone class.

There is also a version over the chosen lifted bundle:
\[
  \operatorname{Rel}^{m,{\rm univ}}_{R,L}(P_K;E)
  :=
  \operatorname{im}\bigl(
    E^m(BL\to BR)\longrightarrow E^m(P_K/L\to P_K/R)
  \bigr).
\]
This is the subgroup of relative classes on $P_K/L\to P_K/R$ obtained from the universal map
$BL\to BR$.
Thus a universal relative class $u\in E^m(BL\to BR)$ gives a full relative class on the incidence
map
\[
  f_{LR}:P_K/L\longrightarrow P_K/R,
  \qquad
  u_P\in E^m(f_{LR}).
\]
The class $u_P$ has an absolute image on $P_K/R$.  Pulling this image back by $\tau$ gives an
absolute class on $X$.  Removing ordinary classes induced from the unreduced $K$ bundle gives
the quotient on the fixed base $X$
\begin{equation}
\label{eq:quotient-on-X}
  \frac{
    f_R^*E^m_{\rm abs}(R,L)
  }{
    f_R^*E^m_{\rm abs}(R,L)
    \cap f_K^*E^m(BK)
  }.
\end{equation}
The quotient \eqref{eq:quotient-on-X} records only the absolute classes on $X$.  The full relative invariant is $u_P\in E^m(f_{LR})$, or its
image in $\operatorname{Rel}^{m,{\rm univ}}_{R,L}(P_K;E)$.

\subsection{Motivic, Chow, and \'{e}tale versions}

All motivic and Chow groups of quotient stacks are taken in the same approximation model fixed
at the beginning of this section.  For motivic cohomology we set
\[
  H^{p,q}_{\rm mot}(BL\to BR)
  :=\mathbb H^p\!\bigl(
       \operatorname{Cone}(R\Gamma(BR,\mathbb Z(q))
       \to R\Gamma(BL,\mathbb Z(q)))[-1]
     \bigr).
\]
The corresponding quotient on the fixed base is
\[
  \frac{
    f_R^*\operatorname{im}\bigl(
      H^{p,q}_{\rm mot}(BL\to BR)\to H^{p,q}_{\rm mot}(BR)
    \bigr)
  }{
    f_R^*\operatorname{im}(\cdots)\cap f_K^*H^{p,q}_{\rm mot}(BK)
  }.
\]
In particular, relative Chow groups are best regarded as the cycle part of relative motivic
cohomology:
\[
  CH^r_{\rm rel}(BL\to BR):=H^{2r,r}_{\rm mot}(BL\to BR).
\]
One should not replace this definition in higher codimension by the kernel of
$CH^r(BR)\to CH^r(BL)$ unless only the absolute class is desired.

For torsion and finite-component effects, fix a prime $\ell\ne \operatorname{char}(k)$ and use
\[
  H^p_{\acute{e}t}(BL\to BR,\mu_{\ell}^{\otimes q})
  :=\mathbb H^p\!\bigl(
      \operatorname{Cone}(R\Gamma_{\acute{e}t}(BR,\mu_{\ell}^{\otimes q})
      \to R\Gamma_{\acute{e}t}(BL,\mu_{\ell}^{\otimes q}))[-1]
    \bigr),
\]
or the $\ell$--adic inverse limit with coefficients $\mathbb Z_\ell(q)$.  This version retains finite normalizer monodromy and torsion descent data.

\subsection{Relative determinant lines and sector formulas}

The basic degree $(2,1)$ class comes from a character.  Let
\[
  \chi:R\longrightarrow \Gm,
  \qquad \chi|_L=1.
\]
The universal associated line bundle
\[
  \mathcal L^{\rm univ}_{\chi}
  :=[\mathbb A^1_{\chi}/R]\longrightarrow [\Spec k/R]=BR
\]
has a tautological trivialization $s_L$ after pullback to $BL$.  Here $i:BL\to BR$, and $\Pic(BL\to BR)$ denotes
the group of pairs $(\mathcal L,s)$, where $\mathcal L\in\Pic(BR)$ and
$s:i^*\mathcal L\simeq\mathcal O_{BL}$ is a specified trivialization.  Hence
\[
  (\mathcal L^{\rm univ}_{\chi},s_L)
  \in \Pic(BL\to BR),
\]
and its relative first Chern class is
\[
  c^{\rm rel}_{1,{\rm alg}}(\chi)
  :=c^{\rm rel}_1(\mathcal L^{\rm univ}_{\chi},s_L)
  \in H^{2,1}_{\rm mot}(BL\to BR).
\]
Pulling back to the chosen reduction gives the ordinary line bundle on $X$
\[
  \mathcal L_{\chi,X}:=Q_R\times^{R} \mathbb A^1_{\chi},
\]
and its absolute class
\[
  c^E_{1,{\rm alg}}(Q_R,L;\chi)
  :=c_1(\mathcal L_{\chi,X})
  = f_R^*\operatorname{abs}\bigl(c^{\rm rel}_{1,{\rm alg}}(\chi)\bigr)
  \in CH^1(X).
\]
If $\tau:X\to P_K/R$ lifts to a section $X\to P_K/L$, the tautological $L$ trivialization pulls back to a trivialization of
$\mathcal L_{\chi,X}$, and this class vanishes.  The converse need not hold.

Higher classes are most naturally defined from relative $K$-classes.  Let
\[
  [V,s_L]\in K^0(BL\to BR).
\]
Its Chern classes and Chern character give
\[
  c_i^{\rm rel}(V,s_L)\in H^{2i,i}_{\rm mot}(BL\to BR),
  \qquad
  \operatorname{ch}^{\rm rel}(V,s_L)\in H^{\rm even,*}_{\rm mot}(BL\to BR)\otimes\mathbb Q.
\]
This relative $K$-theoretic formulation is stable under representations, filtrations, exact sequences, and virtual differences.

For a sector decomposition
\[
  \mathcal E=\bigoplus_{a\in I}\mathcal E_a,
  \qquad n_a=\operatorname{rk}(\mathcal E_a),
\]
set
\[
  R=\prod_{a\in I}\GL_{n_a}.
\]
The corresponding principal $R$ bundle $Q_R$ is the product of the frame bundles of the $\mathcal E_a$.  Here
$I=\prod_{\alpha=1}^r I_\alpha$.  Determinant characters have the form
\[
  \chi_w((g_a)_a)=\prod_{a\in I}\det(g_a)^{w_a},
  \qquad w=(w_a)\in\mathbb Z^I.
\]
For the standard scalar-incidence subgroup, the condition $\chi_w|_L=1$ is the rank weighted
zero marginal condition
\[
  \sum_{a:a_\alpha=b} n_a w_a=0
  \qquad\text{for all }\alpha\text{ and }b\in I_\alpha.
\]
Equivalently,
\[
  \Lambda_E=\ker\!\bigl(M:\mathbb Z^I\to\bigoplus_\alpha\mathbb Z^{I_\alpha}\bigr),
  \qquad
  (Mw)_{\alpha,b}=\sum_{a:a_\alpha=b} n_a w_a.
\]
For $w\in\Lambda_E$ the determinant line and class are
\[
  \mathcal L_w=\bigotimes_{a\in I}(\det \mathcal E_a)^{\otimes w_a},
  \qquad
  c^E_{1,{\rm alg}}(w)=\sum_{a\in I}w_a\,c_1(\det \mathcal E_a).
\]
The superscript $E$ in $c^E_{1,{\rm alg}}$ stands for entangling and is unrelated to the cohomology theory $E^\bullet$.
In the binary line-sector case this gives the primitive two-body vector
\[
  w_{AB}=(1,-1,-1,1)
\]
and
\[
  c^E_{1,{\rm alg}}(w_{AB})
  =c_1(\mathcal E_{++})-c_1(\mathcal E_{+-})-c_1(\mathcal E_{-+})+c_1(\mathcal E_{--}).
\]

\subsection{Finite symmetries and quotient stacks}

Suppose a finite group $\Gamma$ acts by normalizer symmetries on the refined datum.  If the
action on $R$ and $L$ is strict, then
\[
  [BR/\Gamma]\simeq B(R\rtimes\Gamma),
  \qquad
  [BL/\Gamma]\simeq B(L\rtimes\Gamma),
\]
and the relative cohomology group is
\[
  E^m([BL/\Gamma]\to[BR/\Gamma]).
\]
If the symmetry is realized by a non-split normalizer extension, one should replace
$R\rtimes\Gamma$ and $L\rtimes\Gamma$ by the corresponding extension groups $N_R$ and
$N_L$, where $N_L\subset N_R$ is the subgroup preserving $L$, and use the map
$BN_L\to BN_R$.  In applications these groups are often subgroups of a normalizer inside $K$.
In either formulation the correct unlabeled object is not obtained
merely by taking the invariant sublattice $\Lambda_E^\Gamma$.  The quotient-stack or normalizer
formulation retains finite equivariant and torsion data that may disappear on a coarse quotient.

\subsection{Algebraic obstruction data}

The definition gives the following obstruction data.

\begin{definition}[Algebraic obstruction data]
\label{def:algebraic-obstruction-data}
Let $\mathcal A$ be a degree-$n$ Azumaya algebra on $X$, let $P\to X$ be its associated
torsor under $\PGL_n$, and fix a factorization type $\mathbf d$.
\begin{enumerate}
\item The first datum is the stack of reductions to $G_{\mathbf d}$,
\[
  \operatorname{Red}_{\mathbf d}(P)
  =\operatorname{Map}_{/B\PGL_n}(X,BG_{\mathbf d}).
\]
The projective obstruction is the possible emptiness of this stack.
\item Before choosing $P_K$, the passage from projective data to linear data for determinant classes is
measured by the lifting stack
\[
  \mathfrak L_K(P)=X\times_{B\PGL_n}BK.
\]
When a global section has been chosen, we denote the resulting $K$ lift by $P_K$.
\item After choosing $\rho\in\operatorname{Red}_{\mathbf d}(P)$, the factor lifting problem is
\[
  \operatorname{Lift}_{\rm fac}(\rho)=\Gamma_X(\mathfrak L_{\rm fac}(\rho)).
\]
\item For determinant data $L\subset R\subset K$, a $K$ lift $P_K$, and a reduction to $R$
represented by $Q_R$, the incidence problem is
\[
  \operatorname{Lift}_L(\tau)
  =\operatorname{hofib}_{f_R}
  \bigl(\operatorname{Map}(X,BL)\to\operatorname{Map}(X,BR)\bigr).
\]
\item The cohomological part consists of the relative classes
$\operatorname{Rel}^{\bullet,{\rm univ}}_{R,L}(P_K;E)$ and the quotient on $X$ in
\eqref{eq:quotient-on-X}.
\end{enumerate}
This collection is the algebraic obstruction data associated with the chosen $\rho$, $P_K$,
$Q_R$, and cohomology theory $E^\bullet$.
\end{definition}

Here it is understood that the factor lift term is defined only after choosing a labeled reduction
$\rho$.  The lifting stack $\mathfrak L_K(P)$ is a gerbe only in the central-extension cases
considered above.  Once $P_K$ and $Q_R$ have been fixed, the last terms are invariants of the
chosen determinant data and are not invariants of $\mathcal A$ alone.  The relative group
$\operatorname{Rel}^{\bullet,{\rm univ}}_{R,L}(P_K;E)$ retains the mapping-cone information on the incidence map,
whereas \eqref{eq:quotient-on-X} is the absolute part on $X$ after ordinary classes
from $BK$ have been removed.

After analytification, once a connection on the lifted bundle has been chosen, the relative
determinant classes above have the usual differential refinements by line bundles with connection
and boundary trivialization.  We shall not use these refinements in the algebraic arguments of
this paper.

\section{Numerical invariants of Schmidt-rank loci}
\label{sec:numerics}

Throughout this section we are in the bipartite case $\mathbf d=(d_A,d_B)$ with $n=d_A d_B$.
Let
\[
  R_{\le r}\subset \mathbb P(\,k^{d_A}\otimes k^{d_B}\,)
\]
denote the classical determinantal variety of tensors (matrices) of rank $\le r$.

For each geometric point $x\to X$, choose an identification $SB(\mathcal A)_x \simeq \PP^{n-1}$ and let
\[
  H_x \in \Pic\bigl(SB(\mathcal A)_x\bigr)\cong \mathbb{Z}
\]
denote the resulting hyperplane class, namely the class of $\cO_{\PP^{n-1}}(1)$. This class is independent of the chosen identification, since any two such identifications differ by an element of $\PGL_n$ and the $\PGL_n$--action preserves $\cO_{\PP^{n-1}}(1)$.

Throughout Sec.~\ref{sec:numerics}, whenever we speak of the degree or the Hilbert polynomial of a fiber $(\Sigma_{\le r})_x \subset SB(\mathcal A)_x$, it is always taken with respect to the standard fiberwise polarization $\cO_{SB(\mathcal A)_x}(1)$ (equivalently, with respect to $H_x$) and its restriction to $(\Sigma_{\le r})_x$. In the split model $R_{\le r}\subset \PP(V_A\otimes V_B)$, this agrees with the standard polarization coming from $\cO_{\PP(V_A\otimes V_B)}(1)$.

\begin{proposition}[Basic dimension data]
\label{prop:det-dim}
For $1\le r\le \min(d_A,d_B)$ the determinantal variety $R_{\le r}$ has
\[
  \mathrm{codim}_{\mathbb P^{d_A d_B-1}}(R_{\le r})=(d_A-r)(d_B-r),
\]
equivalently
\[
  \dim(R_{\le r}) = r(d_A+d_B-r)-1.
\]
\end{proposition}

\begin{proof}
Standard determinantal geometry: the affine cone is the rank $\le r$ locus in $\Mat_{d_A\times d_B}$, which has codimension $(d_A-r)(d_B-r)$.
Projectivizing decreases dimension by $1$.
\end{proof}

\begin{proposition}[Degree of the determinantal locus]
\label{prop:det-degree}
Assume $d_A\le d_B$ for definiteness. The projective degree of $R_{\le r}$ in $\mathbb P^{d_A d_B-1}$ is
\[
  \deg(R_{\le r})\;=\;
  \prod_{i=0}^{d_A-r-1}
  \frac{(d_B+i)!\, i!}{(r+i)!\, (d_B-r+i)!}.
\]
In particular, for $r=1$ (the product-state locus) one recovers $\deg(\mathbb P^{d_A-1}\times \mathbb P^{d_B-1})
=\binom{d_A+d_B-2}{d_A-1}$.
\end{proposition}

\begin{proof}
Set \(E=k^{d_A}\), \(F=k^{d_B}\), and identify the ambient projective space with
\(\PP(\Hom(E,F))\); this only fixes a duality convention for one tensor factor. Let
\(H=c_1(\cO_{\PP(\Hom(E,F))}(1))\). The tautological line
\(\cO_{\PP}(-1)\subset \Hom(E,F)\otimes \cO_{\PP}\) gives the universal map
\[
  \phi:E\otimes \cO_{\PP}(-1)\longrightarrow F\otimes \cO_{\PP}.
\]
Its rank \(\le r\) degeneracy locus is precisely \(R_{\le r}\). If \(r=d_A\), this locus is the
whole projective space and the displayed product is empty, hence equal to \(1\). We may therefore
assume \(r<d_A\).

Put
\[
  a=d_A-r,
  \qquad
  b=d_B-r,
  \qquad
  m=d_A.
\]
The expected codimension of \(D_r(\phi)\) is \(ab=(d_A-r)(d_B-r)\), which is the actual
codimension by Proposition~\ref{prop:det-dim}. Thom--Porteous therefore gives
\[
  [R_{\le r}]
  =
  \det\!\left(
    c_{\,b+j-i}\bigl((F\otimes\cO_{\PP})-(E\otimes\cO_{\PP}(-1))\bigr)
  \right)_{1\le i,j\le a}.
\]
Since \(F\otimes\cO_{\PP}\) is trivial and
\[
  c(E\otimes\cO_{\PP}(-1))=(1-H)^m,
\]
we have
\[
  c\bigl((F\otimes\cO_{\PP})-(E\otimes\cO_{\PP}(-1))\bigr)
  =(1-H)^{-m}
  =\sum_{q\ge 0}\binom{m+q-1}{q}H^q.
\]
Hence
\[
  [R_{\le r}]
  =
  \det\!\left(\binom{m+b+j-i-1}{b+j-i}\right)_{1\le i,j\le a}H^{ab}.
\]
The determinant is the Jacobi--Trudi determinant for the rectangular Schur polynomial
\(s_{(b^a)}(1^m)\), because \(h_q(1^m)=\binom{m+q-1}{q}\). By the hook--content formula,
\[
  s_{(b^a)}(1^m)
  =
  \prod_{p=1}^{a}\prod_{q=1}^{b}
  \frac{m+q-p}{a+b-p-q+1}.
\]
For each fixed \(p\), the numerator and denominator products are
\[
  \prod_{q=1}^{b}(m+q-p)=\frac{(m+b-p)!}{(m-p)!},
  \qquad
  \prod_{q=1}^{b}(a+b-p-q+1)=\frac{(a+b-p)!}{(a-p)!}.
\]
Thus
\[
  s_{(b^a)}(1^m)
  =
  \prod_{p=1}^{a}
  \frac{(m+b-p)!(a-p)!}{(m-p)!(a+b-p)!}.
\]
Substituting \(m=a+r\) and then writing \(i=a-p\) gives
\[
  s_{(b^a)}(1^m)
  =
  \prod_{i=0}^{a-1}
  \frac{(b+r+i)!\, i!}{(r+i)!\,(b+i)!}
  =
  \prod_{i=0}^{d_A-r-1}
  \frac{(d_B+i)!\, i!}{(r+i)!\,(d_B-r+i)!}.
\]
Finally, \([R_{\le r}]\neq 0\) has codimension \(ab\), and
\[
  \deg(R_{\le r})
  =
  \int_{\PP^{d_Ad_B-1}}[R_{\le r}]\,H^{\dim R_{\le r}}.
\]
Since \(ab+\dim R_{\le r}=d_Ad_B-1\) and \(\int_{\PP^{d_Ad_B-1}}H^{d_Ad_B-1}=1\), the coefficient above is exactly the projective degree. For \(r=1\), the product becomes
\[
  \prod_{i=0}^{d_A-2}\frac{d_B+i}{i+1}
  =\binom{d_A+d_B-2}{d_A-1},
\]
which is the usual degree of the Segre embedding.
\end{proof}

Let $S:=k[z_{ij}\mid 1\le i\le d_A,\ 1\le j\le d_B]$ and let $I_{r+1}$ be the ideal of $(r+1)\times(r+1)$ minors. Then $R_{\le r}=\Proj(S/I_{r+1})$.

\begin{proposition}[Representation-theoretic Hilbert series]
\label{prop:det-hilbert-series}
The graded coordinate ring $S/I_{r+1}$ has the following decomposition as a $\GL_{d_A}\times \GL_{d_B}$--representation:
\[
  (S/I_{r+1})_t
  \;\cong\;
  \bigoplus_{\substack{\lambda\vdash t\\ \ell(\lambda)\le r}}
  \mathbb S_\lambda(k^{d_A})\ \otimes\ \mathbb S_\lambda(k^{d_B}),
\]
where $\lambda$ ranges over partitions of $t$ with at most $r$ parts and $\mathbb S_\lambda$ denotes the Schur functor. Consequently, the Hilbert series is
\[
  \mathrm{Hilb}_{R_{\le r}}(z)
  \;=\;
  \sum_{\lambda:\,\ell(\lambda)\le r}
  \dim\mathbb S_\lambda(k^{d_A})\cdot \dim\mathbb S_\lambda(k^{d_B})\ z^{|\lambda|}.
\]
In particular, the Hilbert polynomial of $R_{\le r}$ is determined explicitly by $(d_A,d_B,r)$.
\end{proposition}

\begin{remark}
For $r=1$ (product-state), one gets a closed formula:
\[
  h^0(\Sigma_{A,B},\mathcal O(t))
  = \binom{t+d_A-1}{d_A-1}\binom{t+d_B-1}{d_B-1},
\]
hence the Hilbert polynomial is the product of two binomials. For general $r$, the decomposition above gives an algorithm (via hook-length formula for $\dim\mathbb S_\lambda$) to compute the Hilbert polynomial.
\end{remark}

\begin{theorem}[Relative constancy of numerics for Schmidt-rank locus]
\label{thm:relative-numerics}
Let $(\mathcal A,P_{\mathbf d})$ be an Azumaya algebra on $X$ with a chosen bipartite type $\mathbf d=(d_A,d_B)$. For each $r$, the Schmidt-rank $\le r$ locus $\Sigma_{\le r}(\mathcal A,\mathbf d)\subset SB(\mathcal A)$ is flat over $X$, and for every $x\in X$ the fiber $(\Sigma_{\le r})_x$ is isomorphic to the classical determinantal variety $R_{\le r}$. Hence:
\begin{itemize}
\item $\dim((\Sigma_{\le r})_x)$ and $\codim((\Sigma_{\le r})_x\subset SB(\mathcal A)_x)$ are constant in $x$ and given by
Proposition~\ref{prop:det-dim};
\item the projective degree of each fiber is constant and given by Proposition~\ref{prop:det-degree};
\item the Hilbert polynomial of each fiber is constant and given by Proposition~\ref{prop:det-hilbert-series}.
\end{itemize}
\end{theorem}

\begin{proof}
By construction, $\Sigma_{\le r}(\mathcal A,\mathbf d)$ is the associated bundle $P_{\mathbf d}\times^{G_{\mathbf d}} R_{\le r}$, hence is locally (fppf) the product $R_{\le r}\times U$. Flatness and fiber identification follow. Numerical invariants are then fiberwise identical to those of $R_{\le r}$.
\end{proof}

\appendix
\section{A spin chain example on a torus}
\label{sec:spin-chain-toy}

We give a spin chain example on a torus. On each simply connected chart the ground state is a product state for a chosen local bipartite splitting. After analytic continuation around one coordinate circle, the transition element leaves $G_{(2,2)}$. Thus the local loci of product states do not glue to a global relative Segre subscheme. Further detailed studies beyond this simple example are given in \cite{Ikeda:2026fen}.

\subsection{Parameter space and the one-magnon Hilbert space}
Let
\[
X := (\mathbb{C}^\times)^2, \qquad (u,v)\in X.
\]
For a Hermitian (physical) family one may restrict to the real two-torus
\[
X_{\mathrm{phys}} := (S^1)^2 \subset (\mathbb{C}^\times)^2,
\qquad u=e^{i\theta_u},\ v=e^{i\theta_v}.
\]
We focus on the case \(p=2\), hence \(m=p^2=4\) and \(\zeta:=e^{2\pi i/m}=i\).

Consider a periodic spin-\(\tfrac12\) chain with four sites \(r\in\mathbb{Z}/4\mathbb{Z}\). Let \(\sigma_r^\pm\) denote the raising/lowering operators at site \(r\), and let
\[
n_r := \frac{1-\sigma_r^z}{2}
\]
be the down-spin (magnon) number operator.

Let \(|\!\uparrow\uparrow\uparrow\uparrow\rangle\) be the fully polarized state, and define the
one-magnon basis
\[
|r\rangle := \sigma_r^-|\!\uparrow\uparrow\uparrow\uparrow\rangle,
\qquad r\in\mathbb{Z}/4\mathbb{Z}.
\]
The one-magnon subspace
\[
\mathcal{H} := \mathrm{span}\{|0\rangle,|1\rangle,|2\rangle,|3\rangle\} \cong \mathbb{C}^4
\]
is preserved by any Hamiltonian conserving total magnon number.

On \(\mathcal{H}\) consider the operators \eqref{eq:gate}
\begin{equation}
\label{eq:Weyl}
{\mathbf X}|r\rangle = |r+1\rangle,\qquad
{\mathbf Z}|r\rangle = \zeta^{\,r}|r\rangle \quad(\zeta=i),
\end{equation}
so that \(\mathbf{ZX}=\zeta \mathbf{XZ}\).
(Physically, \(\mathbf{X}\) is the lattice translation restricted to \(\mathcal{H}\), while \(\mathbf{Z}\) can be implemented by a site-dependent phase gate by controlling a local magnetic field.)

\subsection{A local Hamiltonian and its explicit ground state}

Fix constants \(J>0\) and \(\Delta>J\). On any simply connected chart \(U\subset X^{\mathrm{an}}\) choose an analytic branch \(u^{1/4}\). Define a local, strictly short-range spin Hamiltonian on the full spin Hilbert space by
\begin{equation}
\label{eq:localH-U}
H_U(u,v)
:=
-J\Big(u^{-1/4}\,\sigma_0^+\sigma_1^- + u^{1/4}\,\sigma_0^-\sigma_1^+\Big)
\;+\;\Delta\,(n_2+n_3).
\end{equation}
This Hamiltonian preserves magnon number. In particular it restricts to an operator on \(\mathcal{H}\). In the basis \(\{|0\rangle,|1\rangle,|2\rangle,|3\rangle\}\) one has
\begin{equation}
\label{eq:matrixHU}
H_U\big|_{\mathcal{H}}
=
\begin{pmatrix}
0 & -J u^{1/4} & 0 & 0\\
-J u^{-1/4} & 0 & 0 & 0\\
0 & 0 & \Delta & 0\\
0 & 0 & 0 & \Delta
\end{pmatrix}.
\end{equation}

\begin{proposition}[Local ground state on a chart]
\label{prop:localGS}
Assume \(\Delta>J>0\) and \(|u|=1\) (so \(|u^{1/4}|=1\)). Then \(H_U|_{\mathcal{H}}\) has eigenvalues \(-J,+J,\Delta,\Delta\), and its unique ground state (energy \(-J\)) is
\begin{equation}
\label{eq:GSU}
|\mathrm{GS}\rangle_U=
\frac{1}{\sqrt{2}}\Big(u^{1/4}|0\rangle + |1\rangle\Big).
\end{equation}
Moreover, since the fully polarized state has energy \(0\) and all states with magnons on sites \(2,3\) cost at least \(\Delta>J\), this is also the ground state of the full Hamiltonian \(H_U\).
\end{proposition}

\begin{proof}
The \(2\times 2\) block on \(\mathrm{span}\{|0\rangle,|1\rangle\}\) has eigenvalues \(\pm J\). The remaining basis states \(|2\rangle,|3\rangle\) have energy \(\Delta\). Since \(\Delta>J\), the lowest eigenvalue is \(-J\), and a direct eigenvector computation yields \eqref{eq:GSU}. The final statement follows by comparing energies in different magnon-number sectors.
\end{proof}

\subsection{Gluing by projective conjugation and entangling transport}

The Hamiltonian \eqref{eq:localH-U} is written in a \emph{local coordinate/branch}. To define a globally twisted family (a family defined only up to fiberwise conjugation), we specify its gluing on overlaps.

Let \(T\) denote the unitary translation operator on the spin chain, characterized by \(T\sigma_r^\pm T^{-1}=\sigma_{r+1}^\pm\) (indices mod \(4\)). Its restriction to \(\mathcal{H}\) is exactly the Weyl shift \(X\) in \eqref{eq:Weyl}. Define, on another simply connected chart \(U'\subset X^{\mathrm{an}}\), the Hamiltonian by conjugation:
\begin{equation}
\label{eq:localH-Up}
H_{U'}(u,v) := T^{-1}H_U(u,v)\,T.
\end{equation}
In spin operators this is still strictly local:
\begin{equation}
\label{eq:localH-Up-explicit}
H_{U'}(u,v)=
-J\Big(u^{-1/4}\,\sigma_3^+\sigma_0^- + u^{1/4}\,\sigma_3^-\sigma_0^+\Big)\;+\;\Delta\,(n_1+n_2).
\end{equation}
By construction, on \(U\cap U'\) the two local descriptions are related by a gauge change (conjugation) in the projective group.

\begin{proposition}[Ground-state gluing]
\label{prop:GSglue}
On \(U\cap U'\), the (projective) ground state lines glue via
\[
[\mathrm{GS}]_{U'} = [\mathbf{X}^{-1}\mathrm{GS}]_{U}\ \in\ \mathbb{P}(\mathcal{H}),
\]
hence
\begin{equation}
\label{eq:GSUprime}
|\mathrm{GS}\rangle_{U'}
=
\mathbf{X}^{-1}|\mathrm{GS}\rangle_{U}
=
\frac{1}{\sqrt{2}}\Big(u^{1/4}|3\rangle + |0\rangle\Big).
\end{equation}

\end{proposition}

\begin{proof}
Since \(H_{U'}=\mathbf{X}^{-1}H_U \mathbf{X}\) on \(\mathcal{H}\), eigenlines are transported by \(\mathbf{X}^{-1}\). Applying \(\mathbf{X}^{-1}\) to \eqref{eq:GSU} yields \eqref{eq:GSUprime}.
\end{proof}

\subsection{Local bipartite splitting and emergence of entanglement}

On each simply connected chart we may impose a \emph{local} bipartite identification \(\mathcal{H}\cong \mathbb{C}^2\otimes\mathbb{C}^2\) by writing
\[
r=a+2b,\qquad a,b\in\{0,1\},
\]
and declaring
\[
|a\rangle_A\otimes|b\rangle_B \ \longleftrightarrow\ |r=a+2b\rangle.
\]
Equivalently,
\[
|0\rangle=|00\rangle,\quad |1\rangle=|10\rangle,\quad |2\rangle=|01\rangle,\quad |3\rangle=|11\rangle.
\]

Below, the Schmidt rank is taken with respect to the local identification $\mathcal{H}\cong\mathbb{C}^2\otimes\mathbb{C}^2$ introduced above (not the physical tensor product $(\mathbb{C}^2)^{\otimes 4}$ of the full spin chain).
\begin{proposition}[Product locally, entangled after gluing]
\label{prop:product-vs-entangled}
With respect to the above local splitting \(\mathcal{H}\cong\mathbb{C}^2\otimes\mathbb{C}^2\), the local ground state \(|\mathrm{GS}\rangle_U\) is a product state (Schmidt rank \(1\)), whereas the glued ground state \(|\mathrm{GS}\rangle_{U'}\) is entangled (Schmidt rank \(2\)). More explicitly,
\begin{align}
|\mathrm{GS}\rangle_U
&=
\frac{1}{\sqrt{2}}\Big(u^{1/4}|00\rangle+|10\rangle\Big)
=
\Big(\tfrac{u^{1/4}|0\rangle_A+|1\rangle_A}{\sqrt{2}}\Big)\otimes |0\rangle_B,
\\[3pt]
|\mathrm{GS}\rangle_{U'}
&=
\frac{1}{\sqrt{2}}\Big(|00\rangle+u^{1/4}|11\rangle\Big),
\end{align}
and the latter has Schmidt rank \(2\). In particular, if \(u=1\) then \(|\mathrm{GS}\rangle_{U'}\propto |00\rangle+|11\rangle\) is a Bell state.
\end{proposition}

\begin{proof}
The expression for \(|\mathrm{GS}\rangle_U\) is immediate from \eqref{eq:GSU}. For \(|\mathrm{GS}\rangle_{U'}\), use \eqref{eq:GSUprime} and the dictionary \(|3\rangle=|11\rangle\), \(|0\rangle=|00\rangle\). The coefficient matrix of \(|\mathrm{GS}\rangle_{U'}\) in the product basis is diagonal with two nonzero entries, hence has rank \(2\), so the Schmidt rank is \(2\).
\end{proof}

\subsection{Twisted gluing and failure of a global subsystem structure}

The transition element transporting local trivializations is projectively \(g_v=[\mathbf{X}^{-1}]\in \mathrm{PGL}(\mathcal{H})\). Relative to the local bipartite type \((2,2)\), this element is not local: it does not lie in the stabilizer \(G_{(2,2)}\subset \mathrm{PGL}_4\). Consequently, even though one can choose a local locus of product states on each chart, these loci cannot glue to a global relative Segre subscheme. Therefore the family admits no global subsystem structure of type \((2,2)\).

Finally, the underlying obstruction appears as a scalar discrepancy in \(\mathrm{GL}(\mathcal{H})\): if one also considers the monodromy in the \(u\)-direction, one can take \(g_u=[\mathbf{Z}]\). Then the commutator of lifts is
\[
\mathbf{ZX^{-1}Z^{-1}X} = \zeta^{-1}\,\mathrm{id}_{\mathcal{H}} = i^{-1}\,\mathrm{id}_{\mathcal{H}},
\]
a nontrivial scalar of order \(4\). This is the local manifestation of the Brauer twisting responsible for the failure of a global reduction to \(G_{(2,2)}\).

\begin{remark}
Obstruction classes also appear in familiar settings in physics. For example, in condensed matter theory, \emph{band topology} studies how the occupied-band eigenspaces of a gapped Bloch Hamiltonian over the Brillouin torus $T^d$ assemble into a complex vector bundle (the Bloch bundle). The nontriviality of its $U(r)$-valued gluing data underlies the Berry connection/curvature, and characteristic invariants such as Chern numbers may be viewed as detecting the obstruction to choosing a globally consistent frame (or trivialization) of the occupied subbundle.

By contrast, the obstruction considered in this paper concerns the \emph{global compatibility of subsystem structure}. The family of pure-state spaces is allowed to globalize only projectively, and we ask whether a prescribed tensor-product type $\mathbf d$ can be made globally coherent. This amounts to asking equivalently whether the associated $\PGL_n$-torsor admits a reduction to $G_{\mathbf d}$. When such a reduction fails, transition functions necessarily exit the local operation group, so a state that is locally product can become entangled after gluing. These obstructions are therefore different from Berry or Chern invariants of eigenbundles. They arise from globalizing the subsystem decomposition itself rather than from globalizing eigenstates or eigenbundles.

\end{remark}
\section{Terms and notations}
\begingroup
\renewcommand{\arraystretch}{1.15}
\begin{longtable}{@{}p{0.30\textwidth}p{0.66\textwidth}@{}}
\toprule
Symbol & Meaning \\
\midrule
\endhead
\bottomrule
\endlastfoot

$X$ & Base scheme (classical parameter space).\\
$\mathcal A$ & Azumaya algebra on $X$ (family of $n$-level quantum systems).\\
$n$ & Degree of $\mathcal A$.\\
$P\to X$ & Principal $\PGL_n$-torsor associated to $\mathcal A$.\\
$SB(\mathcal A)\to X$ & Severi--Brauer scheme (pure-state fibration).\\[2pt]

$\mathbf d=(d_1,\dots,d_s)$ & Subsystem type with $n=\prod_{i=1}^s d_i$.\\
$\Sigma_{\mathbf d}$ & Segre variety (product-state locus in the split model).\\
$G_{\mathbf d}\subset \PGL_n$ & Stabilizer of $\Sigma_{\mathbf d}$.\\
$P_{\mathbf d}\to X$ & A subsystem structure: reduction of $P$ to $G_{\mathbf d}$.\\
$\Sigma_{\mathbf d}(\mathcal A)\subset SB(\mathcal A)$ & Descended Segre subfibration (product-state locus in families).\\[2pt]

$\Sigma_{\le r}(\mathcal A,\mathbf d)\subset SB(\mathcal A)$ & When $r>1$, Schmidt-rank $\le r$ locus. In general, $\Sigma_{\le r}$ contains both entangled and product states. When $r=1$, product-state locus.\\
$\Sigma_{= r}:=\Sigma_{\le r}\setminus \Sigma_{\le r-1}$ & When $r>1$, Schmidt-rank $r$ stratum. Every state in $\Sigma_{=r}$ is entangled for $r>1$. When $r=1$, product-state stratum. ($\Sigma_{=0}=\emptyset$) \\
$\mathrm{SR}_{\mathbf d}([\psi])$ & Schmidt-rank entanglement measure defined via the filtration $\Sigma_{\le r}$.\\
$\widetilde{\Sigma}_{\le r}(\mathcal A,\mathbf d)$ & Relative incidence resolution of $\Sigma_{\le r}(\mathcal A,\mathbf d)$.\\
$\rho_r:\widetilde{\Sigma}_{\le r}\to \Sigma_{\le r}$ & Canonical (projective birational) resolution map.\\[2pt]

$\Br(X)$ & Brauer group of $X$.\\
$\beta=[\mathcal A]\in \Br(X)$ & Brauer class of $\mathcal A$.\\
$\Br_{\mathbf d}(X)\subset \Br(X)$ & Subset of classes admitting a subsystem structure of type $\mathbf d$.\\
$\ell=\operatorname{lcm}(d_1,\dots,d_s)$ & Torsion bound appearing in the period constraint for $\Br_{\mathbf d}(X)$.\\[2pt]

$\Hilb^{\Sigma_{\mathbf d}}(SB(\mathcal A)/X)$ & subsystem-structure locus in the relative Hilbert scheme.\\
$P/G_{\mathbf d}(\simeq \Hilb^{\Sigma_{\mathbf d}}(SB(\mathcal A)/X))$ & Quotient scheme representing subsystem structures.\\
$\operatorname{Red}_{\mathbf d}(P)$ & Stack of reductions of $P$ to $G_{\mathbf d}$.\\
$\mathfrak L_K(P)$ & Stack of lifts of $P$ through $BK\to B\PGL_n$.\\
$\operatorname{Lift}_L(\tau)$ & Groupoid of $L$ incidence lifts of the reduction to $R$.\\
$Q_R$ & Principal $R$ bundle obtained from the chosen reduction to $R$.\\
$f_R,f_K$ & Classifying maps of $Q_R$ and $P_K$.\\
$\operatorname{Rel}^{m,{\rm univ}}_{R,L}(P_K;E)$ & Relative classes on $P_K/L\to P_K/R$ obtained from $BL\to BR$.\\[2pt]

\end{longtable}
\endgroup

\bibliographystyle{amsalpha}
\bibliography{ref_v2}

@article{Ikeda:2026yld,
    author = {{Ikeda}, Kazuki and {Oz}, Yaron},
        title = "{Entangling Topological Invariants}",
      journal = {arXiv e-prints},
         year = 2026,
        month = aug,
          eid = {arXiv:2608.03634},
        pages = {arXiv:2608.03634},
          doi = {10.48550/arXiv.2608.03634},
archivePrefix = {arXiv},
       eprint = {2608.03634},
 primaryClass = {quant-ph},
       adsurl = {https://ui.adsabs.harvard.edu/abs/2026arXiv260803634}
}

@article{Ikeda:2026fen,
    author = {{Ikeda}, Kazuki and {Oz}, Yaron},
        title = "{Loop-dependent entangling holonomies in localized topological quartets}",
      journal = {arXiv e-prints},
         year = 2026,
        month = apr,
          eid = {arXiv:2604.11596},
        pages = {arXiv:2604.11596},
          doi = {10.48550/arXiv.2604.11596},
archivePrefix = {arXiv},
       eprint = {2604.11596},
 primaryClass = {cond-mat.mes-hall},
       adsurl = {https://ui.adsabs.harvard.edu/abs/2026arXiv260411596I}
}

@article{Azumaya1951MaximallyCentral,
  title={On Maximally Central Algebras}, 
volume={2}, 
DOI={10.1017/S0027763000010114}, 
journal={Nagoya Mathematical Journal}, 
author={Azumaya, Gorô}, 
year={1951}, 
pages={119–150}
}

@article{AuslanderGoldman1960Brauer,
  author  = {Auslander, Maurice and Goldman, Oscar},
  title   = {The {B}rauer group of a commutative ring},
  journal={Transactions of the American Mathematical Society},
  year={1960},
  volume={97},
  pages={367-409},
  doi     = {10.1090/S0002-9947-1960-0121392-6}
}

@incollection{Grothendieck1968BrauerI,
  author = {Grothendieck, Alexander},
journal = {Séminaire Bourbaki},
language = {fre},
pages = {199-219},
publisher = {Société Mathématique de France},
title = "{Le groupe de Brauer : I. Algèbres d'Azumaya et interprétations diverses}",
url = {http://eudml.org/doc/109691},
volume = {9},
year = {1964-1966},
}

@incollection{Grothendieck1968BrauerII,
  author = {Grothendieck, Alexander},
     title = "{Le groupe de {Brauer} : {II.} {Th\'eories} cohomologiques}",
     booktitle = {S\'eminaire Bourbaki : ann\'ees 1964/65 1965/66, expos\'es 277-312},
     series = {S\'eminaire Bourbaki},
     note = {talk:297},
     pages = {287--307},
     year = {1966},
     publisher = {Soci\'et\'e math\'ematique de France},
     number = {9},
     mrnumber = {1608805},
     language = {fr},
     url = {https://www.numdam.org/item/SB_1964-1966__9__287_0/}
}

@incollection{Grothendieck1968BrauerIII,
  title="{Le groupe de Brauer III: exemples et compl{\'e}ments}",
  author={Grothendieck, Alexander},
  journal={Dix Exposes sur la Cohomologie des Schemas, Masson et Cie},
  year={1968},
  publisher={North-Holland Pub. Comp.}
}

@book{Giraud1971CohomologieNonAbelienne,
  author    = {Giraud, Jean},
  title     = {Cohomologie non ab{\'e}lienne},
  series    = {Grundlehren der mathematischen Wissenschaften},
  volume    = {179},
  publisher = {Springer-Verlag},
  address   = {Berlin, Heidelberg},
  year      = {1971},
  isbn      = {978-3-540-05307-1},
  doi       = {10.1007/978-3-662-62103-5}
}

@article{1983IzMat..21..307M,
  author = {{Merkur'ev}, A.~S. and {Suslin}, A.~A.},
        title = "{K-COHOMOLOGY of Severi-Brauer Varieties and the Norm Residue Homomorphism}",
      journal = {Izvestiya: Mathematics},
         year = 1983,
        month = apr,
       volume = {21},
       number = {2},
        pages = {307-340},
          doi = {10.1070/IM1983v021n02ABEH001793},
       adsurl = {https://ui.adsabs.harvard.edu/abs/1983IzMat..21..307M}
}

@article{westwick1967transformations,
  title={Transformations on tensor spaces},
  author={Westwick, Roy},
  journal={Pacific Journal of Mathematics},
  volume={23},
  number={3},
  pages={613--620},
  year={1967},
  publisher={Mathematical Sciences Publishers}
}

@article{Abramsky:2011sbx,
    author = "Abramsky, Samson and Brandenburger, Adam",
    title = "{The sheaf-theoretic structure of non-locality and contextuality}",
    eprint = "1102.0264",
    archivePrefix = "arXiv",
    primaryClass = "quant-ph",
    doi = "10.1088/1367-2630/13/11/113036",
    journal = "New J. Phys.",
    volume = "13",
    number = "11",
    pages = "113036",
    year = "2011"
}

@ARTICLE{2024arXiv240716767G,
       author = {{Gesmundo}, Fulvio and {In Han}, Young and {Lovitz}, Benjamin},
        title = "{Linear preservers of secant varieties and other varieties of tensors}",
      journal = {arXiv e-prints},
         year = 2024,
        month = jul,
          eid = {arXiv:2407.16767},
        pages = {arXiv:2407.16767},
          doi = {10.48550/arXiv.2407.16767},
archivePrefix = {arXiv},
       eprint = {2407.16767},
 primaryClass = {math.AG},
       adsurl = {https://ui.adsabs.harvard.edu/abs/2024arXiv240716767G}
}

@article{McKinniePrimeToP2007,
  title={Prime to p extensions of the generic abelian crossed product},
  author={McKinnie, Kelly},
  journal={Journal of Algebra},
  volume={317},
  number={2},
  pages={813--832},
  year={2007},
  publisher={Elsevier}
}

@book{Weyman_2003, place={Cambridge}, series={Cambridge Tracts in Mathematics}, title={Cohomology of Vector Bundles and Syzygies}, publisher={Cambridge University Press}, author={Weyman, Jerzy}, year={2003}, collection={Cambridge Tracts in Mathematics}}

@book{BrunsVetter1988,
  author    = {Winfried Bruns and Udo Vetter},
  title     = {Determinantal Rings},
  series    = {Lecture Notes in Mathematics},
  volume    = {1327},
  publisher = {Springer},
  year      = {1988}
}

@article{2018arXiv180402494B,
  title={Brauer group of the moduli spaces of stable vector bundles of fixed determinant over a smooth curve},
  author={Biswas, Indranil and Sengupta, Tathagata},
  journal={Bulletin des Sciences Math{\'e}matiques},
  volume={144},
  pages={55--63},
  year={2018},
  publisher={Elsevier}
}

@article{Ikeda:2025mgj,
    author = {{Ikeda}, Kazuki},
        title = "{Quantum Entanglement as a Cohomological Obstruction}",
 journal = {arXiv e-prints},
         year = 2025,
        month = nov,
          eid = {arXiv:2511.04326},
        pages = {arXiv:2511.04326},
          doi = {10.48550/arXiv.2511.04326},
archivePrefix = {arXiv},
       eprint = {2511.04326},
 primaryClass = {math-ph},
       adsurl = {https://ui.adsabs.harvard.edu/abs/2025arXiv251104326I}
}

@inproceedings{Totaro1999Chow,
  author    = {Totaro, Burt},
  title     = {The Chow ring of a classifying space},
  booktitle = {Algebraic K-theory (Seattle, WA, 1997)},
  series    = {Proceedings of Symposia in Pure Mathematics},
  volume    = {67},
  pages     = {249--281},
  publisher = {American Mathematical Society},
  year      = {1999}
}

@article{EdidinGraham1998EquivariantChow,
  author  = {Edidin, Dan and Graham, William},
  title   = {Equivariant intersection theory},
  journal = {Inventiones Mathematicae},
  volume  = {131},
  number  = {3},
  pages   = {595--634},
  year    = {1998}
}
\end{document}